\documentclass[11pt,a4paper,oneside,english]{amsart}
\usepackage[T1]{fontenc}
\usepackage[latin9]{inputenc}
\usepackage{amsthm}
\usepackage{amsbsy}
\usepackage{amstext}
\usepackage{amssymb}
\usepackage[all]{xy}

\makeatletter

\pdfpageheight\paperheight
\pdfpagewidth\paperwidth

\numberwithin{equation}{section}
\numberwithin{figure}{section}
\theoremstyle{plain}
\swapnumbers
\newtheorem{thm}{\sc Theorem}[section]
  \theoremstyle{plain}
  \newtheorem{rem}[thm]{\sc Remark}
  \theoremstyle{plain}
  \newtheorem{defn}[thm]{\sc Definition}
  \theoremstyle{plain}
  \newtheorem{lem}[thm]{\sc Lemma}
  \theoremstyle{plain}
  \newtheorem{cor}[thm]{\sc Corollary}
  \theoremstyle{plain}
  \newtheorem{prop}[thm]{\sc Proposition}

\advance\hoffset by -0.5 truecm

\usepackage{bbold}
\usepackage{times}

\makeatother

\usepackage{babel}
\begin{document}

\title{On the Rabinowitz Floer homology of twisted cotangent bundles}

\author{Will J. Merry}

\address{Department of Pure Mathematics and Mathematical Statistics, University
of Cambridge, Cambridge CB3 0WB, England\texttt{ }}

\email{\texttt{w.merry@dpmms.cam.ac.uk}}
\begin{abstract}
Let $(M,g)$ be a closed connected orientable Riemannian manifold
of dimension $n\geq2$. Let $\omega:=\omega_{0}+\pi^{*}\sigma$ denote
a twisted symplectic form on $T^{*}M$, where $\sigma\in\Omega^{2}(M)$
is a closed $2$-form and $\omega_{0}$ is the canonical symplectic
structure $dp\wedge dq$ on $T^{*}M$. Suppose that $\sigma$ is weakly
exact and its pullback to the universal cover $\widetilde{M}$ admits
a bounded primitive. Let $H:T^{*}M\rightarrow\mathbb{R}$ be a Hamiltonian
of the form $(q,p)\mapsto\frac{1}{2}\left|p\right|^{2}+U(q)$ for
$U\in C^{\infty}(M,\mathbb{R})$. Let $\Sigma_{k}:=H^{-1}(k)$, and
suppose that $k>c(g,\sigma,U)$, where $c(g,\sigma,U)$ denotes the
Ma\~n\'e critical value. In this paper we compute the Rabinowitz
Floer homology of such hypersurfaces. 

Under the stronger condition that $k>c_{0}(g,\sigma,U)$, where $c_{0}(g,\sigma,U)$
denotes the strict Ma\~n\'e critical value, Abbondandolo and Schwarz
\cite{AbbondandoloSchwarz2009} recently computed the Rabinowitz Floer
homology of such hypersurfaces, by means of a short exact sequence
of chain complexes involving the Rabinowitz Floer chain complex and
the Morse (co)chain complex associated to the free time action functional.
We extend their results to the weaker case $k>c(g,\sigma,U)$, thus
covering cases where $\sigma$ is not exact. 

As a consequence, we deduce that the hypersurface $\Sigma_{k}$ is
never (stably) displaceable for any $k>c(g,\sigma,U)$. This removes
the hypothesis of negative curvature in \cite[Theorem 1.3]{CieliebakFrauenfelderPaternain2010}
and thus answers a conjecture of Cieliebak, Frauenfelder and Paternain
raised in \cite{CieliebakFrauenfelderPaternain2010}. Moreover, following
\cite{AlbersFrauenfelder2010c,AlbersFrauenfelder2008} we prove that
for $k>c(g,\sigma,U)$, any $\psi\in\mbox{Ham}_{c}(T^{*}M,\omega)$
has a leaf-wise intersection point in $\Sigma_{k}$, and that if in
addition $\dim\, H_{*}(\Lambda M;\mathbb{Z}_{2})=\infty$, $\dim\, M\geq2$,
and the metric $g$ is chosen generically, then for a generic $\psi\in\mbox{Ham}_{c}(T^{*}M,\omega)$
there exist infinitely many such leaf-wise intersection points.
\end{abstract}
\maketitle
\tableofcontents{}

\section{Introduction}

Let $(M,g)$ denote a closed connected orientable Riemannian manifold
of dimension $n\geq2$, with cotangent bundle $\pi:T^{*}M\rightarrow M$.
Let $\omega_{0}=d\lambda_{0}$ denote the canonical symplectic form
$dp\wedge dq$ on $T^{*}M$, where $\lambda_{0}$ is the Liouville
$1$-form. Let $\widetilde{M}$ denote the universal cover of $M$.
Let $\sigma\in\Omega^{2}(M)$ denote a closed \textbf{weakly exact}
$2$-form, by this we mean that the pullback $\widetilde{\sigma}\in\Omega^{2}(\widetilde{M})$
is exact. We assume in addition that $\widetilde{\sigma}$ admits
a \textbf{bounded} primitive. This means that there exists $\theta\in\Omega^{1}(\widetilde{M})$
with $d\theta=\widetilde{\sigma}$, and such that \[
\left\Vert \theta\right\Vert _{\infty}:=\sup_{q\in\widetilde{M}}\left|\theta_{q}\right|<\infty,
\]
where $\left|\cdot\right|$ denotes the lift of the metric $g$ to
$\widetilde{M}$. Let \[
\omega:=\omega_{0}+\pi^{*}\sigma
\]
 denote the \textbf{twisted symplectic form}\emph{ }determined by
$\sigma$. We call the symplectic manifold $(T^{*}M,\omega)$ a \textbf{twisted
cotangent bundle}.\newline

Let $H_{g}:T^{*}M\rightarrow\mathbb{R}$ denote the standard {}``kinetic
energy'' Hamiltonian \[
H_{g}(q,p):=\frac{1}{2}\left|p\right|^{2}.
\]
Given a potential $U\in C^{\infty}(M,\mathbb{R})$, we study the autonomous
Hamiltonian system defined by the convex \textbf{mechanical }Hamiltonian
$H:=H_{g}+\pi^{*}U$. Let $X_{H}$ denote the symplectic gradient
of $H$ with respect to the twisted symplectic form $\omega$, and
let $\phi_{t}^{H}:T^{*}M\rightarrow T^{*}M$ denote the flow of $X_{H}$.
The flow $\phi_{t}^{H}$ has a physical interpretation as the flow
of a particle of unit mass and unit charge moving under the effect
of an electric potential and a magnetic field, the former being represented
by $U$ and the latter being represented by $\sigma$ (see for instance
\cite{ArnoldGivental1990,Ginzburg1996}). The \textbf{Lorentz force}\emph{
}$Y:TM\rightarrow TM$ of $\sigma$ is the bundle map determined uniquely
by \begin{equation}
\sigma_{q}(v,w)=\left\langle Y_{q}(v),w\right\rangle \label{eq:lorentz map}
\end{equation}
for $q\in M$ and $v,w\in T_{q}M$.\newline

Given $k\in\mathbb{R}$, we let $\Sigma_{k}:=H^{-1}(k)\subseteq T^{*}M$.
There are two particular {}``critical values'' $c$ and $c_{0}$
of $k$, known as the\emph{ }\textbf{Ma\~n\'e critical values}. They
are such that the dynamics of the hypersurface $\Sigma_{k}$ differ
dramatically depending on the relation of $k$ to these numbers. They
satisfy $c<\infty$ if and only if $\widetilde{\sigma}$ admits a
bounded primitive, and $c_{0}<\infty$ if and only if $\sigma$ is
actually exact. If $\sigma$ is exact then whilst in a lot of cases
one has $c=c_{0}$ (for instance, whenever $\pi_{1}(M)$ is \textbf{amenable}
\cite{FathiMaderna2007}), there may in general be a non-trivial interval
$[c,c_{0}]$. In fact, this latter option happens quite frequently;
see \cite{CieliebakFrauenfelderPaternain2010} for many explicit examples.\newline

Our tool for investigating the hypersurfaces $\Sigma_{k}$ is \textbf{Rabinowitz
Floer homology}, which was introduced by Cieliebak and Frauenfelder
in \cite{CieliebakFrauenfelder2009}, and then extended in various
other directions by several other authors\emph{ }(\cite{AbbondandoloSchwarz2009,AlbersFrauenfelder2010c,CieliebakFrauenfelderPaternain2010,AlbersFrauenfelder2010,CieliebakFrauenfelderOancea2010,AlbersFrauenfelder2010a,AlbersFrauenfelder2008,Kang2010}).
We refer the reader to the survey article \cite{AlbersFrauenfelder2010b}
for a summary of the applications Rabinowitz Floer homology has generated
so far. The present paper should be thought of as a supplement to
\cite{AbbondandoloSchwarz2009}. Indeed, phrased in the language above,
Theorem 2 of \cite{AbbondandoloSchwarz2009} deals with energy levels
$k>c_{0}$ (in which case $\sigma$ is then necessarily exact). In
this paper we study the weaker condition $k>c$. More precisely, we
compute the Rabinowitz Floer homology (as defined in \cite{CieliebakFrauenfelderPaternain2010})
for any energy level $\Sigma_{k}$ with $k>c$. These computations
are then used to answer a conjecture of Cieliebak, Frauenfelder and
Paternain \cite{CieliebakFrauenfelderPaternain2010}; namely that
for $k>c$ the hypersurface $\Sigma_{k}$ is never displaceable.\newline

The starting point of Rabinowitz Floer homology is to work with a
different action functional than the one normally used in Floer homology.
This functional was originally introduced by Rabinowitz \cite{Rabinowitz1978},
and has the advantage that its critical points detect periodic orbits
lying in a \textbf{fixed energy level} of the Hamiltonian. Let $\Lambda T^{*}M$
denote the free loop space of maps $x:S^{1}\rightarrow T^{*}M$ of
Sobolev class $W^{1,2}$. Note that elements of $\Lambda T^{*}M$
are continuous. Given a free homotopy class $\alpha\in[S^{1},M]$,
let $\Lambda_{\alpha}T^{*}M$ denote the component of $\Lambda T^{*}M$
of loops whose projection to $M$ belong to $\alpha$. Fix a potential
$U\in C^{\infty}(M,\mathbb{R})$ and put $H=H_{g}+\pi^{*}U$. Fix
a regular energy value $k\in\mathbb{R}$ of $H$, and set $\Sigma_{k}:=H^{-1}(k)$.
In order to introduce the Rabinowitz action functional, we begin by
considering the $1$-form $a_{H-k}\in\Omega^{1}(\Lambda T^{*}M\times\mathbb{R})$
defined for $(x,\eta)\in\Lambda T^{*}M\times\mathbb{R}$ and $(\xi,b)\in T_{(x,\eta)}(\Lambda T^{*}M\times\mathbb{R})$
by \[
(a_{H-k})_{(x,\eta)}(\xi,b):=\int_{S^{1}}\omega(\xi,\dot{x}-\eta X_{H}(x))dt-b\int_{S^{1}}(H(x(t))-k)dt.
\]
The assumption that $\sigma$ is weakly exact implies the symplectic
form $\omega$ is \textbf{symplectically aspherical}, that is, given
any smooth function $f:S^{2}\rightarrow T^{*}M$ it holds that \[
\int_{S^{2}}f^{*}\omega=0.
\]
This implies that $a_{H-k}$ is exact on $\Lambda_{0}T^{*}M\times\mathbb{R}$,
where $\Lambda_{0}T^{*}M\subseteq\Lambda T^{*}M$ denotes the component
of $\Lambda T^{*}M$ of loops whose projection to $M$ is contractible.
That is, there exists a function $A_{H-k}:\Lambda_{0}T^{*}M\times\mathbb{R}\rightarrow\mathbb{R}$
called the \textbf{Rabinowitz action functional}\emph{ }with the property
that \[
a_{H-k}|_{\Lambda_{0}T^{*}M\times\mathbb{R}}=dA_{H-k}.
\]
The functional $A_{H-k}$ is defined by \[
A_{H-k}(x,\eta):=\int_{D^{2}}\bar{x}^{*}\omega-\eta\int_{S^{1}}(H(x(t))-k)dt,
\]
where $\bar{x}\in C^{0}(D^{2},T^{*}M)\cap W^{1,2}(D^{2},T^{*}M)$
is any map such that $\bar{x}|_{\partial D^{2}}=x$. The symplectic
asphericity condition implies that the value of $\int_{D^{2}}\bar{x}^{*}\omega$
is independent of the choice of filling disc $\bar{x}$. Our first
observation is that the additional assumption that the lift $\widetilde{\sigma}$
of $\sigma$ to $\widetilde{M}$ admits a \textbf{bounded}\emph{ }primitive
implies that the symplectic form $\omega$ is \textbf{symplectically
atoroidal}, that is, given any smooth function $f:\mathbb{T}^{2}\rightarrow T^{*}M$
it holds that \[
\int_{\mathbb{T}^{2}}f^{*}\omega=0
\]
(see Lemma \ref{lem:key observation}). In this case $a_{H-k}$ is
actually exact on all of $\Lambda T^{*}M\times\mathbb{R}$. Indeed,
for each $\alpha\in[S^{1},M]$, fix a reference loop $x_{\alpha}\in\Lambda_{\alpha}T^{*}M$.
Let $C:=S^{1}\times[0,1]$. Let $\bar{x}\in C^{0}(C,T^{*}M)\cap W^{1,2}(C,T^{*}M)$
denote any map such that $\bar{x}(\cdot,0)=x$ and $\bar{x}(\cdot,1)=x_{\alpha}$.
Since $\omega$ is symplectically atoroidal, the value of $\int_{C}\bar{x}^{*}\omega$
is independent of the choice of $\bar{x}$. Thus we may define $A_{H-k}:\Lambda T^{*}M\times\mathbb{R}\rightarrow\mathbb{R}$
by \[
A_{H-k}(x,\eta):=\int_{C}\bar{x}^{*}\omega-\eta\int_{S^{1}}(H(x(t))-k)dt,
\]
so that \[
a_{H-k}=dA_{H-k}.
\]
The critical points of $A_{H-k}$ are easily seen to satisfy:\[
\dot{x}=\eta X_{H}(x(t))\ \ \ \mbox{for all }t\in S^{1};
\]
\[
\int_{S^{1}}(H(x(t))-k)dt=0.
\]
Since $H$ is invariant under its Hamiltonian flow, the second equation
implies \[
H(x(t))-k=0\ \ \ \mbox{for all }t\in S^{1},
\]
that is, \[
x(S^{1})\subseteq\Sigma_{k}.
\]
Thus if $\mbox{Crit}(A_{H-k})$ denotes the set of critical points
of $A_{H-k}$, we can characterize $\mbox{Crit}(A_{H-k})$ by\begin{eqnarray*}
\mbox{Crit}(A_{H-k}) & = & \left\{ (x,\eta)\in\Lambda T^{*}M\times\mathbb{R}\,:\, x\in C^{\infty}(S^{1},T^{*}M)\right.\\
 &  & \left.\ \ \ \dot{x}(t)=\eta X_{H}^{\sigma}(x(t)),\ x(S^{1})\subseteq\Sigma_{k}\right\} .
\end{eqnarray*}
For a generic choice of the metric $g$, the set $\mbox{Crit}(A_{H-k})$
consists of a copy of the hypersurface $\Sigma_{k}$ (corresponding
to the constant loops with $\eta=0$) and a discrete union of circles.\newline 

On the Lagrangian side we can play a similar game. Let $L_{g}:TM\rightarrow\mathbb{R}$
denote the standard {}``kinetic energy'' Lagrangian defined by $L_{g}(q,v):=\frac{1}{2}\left|v\right|^{2}$,
and given $U\in C^{\infty}(M,\mathbb{R})$ consider the Lagrangian
$L:=L_{g}-\pi^{*}U$ (here we denote also by $\pi$ the footpoint
map $TM\rightarrow M$). The Lagrangian $L$ is the \textbf{Fenchel
transform }of the Hamiltonian $H=H_{g}+\pi^{*}U$ from above. Let
$q_{\alpha}:=\pi\circ x_{\alpha}$, so that $q_{\alpha}$ is an element
of the component $\Lambda_{\alpha}M$ corresponding to $\alpha$ of
the free loop space $\Lambda M$. Given any $q\in\Lambda_{\alpha}M$,
let $\bar{q}\in C^{0}(C,M)\cap W^{1,2}(C,M)$ denote any map such
that $\bar{q}(\cdot,0)=q$ and $\bar{q}(\cdot,1)=q_{\alpha}$ (where
$C=S^{1}\times[0,1]$ is as above). Then we define the \textbf{free
time action functional}\emph{ }$S_{L+k}:\Lambda M\times\mathbb{R}^{+}\rightarrow\mathbb{R}$
by \[
S_{L+k}(q,T):=T\int_{S^{1}}\left(L\left(q(t),\frac{\dot{q}(t)}{T}\right)+k\right)dt+\int_{C}\bar{q}^{*}\sigma.
\]
If $\sigma$ is exact, this reduces to the definition of the standard
free time action functional studied in \cite{ContrerasIturriagaPaternainPaternain2000,Contreras2006}
(up to a constant).

If $\mbox{Crit}(S_{L+k})$ denotes the set of critical points of $S_{L+k}$,
then if $g$ is chosen genericaly the set $\mbox{Crit}(S_{L+k})$
consists of a discrete union of circles. If $L=L_{g}-\pi^{*}U$ and
$H=H_{g}+\pi^{*}U$ then there is a close relationship between critical
points of $S_{L+k}$ and critical points of $A_{H-k}$. Namely, each
critical point $w=(q,T)\in\mbox{Crit}(S_{L+k})$ determines two critical
points $Z^{\pm}(w)=(x^{\pm},\pm T)$ of $A_{H-k}$. Here $x^{+}(t):=(q(t),\dot{q}(t))$
(where we have identified $TM$ with $T^{*}M$ via the Riemannian
metric to see $\dot{q}(t)$ as an element of $T_{q(t)}^{*}M$) and
$x^{-}(t):=x^{+}(-t)$. Then we have \[
\{Z^{\pm}(w)\,:\, w\in\mbox{Crit}(S_{L+k})\}=\left\{ (x,\eta)\in\mbox{Crit}(A_{H-k})\,:\,\eta\ne0\right\} .
\]
The {}``extra'' critical points $(x,0)$ of $A_{H-k}$ correspond
to the so-called \textbf{critical points at infinity} of $S_{L+k}$,
in the sense of Bahri \cite{Bahri1989}. Following \cite{AbbondandoloSchwarz2009},
this motivates us to extend $\mbox{Crit}(S_{L+k})$ to a new set \[
\overline{\mbox{Crit}}(S_{L+k}):=\mbox{Crit}(S_{L+k})\cup\{(q,0)\,:q\in M\}.
\]
For $k>c$, it turns out that one can do Morse theory with $S_{L+k}$.
More precisely, after picking a Morse function $f:\overline{\mbox{Crit}}(S_{L+k})\rightarrow\mathbb{R}$,
one can combine Frauenfelder's \textbf{Morse-Bott homology with cascades}\emph{
}\cite[Appendix A]{Frauenfelder2004} with Abbondandolo and Majer's
infinite dimensional Morse theory \cite{AbbondandoloMajer2006} to
construct a chain complex $CM_{*}(S_{L+k},f)$ and a cochain complex
$CM^{*}(S_{L+k},f)$ whose associated \textbf{Morse (co)homology}\emph{
}$HM_{*}(S_{L+k},f)$ and $HM^{*}(S_{L+k},f)$ coincide with the singular
(co)homology of $\Lambda M\times\mathbb{R}^{+}$.\newline 

The fact that there is such a strong relation between the critical
points of $S_{L+k}$ and $A_{H-k}$ means that one is tempted to try
and relate the Morse (co)homology of $S_{L+k}$ with the Rabinowitz
Floer homology of $A_{H-k}$. This is precisely what Abbondandolo
and Schwarz did, and in \cite[Theorem 2]{AbbondandoloSchwarz2009}
they construct (for $k>c_{0}$) a short exact sequence of chain complexes\begin{equation}
0\rightarrow CM_{*}(S_{L+k},f)\rightarrow RF_{*}(A_{H-k},h)\rightarrow CM^{1-*}(S_{L+k},-f)\rightarrow0.\label{eq:ses}
\end{equation}
Here $h:\mbox{Crit}(A_{H-k})\rightarrow\mathbb{R}$ denotes a Morse
function on $\mbox{Crit}(A_{H-k})$ and $RF_{*}(A_{H-k},h)$ denotes
the Rabinowitz Floer chain complex of the pair $(A_{H-k},h)$. We
remark here that the Morse functions $f$ and $h$ must be related
to each other in a fairly special way in order for such a short exact
sequence to hold. Anyway, passing to the long exact sequence associated
to this short exact of chain complexes and making the identification
of the Morse (co)homology with the singular (co)homology of the loop
space, this provides a way of computing the Rabinowitz Floer homology
$RFH_{*}(A_{H-k})$. Actually it must be said that this long exact
sequence is a special case of a more general construction of Cieliebak,
Frauenfelder and Oancea \cite{CieliebakFrauenfelderOancea2010}, which
links Rabinowitz Floer homology with symplectic homology.\newline 

The aim of this paper is to show how the sequence \eqref{eq:ses}
can be extended to the weaker case of $k>c$. In order to keep our
exposition from being unnecessarily long, we only provide full details
where there are substantial differences from \cite{AbbondandoloSchwarz2009}.
Let us now summarize exactly what we do differently. On the Lagrangian
side, more work must be done in order to define the Morse (co)complex;
the key problem is to show that the Palais-Smale condition holds,
which was shown in our previous work \cite{Merry2010}. On the Hamiltonian
side, we work directly with the Hamiltonians $H_{g}+\pi^{*}U$ that
define the energy level $\Sigma_{k}$. This means that we cannot use
the $L^{\infty}$ estimates on gradient flow lines of $A_{H-k}$ previously
obtained in \cite{CieliebakFrauenfelder2009,AbbondandoloSchwarz2009,CieliebakFrauenfelderPaternain2010,CieliebakFrauenfelderOancea2010}.
Instead, we adapt the method of Abbondandolo and Schwarz in \cite{AbbondandoloSchwarz2006}
to obtain our $L^{\infty}$ bounds. In fact, we are only able to obtain
these $L^{\infty}$ bounds if we make an \textbf{additional }assumption
on $\sigma$, namely that $\left\Vert \sigma\right\Vert _{\infty}$
is sufficiently small (cf. Remark \ref{rem:shrinking sigma}; specifically
\eqref{eq:epsilon 2}) . However, a scaling argument, combined with
invariance of the Rabinowitz Floer homology defined in \cite{CieliebakFrauenfelderPaternain2010}
(see below) implies this is in fact \textbf{no} extra restriction
at all.\newline 

A further difference is the question of grading; since we are working
with the twisted symplectic form $\omega$, results such as Duistermaat's
\textbf{Morse index theorem} \cite{Duistermaat1976} are not immediately
available to us. Secondly the Hamiltonian $H$ is no longer a \textbf{defining
Hamiltonian }(in the sense of \cite{CieliebakFrauenfelder2009}).
This makes the computation of the Fredholm index of the operator obtained
by linearizing the gradient of the Rabinowitz action functional along
a flow line somewhat more complicated. Moreover unlike the corresponding
situation in \cite{AbbondandoloSchwarz2009}, the relationship between
the Morse index of the fixed period action functional and the free
time action functional is not so clear (cf. Theorem \ref{thm:lag index}
and Remark \ref{rem:Index jumping on the morse side}). Full details
of these index computations can be found in a supplementary paper
joint with Gabriel P. Paternain \cite{MerryPaternain2010}.\newline

Anyway, having proved such a short exact sequence \eqref{eq:ses},
it is then clear that the Rabinowitz Floer homology $RFH_{*}(A_{H-k})$
is non-zero whenever $k>c$. A key property of the Rabinowitz Floer
homology $RFH_{*}(\Sigma,V)$ constructed in \cite{CieliebakFrauenfelderPaternain2010},
which is associated to a hypersurface $\Sigma$ of \textbf{virtual
restricted contact type} in a \textbf{geometrically bounded symplectically
aspherical} symplectic manifold $V$, is that if the hypersurface
is displaceable then $RFH_{*}(\Sigma,V)$ vanishes. Assuming that
our Rabinowitz Floer homology $RFH_{*}(A_{H-k})$ is the same as the
Rabinowitz Floer homology%
\footnote{The hypersurface $\Sigma_{k}$ is virtually contact if $k>c$ \cite[Lemma 5.1]{CieliebakFrauenfelderPaternain2010},
so $RFH_{*}(\Sigma_{k},T^{*}M)$ as defined in \cite{CieliebakFrauenfelderPaternain2010}
is well defined.%
} $RFH_{*}(\Sigma_{k},T^{*}M)$ from \cite{CieliebakFrauenfelderPaternain2010},
this would imply that $\Sigma_{k}$ is never displaceable for $k>c$.
In Section \ref{sec:Non-displaceability-above-the} we prove that
the two Rabinowitz Floer homologies are indeed isomorphic, and thus
we arrive at the main result of this paper.
\begin{thm}
\label{thm:my main theorem}Let $(M,g)$ be a closed connected orientable
Riemannian manifold and $\sigma\in\Omega^{2}(M)$ be a closed weakly
exact $2$-form. Let $U\in C^{\infty}(M,\mathbb{R})$ and put $H:=H_{g}+\pi^{*}U$
and $\Sigma_{k}:=H^{-1}(k)$. Then if $k>c(g,\sigma,U)$ the Rabinowitz
Floer homology $RFH_{*}(\Sigma_{k},T^{*}M)$ of \cite{CieliebakFrauenfelderPaternain2010}
is defined and non-zero. In particular, $\Sigma_{k}$ is not displaceable.\end{thm}
\begin{rem}
An alternative proof of Theorem \ref{thm:my main theorem} is given
by Bae and Frauenfelder in \cite{BaeFrauenfelder2010}. Their idea
is to show directly that the Rabinowitz Floer homology $RFH_{*}(\Sigma_{k},T^{*}M;\omega)$
as defined in \cite{CieliebakFrauenfelderPaternain2010} (where we
temporarily add {}``$\omega$'' to the notation to indicate which
symplectic form we are working with) is independent under certain
perturbations of $\omega$. Using this, they prove that $RFH_{*}(\Sigma_{k},T^{*}M;\omega)\cong RFH_{*}(\Sigma_{k},T^{*}M;\omega_{0})$,
from which they can deduce Theorem \ref{thm:my main theorem} from
the corresponding results in \cite{CieliebakFrauenfelderOancea2010,AbbondandoloSchwarz2009}.
See also Remark \ref{rem:inv under g} below.
\end{rem}

\begin{rem}
\label{rem:stably displaceable remark}In fact, Theorem \ref{thm:my main theorem}
proves that for $k>c$ the hypersurface $\Sigma_{k}$ is never \textbf{\emph{stably
displaceable}}. The concept of being stably displaceable is useful
when the Euler characteristic $\chi(M)$ is non-zero. Indeed, when
$\chi(M)\ne0$, $\Sigma_{k}$ is never displaceable for topological
reasons. However, it may be stably displaceable. To define stably
displaceability, one considers the symplectic manifold $(T^{*}M\times T^{*}S^{1},\omega\oplus\omega_{S^{1}})$,
where $\omega_{S^{1}}$ is the standard symplectic form on $T^{*}S^{1}$
(note that $\chi(M\times S^{1})=0$). If $H=H_{g}+\pi^{*}U$ is a
mechanical Hamiltonian on $T^{*}M$, consider the new Hamiltonian
$\widehat{H}:T^{*}(M\times S^{1})\rightarrow\mathbb{R}$ defined by
\begin{align*}
\widehat{H}(q,p,t,p_{t}): & =H(q,p)+\frac{1}{2}\left|p_{t}\right|^{2}\ \ \ \ \ \ \ \ \ \ \ \ p\in T_{q}^{*}M,\ p_{t}\in T_{t}^{*}S^{1}\\
 & =\frac{1}{2}\left|p\right|^{2}+U(q)+\frac{1}{2}\left|p_{t}\right|^{2}.
\end{align*}
Let $\widehat{\Sigma}_{k}:=\widehat{H}^{-1}(k)$. Then by definition
$\Sigma_{k}$ is stably displaceable if $\widehat{\Sigma}_{k}$ is
displaceable. In order to see why our theorem implies that $\Sigma_{k}$
is never stably displaceable for $k>c$, one uses the following observation
of Macarini and Paternain \cite[Lemma 2.2]{MacariniPaternain2010}:
if $c$ denotes the Ma\~n\'e critical value of $H$ and $\widehat{c}$
denotes the Ma\~n\'e critical value of $\widehat{H}$ then%
\footnote{Actually \cite[Lemma 2.2]{MacariniPaternain2010} works with the strict
Ma\~n\'e critical values $c_{0}$ and $\widehat{c}_{0}$, but exactly
the same proof (working on $\tilde{M}$ instead of $M$) shows that
$c=\widehat{c}$.%
} $\widehat{c}=c$. Thus if $k>c$ then also $k>\widehat{c}$, and
so applying Theorem \ref{thm:my main theorem} to $\widehat{\Sigma}_{k}$
we see that $\widehat{\Sigma}_{k}$ is not displaceable, and hence
$\Sigma_{k}$ is not stably displaceable.
\end{rem}

\begin{rem}
\label{rem:remark on cfp rfh}Strictly speaking, the Rabinowitz Floer
homology $RFH_{*}(\Sigma_{k},T^{*}M)$ as defined in \cite{CieliebakFrauenfelderPaternain2010}
is only defined for contractible loops, as the observation that the
twisted symplectic form $\omega$ is symplectically atoroidal was
not used in that paper. However, if one uses this observation, the
construction in \cite{CieliebakFrauenfelderPaternain2010} allows
one to define $RFH_{*}(\Sigma_{k},T^{*}M)$ for any free homotopy
class of loops. The proof given in Section \ref{sec:Non-displaceability-above-the}
shows that our $RFH_{*}(A_{H-k})$ agrees with this Rabinowitz Floer
homology $RFH_{*}(\Sigma_{k},T^{*}M)$ (in any free homotopy class).
The reader however may prefer to read Section \ref{sec:Non-displaceability-above-the}
as if we were only working with contractible loops (which is sufficient
for the non-displaceability application we have in mind). 
\end{rem}

\begin{rem}
In \cite{Merry2010b} we compute the \textbf{\emph{Lagrangian Rabinowitz
Floer homology }}of the hypersurface $\Sigma_{k}$, where for the
Lagrangian submanifolds of $T^{*}M$ involved we take two cotangent
fibres $T_{q_{0}}^{*}M$ and $T_{q_{1}}^{*}M$ (where possibly $q_{0}=q_{1}$).
We show that a similar short exact sequence to \eqref{eq:ses} exists
between the Lagrangian Rabinowitz Floer homology \[
RFH_{*}(\Sigma_{k},T_{q_{0}}^{*}M,T_{q_{1}}^{*}M,T^{*}M)
\]
 and the Morse (co)homology of the free time action functional, this
time defined on the \textbf{\emph{path space }}$\Omega(M,q_{0},q_{1})$
of paths in $M$ from $q_{0}$ to $q_{1}$. 
\end{rem}
Having proved that for $k>c$ the Rabinowitz Floer homology $RFH_{*}(\Sigma_{k},T^{*}M)$
is non-zero, one can prove a much stronger statement than non-displaceability,
which we will now explain. Let $\mbox{Ham}_{c}(T^{*}M,\omega)$ denote
the set of compactly supported Hamiltonian diffeomorphisms of the
symplectic manifold $(T^{*}M,\omega)$, that is \[
\mbox{Ham}_{c}(T^{*}M,\omega):=\left\{ \phi_{1}^{F}\,:\, F\in C_{c}^{\infty}(S^{1}\times T^{*}M,\mathbb{R})\right\} ,
\]
where $\phi_{t}^{F}$ is the flow of $X_{F}$; the latter being the
time-dependent symplectic gradient of $F$ with respect to $\omega$. 

Fix $H=H_{g}+\pi^{*}U$ and put $\Sigma_{k}:=H^{-1}(k)$. Given $x\in\Sigma_{k}$,
let us write $\mathcal{L}_{x}$ for the leaf of the characteristic
foliation of $\Sigma_{k}$ passing through $x$, that is,\[
\mathcal{L}_{x}:=\{\phi_{t}^{H}(x)\,:\, t\in\mathbb{R}\},
\]
so that $\Sigma_{k}$ is foliated by the leaves $\{\mathcal{L}_{x}\,:\, x\in\Sigma_{k}\}$.
Given $\psi\in\mbox{Ham}_{c}(T^{*}M,\omega)$, a point $x\in\Sigma_{k}$
is called a \textbf{leaf-wise intersection point for $\psi$}\emph{
}if $\psi(x)\in\mathcal{L}_{x}$. By following through the proofs
in \cite{AlbersFrauenfelder2010c,AlbersFrauenfelder2008} we can prove
the following result. 
\begin{thm}
\label{thm:Leafwise}Let $(M,g)$ be a closed connected orientable
Riemannian manifold and $\sigma\in\Omega^{2}(M)$ be a closed weakly
exact $2$-form. Let $U\in C^{\infty}(M,\mathbb{R})$ and put $H:=H_{g}+\pi^{*}U$.
Choose $k>c(g,\sigma,U)$ and put $\Sigma_{k}:=H^{-1}(k)$. Then for
any $\psi\in\mbox{\emph{Ham}}_{c}(T^{*}M,\omega)$ there exists a
leaf-wise intersection point for $\psi$ in $\Sigma_{k}$. Moreover,
if $\dim\, H_{*}(\Lambda M;\mathbb{Z}_{2})=\infty$ and $g$ is chosen
generically, then for a generic $\psi\in\mbox{\emph{Ham}}_{c}(T^{*}M,\omega)$
there exist infinitely many leaf-wise intersection points for $\psi$
in $\Sigma_{k}$.
\end{thm}
We conclude this introduction with a remark about how the results
of this paper extend to more general Hamiltonian systems.
\begin{rem}
In fact, all of the results in the present paper are valid under more
general hypotheses, as we now explain. Recall that an autonomous Hamiltonian
$K\in C^{\infty}(T^{*}M,\mathbb{R})$ is called \textbf{\emph{Tonelli}}
if $K$ is \textbf{\emph{fibrewise strictly convex}}\textbf{ }and
\textbf{\emph{superlinear}}. In other words, the second differential
$d^{2}(K|_{T_{q}^{*}M})$ of $K$ restricted to each tangent space
$T_{q}^{*}M$ is positive definite, and \[
\lim_{\left|p\right|\rightarrow\infty}\frac{K(q,p)}{\left|p\right|}=\infty
\]
uniformly for $q\in M$. As with mechanical Hamiltonians, given a
Tonelli Hamiltonian $K$ and a weakly exact $2$-form $\sigma$, there
exists a critical value $c(K,\sigma)$ called the \textbf{\emph{Ma\~n\'e
critical value}}\emph{. }As before, $c(K,\sigma)<\infty$ if and only
if $\widetilde{\sigma}$ admits a bounded primitive. Let us say that
a closed connected orientable hypersurface $\Sigma\subseteq T^{*}M$
is a \textbf{\emph{Ma\~n\'e supercritical hypersurface}}\textbf{
}if there exists a Tonelli Hamiltonian $K$ such that $\Sigma=K^{-1}(k)$
for some $k>c(K,\sigma)$.

Both Theorem \ref{thm:my main theorem} and Theorem \ref{thm:Leafwise}
extend to Ma\~n\'e supercritical hypersurfaces. Namely: the Rabinowitz
Floer homology of any Ma\~n\'e supercritical hypersurface is defined
and non-zero. In particular, Ma\~n\'e supercritical hypersurfaces
are never displaceable. Secondly, given any Ma\~n\'e supercritical
hypersurface $\Sigma$ and any $\psi\in\mbox{\emph{Ham}}_{c}(T^{*}M,\omega)$
there exists a leaf-wise intersection point for $\psi$ in $\Sigma$.
Moreover, if $\dim\, H_{*}(\Lambda M;\mathbb{Z}_{2})=\infty$ and
$\Sigma$ is non-degenerate (which holds generically), then for a
generic $\psi\in\mbox{\emph{Ham}}_{c}(T^{*}M,\omega)$ there exist
infinitely many leaf-wise intersection points for $\psi$ in $\Sigma$. 

More details about these results can be found in \cite{Merry2011}.
\end{rem}
\emph{Acknowledgments. }I would like to thank my Ph.D. adviser Gabriel
P. Paternain for many helpful discussions. I am also extremely grateful
to Alberto Abbondandolo, Peter Albers and Urs Frauenfelder for several
stimulating remarks and insightful suggestions, and for pointing out
errors in previous drafts of this work.

\section{Preliminaries}

We denote by $\overline{\mathbb{R}}$ the extended real line $\overline{\mathbb{R}}:=\mathbb{R}\cup\{\pm\infty\}$,
with the differentiable structure induced by the bijection $[-\pi/2,\pi/2]\rightarrow\overline{\mathbb{R}}$
given by \[
s\mapsto\begin{cases}
\tan s & s\in(-\pi/2,\pi/2)\\
\pm\infty & s=\pm\pi/2.
\end{cases}
\]
We denote by $\mathbb{R}^{+},\mathbb{R}_{0}^{+}$ the spaces $(0,\infty)$
and $[0,\infty)$, with similar conventions for $\mathbb{R}^{-},\mathbb{R}_{0}^{-}$.
We will often identify $S^{1}$ with $\mathbb{R}/\mathbb{Z}$. We
adopt throughout the convenient convention that any manifold asserted
to have negative dimension is in fact, empty. Another convention we
use throughout is: given a function $f(s,t)$ of two variables $s,t$
(usually $(s,t)\in\mathbb{R}\times\mathbb{T}$) we let $f':=\partial_{s}f$
and $\dot{f}:=\partial_{t}f$. Throughout the paper we will freely
and ambiguously use the isometry $TM\cong T^{*}M,\, v\mapsto\left\langle v,\cdot\right\rangle $,
determined by the Riemannian metric $g$, to identify points in $T_{q}M$
with points in $T_{q}^{*}M$.

\textbf{All the sign conventions used in this paper match those of
\cite{AbbondandoloSchwarz2009}}.

\subsection{The loop spaces}

$\ $\vspace{6 pt}

Let $W^{1,2}([0,1],M)$ denote the Hilbert manifold of paths $q:[0,1]\rightarrow M$
of Sobolev class $W^{1,2}$. Note that elements of $W^{1,2}([0,1],M)$
are continuous. Let $\Lambda M$ denote the submanifold consisting
of loops $q:S^{1}\rightarrow M$ of Sobolev class $W^{1,2}$. Note
that $\Lambda M$ is homotopy equivalent to both $C^{0}(S^{1},M)$
and $C^{\infty}(S^{1},M)$. We can identify $T_{q}\Lambda M$ with
$W^{1,2}(S^{1},q^{*}TM)$, that is, the sections $\zeta:S^{1}\rightarrow q^{*}TM$
of class $W^{1,2}$. Given a free homotopy class $\alpha\in[S^{1},M]$,
let $\Lambda_{\alpha}M\subseteq\Lambda M$ denote the connected component
of $\Lambda M$ consisting of the loops $q\in\Lambda M$ belonging
to the free homotopy class $\alpha$. Given $\alpha\in[S^{1},M]$,
we write $-\alpha$ for the free homotopy class that contains the
loops $q^{-}(t):=q(-t)$ for $q\in\Lambda_{\alpha}M$. 

Similarly we let $W^{1,2}([0,1],T^{*}M)$ denote the Hilbert manifold
of paths $x:[0,1]\rightarrow T^{*}M$ of Sobolev class $W^{1,2}$.
Note that elements of $W^{1,2}([0,1],T^{*}M)$ are continuous. Denote
by $\Lambda T^{*}M$ the submanifold of loops $x:S^{1}\rightarrow T^{*}M$
of Sobolev class $W^{1,2}$. Note that $\Lambda T^{*}M$ is homotopy
equivalent to both $C^{0}(S^{1},T^{*}M)$ and $C^{\infty}(S^{1},T^{*}M)$.
The tangent space $T_{x}\Lambda T^{*}M$ can be identified with $W^{1,2}(S^{1},x^{*}T^{*}M)$,
that is, the sections $\xi:S^{1}\rightarrow x^{*}TT^{*}M$ of class
$W^{1,2}$. Given $\alpha\in[S^{1},M]$, we let $\Lambda_{\alpha}TM$
denote the set of loops $x\in\Lambda T^{*}M$ whose projection $\pi\circ x$
lies in $\Lambda_{\alpha}M$.\newline 

Using the metric $g=\left\langle \cdot,\cdot\right\rangle $ on $M$
we obtain a metric $\left\langle \left\langle \cdot,\cdot\right\rangle \right\rangle _{g}$$ $on
$\Lambda M\times\mathbb{R}^{+}$ via \begin{equation}
\left\langle \left\langle (\zeta,b),(\vartheta,e)\right\rangle \right\rangle _{g}:=\int_{S^{1}}\left\{ \left\langle \zeta,\vartheta\right\rangle +\left\langle \nabla_{t}\zeta,\nabla_{t}\vartheta\right\rangle \right\} dt+be,\label{eq:the metric on the loop space}
\end{equation}
where $\nabla$ denotes the Levi-Civita connection of $(M,g)$.

Let $\mathcal{J}$ denote the space of $1$-periodic almost complex
structures on $T^{*}M$ with finite $L^{\infty}$ norm, and equip
$\mathcal{J}$ with the $L^{\infty}$ norm. The metric $g$ determines
a special autonomous almost complex structure $J_{g}\in\mathcal{J}$
called the \textbf{metric almost complex structure}. To define the
metric almost complex structure, we first recall that the metric $g$
determines a direct summand $T^{h}T^{*}M$ of the vertical tangent
bundle $T^{v}T^{*}M:=\ker\, d\pi$, together with an isomorphism \[
T_{x}T^{*}M=T_{x}^{h}T^{*}M\oplus T_{x}^{v}T^{*}M\cong T_{q}M\oplus T_{q}^{*}M,\ \ \ x=(q,p)\in T^{*}M.
\]
The metric almost complex structure $J_{g}$ is defined in terms of
this splitting by \begin{equation}
J_{g}:=\left(\begin{array}{cc}
0 & -\mathbb{1}\\
\mathbb{1} & 0
\end{array}\right).\label{eq:metric acs}
\end{equation}
Let $\mathcal{J}(\omega)$ denote the space of 1-periodic almost complex
structures on $T^{*}M$ that are $\omega$-compatible and satisfy
$\left\Vert J\right\Vert _{\infty}<\infty$. In general $J_{g}\notin\mathcal{J}(\omega)$.
However if $B_{r}(J_{g})$ denotes the open ball of radius $r>0$
about the metric almost complex structure $J_{g}$ in $\mathcal{J}$
then \cite[Proposition 4.1]{Lu1996} implies that there exists a constant
$\varepsilon_{0}=\varepsilon_{0}(g)>0$ (which depends continuously
on $g$) such that if $r>\varepsilon_{0}\left\Vert \sigma\right\Vert _{\infty}$
then%
\footnote{In fact, \cite[Proposition 4.1]{Lu1996} shows that for $r>\varepsilon_{0}\left\Vert \sigma\right\Vert _{\infty}$
we may even find\textbf{ }geometrically bounded almost complex structures
in $\mathcal{J}(\omega)\cap B_{r}(J_{g})$; see Remark \ref{rem:geometrically bounded}.%
} \begin{equation}
\mathcal{J}(\omega)\cap B_{r}(J_{g})\ne\emptyset\ \ \ \mbox{if }r>\varepsilon_{0}\left\Vert \sigma\right\Vert _{\infty}.\label{eq:ball not empty}
\end{equation}
This will be important in the proof of Theorem \ref{thm:l infinity};
see also Remark \ref{rem:shrinking sigma}. Given $J\in\mathcal{J}(\omega)$
we obtain a 1-periodic Riemannian metric $\left\langle \cdot,\cdot\right\rangle _{J}=\omega(J\cdot,\cdot)$
on $T^{*}M$. We will write $\left\langle \left\langle \cdot,\cdot\right\rangle \right\rangle _{J}$
for the $L^{2}$-metric on $\Lambda T^{*}M\times\mathbb{R}$ defined
by \begin{equation}
\left\langle \left\langle (\xi,b),(\rho,e)\right\rangle \right\rangle _{J}:=\int_{S^{1}}\left\langle \xi,\rho\right\rangle _{J}+be.\label{eq:the metric Jsigma}
\end{equation}
Finally let us remark that the first Chern class $c_{1}(T^{*}M,J)=0$
for any $J\in\mathcal{J}(\omega)$; one way to see this is that the
twisted symplectic manifold $(T^{*}M,\omega)$ admits a Lagrangian
distribution $T^{v}T^{*}M$ (see for example \cite[Example 2.10]{Seidel2000}).

\subsection{Ma\~n\'e's critical values}

$\ $\vspace{6 pt}

We now recall the definition of the two critical values $c$ and $c_{0}$
associated to the triple $(g,\sigma,U)$, introduced by Ma\~n\'e
in \cite{Mane1996}, which play a decisive role in all that follows.
General references for the results stated below are \cite[Proposition 2-1.1]{ContrerasIturriaga1999}
or \cite[Appendix A]{BurnsPaternain2002}.

Fix $U\in C^{\infty}(M,\mathbb{R})$, and let $H:T^{*}M\rightarrow\mathbb{R}$
be defined by $H:=H_{g}+\pi^{*}U$. Given $k\in\mathbb{R}$, let $\Sigma_{k}:=H^{-1}(k)$.
Define the \textbf{Ma\~n\'e critical value}\emph{ }associated to
the metric $g$, the weakly exact $2$-form $\sigma$ and the potential
$U$ by:\emph{ }\begin{equation}
c=c(g,\sigma,U):=\inf_{\theta}\sup_{q\in\widetilde{M}}\widetilde{H}(q,\theta_{q}),\label{eq:manc}
\end{equation}
where the infimum is taken over all $1$-forms $\theta$ on $\widetilde{M}$
with $d\theta=\widetilde{\sigma}$, and $\widetilde{H}$ is the lift
of $H$ to $T^{*}\widetilde{M}$. Thus $c(g,\sigma,U)<\infty$ if
and only if $\widetilde{\sigma}$ admits a bounded primitive.

If $\sigma$ is exact, define the \textbf{strict Ma}\~n\'e\textbf{
critical value}\emph{ $c_{0}=c_{0}(g,\sigma,U)$ }by \begin{equation}
c_{0}=c_{0}(g,\sigma,U):=\inf_{\theta}\sup_{q\in M}H(q,\theta_{q})<\infty,\label{eq:mansc}
\end{equation}
that is, the same definition only working directly on $T^{*}M$ rather
than lifting to $T^{*}\widetilde{M}$. If $\sigma$ is not exact,
set $c_{0}(g,\sigma,U):=\infty$. Note in all cases we have \[
c\leq c_{0}\leq\infty.
\]

The critical value can also be defined in Lagrangian terms. Let $L:=L_{g}-\pi^{*}U$
denote the Fenchel dual Lagrangian to $H$, and let $\widetilde{L}$
denote the lift of $L$ to $T\widetilde{M}$. Fix a primitive $\theta$
of $\widetilde{\sigma}$, and think of $\theta$ as a smooth function
on $T\widetilde{M}$. Now consider the Lagrangian $\widetilde{L}+\theta$.
The \textbf{action}\emph{ }$\mathbb{A}_{\widetilde{L}+\theta}(\gamma)$
on an absolutely continuous curve $\gamma:[a,b]\rightarrow\widetilde{M}$
is defined by \[
\mathbb{A}_{\widetilde{L}+\theta}(\gamma):=\int_{a}^{b}(\widetilde{L}+\theta)(\gamma(t),\dot{\gamma}(t))dt=\int_{a}^{b}\widetilde{L}(\gamma(t),\dot{\gamma}(t))+\theta_{\gamma(t)}(\dot{\gamma}(t))dt,
\]
and an alternative definition of $c$ is the following:\[
c:=\inf\left\{ k\in\mathbb{R}\,:\,\mathbb{A}_{\widetilde{L}+\theta+k}(\gamma)\geq0\ \forall\mbox{ a.c. closed curves defined on }[0,T],\,\forall T\in\mathbb{R}\right\} .
\]
If $\sigma$ is exact then we can pick a primitive $\theta$ of $\sigma$
and consider the same definition on $TM$. In this case we have: \[
c:=\inf\left\{ k\in\mathbb{R}\,:\,\mathbb{A}_{L+\theta+k}(\gamma)\geq0\ \forall\mbox{ a.c. closed homotopically trivial curves defined on }[0,T],\,\forall T\in\mathbb{R}\right\} ;
\]
\[
c_{0}:=\inf\left\{ k\in\mathbb{R}\,:\,\mathbb{A}_{L+\theta+k}(\gamma)\geq0\ \forall\mbox{ a.c. closed homologically trivial curves defined on }[0,T],\,\forall T\in\mathbb{R}\right\} .
\]
It is immediate from \eqref{eq:mansc} that\begin{equation}
c(g,\sigma,U)\geq\max_{q\in M}U(q).\label{eq:c ge e}
\end{equation}
Let us also denote by\[
e_{0}=e_{0}(g,\sigma,U):=\inf\left\{ k\in\mathbb{R}\,:\,\pi(\Sigma_{k})=M\right\} .
\]
For $k>e_{0}$ the intersection of $\Sigma_{k}$ with any fibre $T_{q}^{*}M$
is diffeomorphic to a sphere $S^{n-1}$. We always have $c\geq e_{0}$,
and in a lot of cases the strict inequality $c>e_{0}$ holds (see
\cite[Theorem 1.3]{PaternainPaternain1997a}). In all cases if $k>c$
then $k$ is necessarily a regular value of $H$.\newline

Denote by $\mathcal{R}(M)$ the set of all (smooth) Riemannian metrics
$g$ on $M$, and denote by $\Omega_{\textrm{we}}^{2}(M)$ the set
of closed weakly exact 2-forms on $M$.
\begin{defn}
\label{def:the set Uk}Denote by \[
\mathcal{O}\subseteq\mathcal{R}(M)\times\Omega_{\textrm{\emph{we}}}^{2}(M)\times C^{\infty}(M,\mathbb{R})\times\mathbb{R}
\]
the set of quadruples $(g,\sigma,U,k)$ such that \[
k>c(g,\sigma,U).
\]

\end{defn}
It will be important later on to know how the critical value scales
when we scale $\sigma$. Specifically, let us note the following lemma,
whose proof is immediate from \eqref{eq:manc} and \eqref{eq:mansc}.
\begin{lem}
\label{lem:scaling c}Given $s\in[0,1]$ it holds that \[
c(g,s\sigma,s^{2}U)=s^{2}c(g,\sigma,U);
\]
\[
c_{0}(g,s\sigma,s^{2}U)=s^{2}c(g,\sigma,U).
\]

\end{lem}

\subsection{\label{sub:The-crucial-observation}Symplectic atoroidality}

$\ $\vspace{6 pt}

We remind the reader that $\sigma\in\Omega^{2}(M)$ is a weakly exact
$2$-form whose pullback $\widetilde{\sigma}\in\Omega^{2}(\widetilde{M})$
admits a bounded primitive $\theta$. In this subsection we state
and prove the key observation mentioned in the introduction that implies
that the symplectic form $\omega$ is symplectically atoroidal. A
similar idea appeared in Niche \cite{Niche2006}, although there the
additional assumption was made that $M$ admits a metric of negative
curvature. Here we require only the weaker assumption that $\widetilde{\sigma}$
is weakly exact and admits a bounded primitive%
\footnote{This really is a weaker assumption; if $M$ admits a metric of negative
curvature then any closed $2$-form in $M$ has bounded primitives
in $\widetilde{M}$ \cite{Gromov1991}, whilst the converse is clearly
not true.%
}.\newline

The key lemma we use is the following, which originally appeared in
\cite[Lemma 2.2]{Merry2010}. In the statement, $\mathbb{T}^{2}$
denotes the $2$-torus.
\begin{lem}
\emph{\label{lem:key observation}}For any smooth map $f:\mathbb{T}^{2}\rightarrow M$,
$f^{*}\sigma$ is exact.\end{lem}
\begin{proof}
Consider $G:=f_{*}(\pi_{1}(\mathbb{T}^{2}))\leq\pi_{1}(M).$ Then
$G$ is amenable, since $\pi_{1}(\mathbb{T}^{2})=\mathbb{Z}^{2}$,
which is amenable. Then \cite[Lemma 5.3]{Paternain2006} tells us
that since $\left\Vert \theta\right\Vert _{\infty}<\infty$ we can
replace $\theta$ by a $G$-invariant primitive $\theta'$ of $\widetilde{\sigma}$,
which descends to define a primitive $\theta''\in\Omega^{1}(\mathbb{T}^{2})$
of $f^{*}\sigma$.
\end{proof}
Given a free homotopy class $\alpha\in[S^{1},M]$, fix a reference
loop $x_{\alpha}=(q_{\alpha},p_{\alpha})\in\Lambda_{\alpha}T^{*}M$.
It will be convenient to insist that $x_{-\alpha}(t)=x_{\alpha}(-t)$,
and that $x_{0}$ has image in one fibre, that is, $q_{0}$ is constant.
Let $C:=S^{1}\times[0,1]$. Let $\bar{x}\in C^{0}(C,T^{*}M)\cap W^{1,2}(C,T^{*}M)$
denote any map such that $\bar{x}(\cdot,0)=x$ and $\bar{x}(\cdot,1)=x_{\alpha}$.
Then thanks to the previous lemma the integral $\int_{C}\bar{x}^{*}\pi^{*}\sigma$
is is independent of the choice of $\bar{x}$. Similarly given any
$q\in\Lambda_{\alpha}M$, let $\bar{q}\in C^{0}(C,M)\cap W^{1,2}(C,M)$
denote any map such that $\bar{q}(\cdot,0)=q$ and $\bar{q}(\cdot,1)=q_{\alpha}$.
Then the integral $\int_{C}\bar{q}^{*}\sigma$ is independent of the
choice of $\bar{q}$. Note that in particular if $q=\pi\circ x$ then\begin{equation}
\int_{C}\bar{x}^{*}\pi^{*}\sigma=\int_{C}\bar{q}^{*}\sigma,\label{eq:czequalscx}
\end{equation}
and hence, recalling that $\lambda_{0}=pdq$ is the Liouville $1$-form
on $T^{*}M$, it holds that\begin{equation}
\int_{C}\bar{x}^{*}\omega=\int_{S^{1}}x^{*}\lambda_{0}+\int_{C}\bar{x}^{*}\pi^{*}\sigma=\int_{S^{1}}x^{*}\lambda_{0}+\int_{C}\bar{q}^{*}\sigma.\label{eq:liouville 1 form and c}
\end{equation}

\section{The free time action functional}

\subsection{The definition of $S_{L+k}$}

$\ $\vspace{6 pt}

The first functional we work with is defined on $\Lambda M\times\mathbb{R}^{+}$.
Given a potential $U\in C^{\infty}(M,\mathbb{R})$ and $k\in\mathbb{R}$
let \[
L(q,v):=\frac{1}{2}\left|v\right|^{2}-U(q)
\]
and define the \textbf{free time action functional}\emph{ }$S_{L+k}:\Lambda M\times\mathbb{R}^{+}\rightarrow\mathbb{R}$
by\emph{ }\[
S_{L+k}(q,T):=T\int_{S^{1}}\left(L\left(q(t),\frac{\dot{q}(t)}{T}\right)+k\right)dt+\int_{C}\bar{q}^{*}\sigma.
\]
This is well defined by the observations in the previous section.
Moreover $S_{L+k}\in C^{2}(\Lambda M\times\mathbb{R}^{+},\mathbb{R})$;
see \cite[p195]{Merry2010}. Let $\mbox{Crit}(S_{L+k})$ denote the
set of critical points of $S_{L+k}$, and given $\alpha\in[S^{1},M]$,
let $\mbox{Crit}(S_{L+k};\alpha)$ denote $\mbox{Crit}(S_{L+k})\cap(\Lambda_{\alpha}M\times\mathbb{R}^{+})$.
Given an interval $(a,b)\subseteq\mathbb{R}$, denote by $\mbox{Crit}^{(a,b)}(S_{L+k})$
the set $\mbox{Crit}(S_{L+k})\cap S_{L+k}^{-1}((a,b))$. This functional
was introduced in \cite{Merry2010}, and is a way of defining the
free time action functional previously studied in \cite{ContrerasIturriagaPaternainPaternain2000,Contreras2006}
when the magnetic form $\sigma$ is not exact.\newline 

It will be convenient to study what is essentially the lift of $S_{L+k}$
to the universal cover $\widetilde{M}$. Let $\widetilde{U}$ denote
a lift of $U$ to $\widetilde{M}$. Let $\widetilde{E}:T\widetilde{M}\rightarrow\mathbb{R}$
denote the \textbf{energy}\emph{ }of the Lagrangian $\widetilde{L}$:
\[
\widetilde{E}(q,v):=\frac{\partial\widetilde{L}}{\partial v}(q,v)(v)-\widetilde{L}(q,v).
\]
Fix a primitive $\theta$ of the lifted form $\widetilde{\sigma}$
on $\widetilde{M}$ with $\left\Vert \theta\right\Vert _{\infty}<\infty$,
and consider again the Lagrangian $\widetilde{L}+\theta:T\widetilde{M}\rightarrow\mathbb{R}$.
Define \[
\mathbb{S}_{\widetilde{L}+\theta+k}:W^{1,2}([0,1],\widetilde{M})\times\mathbb{R}^{+}\rightarrow\mathbb{R}
\]
by \[
\mathbb{S}_{\widetilde{L}+\theta+k}(q,T):=T\int_{0}^{1}\left(\widetilde{L}\left(q(t),\frac{\dot{q}(t)}{T}\right)+\frac{1}{T}\theta_{q(t)}(\dot{q}(t))+k\right)dt.
\]
In other words, $\mathbb{S}_{\widetilde{L}+\theta+k}$ is the standard
free time action functional of the Lagrangian $\widetilde{L}+\theta$
and the energy level $k$. The free time action functional has been
studied extensively in \cite{ContrerasIturriagaPaternainPaternain2000,Contreras2006}.
We wish to relate the functional $\mathbb{S}_{\widetilde{L}+\theta+k}$
to that of $S_{L+k}$. For each $\alpha\in[S^{1},M]$, fix a lift
$\widetilde{q}_{\alpha}:[0,1]\rightarrow\widetilde{M}$ of our reference
loops $q_{\alpha}\in\Lambda_{\alpha}M$. Define\begin{equation}
I(\alpha,\theta):=\int_{0}^{1}(\widetilde{q}_{\alpha}^{-})^{*}\theta\label{eq:ah}
\end{equation}
(where $\widetilde{q}_{\alpha}^{-}(t):=\widetilde{q}_{\alpha}(-t)$).
Note that as $q_{0}$ is constant, $I(0,\theta)=0$. It is shown in
\cite[p8]{Merry2010} that given $q\in\Lambda_{\alpha}M$ and $\widetilde{q}$
a lift of $q$ and $\bar{q}:C\rightarrow M$ a map as above that\begin{equation}
\int_{C}\bar{q}^{*}\sigma=\int_{0}^{1}\widetilde{q}^{*}\theta+I(\alpha,\theta),\label{eq:relating S and F}
\end{equation}
from which it follows that\begin{equation}
S_{L+k}(q,T)=\mathbb{S}_{\widetilde{L}+\theta+k}(\widetilde{q},T)+I(\alpha,\theta).\label{eq:relating to uni cover}
\end{equation}

Since $\left\Vert \theta\right\Vert _{\infty}<\infty$, we can find
constants $e_{1},e_{2},f_{1},f_{2},g_{1},g_{2}>0$ such that for all
$(q,v)\in T\widetilde{M}$ it holds that \begin{equation}
f_{1}\left|v\right|^{2}+f_{2}\geq(\widetilde{L}+\theta)(q,v)\geq e_{1}\left|v\right|^{2}-e_{2};\label{eq:superlinear-1}
\end{equation}
\[
\widetilde{E}(q,v)\geq g_{1}\left|v\right|^{2}-g_{2}.
\]
Given any $(q,T)\in\Lambda_{\alpha}M\times\mathbb{R}^{+}$, let $\widetilde{q}:[0,1]\rightarrow\widetilde{M}$
denote a lift of $q$ and define $\gamma:[0,T]\rightarrow\widetilde{M}$
by $\gamma(t):=\widetilde{q}(t/T)$. One computes\begin{align*}
\frac{\partial S_{L+k}}{\partial T}(q,T) & =\frac{1}{T}\int_{0}^{T}\left(k-\widetilde{E}(\gamma,\dot{\gamma})\right)dt\\
 & \leq\frac{1}{T}\int_{0}^{T}\left(k-\frac{g_{1}}{f_{1}}\widetilde{(L}+\theta)(\gamma,\dot{\gamma})+\frac{g_{1}f_{2}}{f_{1}}+g_{2}\right)dt\\
 & =\frac{g_{1}f_{2}}{f_{1}}+g_{2}+\left(1+\frac{g_{1}}{f_{1}}\right)k-\frac{g_{1}}{f_{1}T}\mathbb{S}_{\widetilde{L}+\theta+k}(\widetilde{q},T)\\
 & =\frac{g_{1}f_{2}}{f_{1}}+g_{2}+\left(1+\frac{g_{1}}{f_{1}}\right)k-\frac{g_{1}}{f_{1}T}S_{L+k}(q,T)+\frac{g_{1}I(\alpha,\theta)}{f_{1}T}.
\end{align*}
In particular, in the case $\alpha=0$, since $I(0,\theta)=0$ we
have proved the following lemma.
\begin{lem}
\label{lem:d/deta}There exists $h_{0}>0$ such that if $(q,T)\in\Lambda_{0}M\times\mathbb{R}^{+}$
and \[
S_{L+k}(q,T)>h_{0}T
\]
 then \[
\frac{\partial S_{L+k}}{\partial T}(q,T)<0.
\]

\end{lem}
Let us recall a few definitions. If $S:\mathcal{M}\rightarrow\mathbb{R}$
is a $C^{2}$ functional on a Hilbert manifold $\mathcal{M}$ equipped
with a Riemannian metric $G$, we say that $S$ satisfies $(\mbox{PS})_{a}$,
that is, \textbf{Palais-Smale condition at the level $a\in\mathbb{R}$,}\emph{
}if any sequence $(x_{i})\subseteq\mathcal{M}$ such that $S(x_{i})\rightarrow a$
and $\left\Vert \nabla S(x_{i})\right\Vert \rightarrow0$ admits a
convergent subsequence (where the gradient $\nabla S$ is taken with
respect to $G$). Let $\Psi_{\tau}$ denote the local flow defined
by the vector field $-\nabla S$, and let $(\tau_{-}(x),\tau_{+}(x))\subseteq\overline{\mathbb{R}}$
denote the maximal interval of existence of the flow line $\tau\mapsto\Psi_{\tau}(x)$.\newline

The next result is the key to defining the Morse (co)complex of $S_{L+k}$
(compare \cite[Proposition 11.1, Proposition 11.2]{AbbondandoloSchwarz2009}).
Recall the definition of the set $\mathcal{O}$ from Definition \ref{def:the set Uk}.
\begin{thm}
\textup{\emph{\label{thm:(Properties-of-)-1}(Properties of $S_{L+k}$
for $k>c(g,\sigma,U)$)}}

Fix $(g,\sigma,U,k)\in\mathcal{O}$. Let $\Psi_{\tau}$ denote the
local flow of $-\nabla S_{L+k}$, where the gradient is taken with
respect to the metric $\left\langle \left\langle \cdot,\cdot\right\rangle \right\rangle _{g}$
from \eqref{eq:the metric on the loop space}. Then:
\begin{enumerate}
\item $S_{L+k}$ is bounded below on $\Lambda M\times\mathbb{R}^{+}$ and
strictly positive on $\Lambda_{0}M\times\mathbb{R}^{+}$. Moreover
\[
\inf_{\Lambda_{0}M\times\mathbb{R}^{+}}S_{L+k}=0,\ \ \ \inf_{\textrm{\emph{Crit}}(S_{L+k};0)}S_{L+k}>0.
\]

\item If $\alpha\in[S^{1},M]$ is a non-trivial free homotopy class then
$\tau_{+}(q,T)=\infty$ for all $(q,T)\in\Lambda_{\alpha}M\times\mathbb{R}^{+}$.
If $(q,T)\in\Lambda_{0}M\times\mathbb{R}^{+}$ and $\tau_{+}(q,T)<\infty$
then if $(q_{\tau},T_{\tau}):=\Psi_{\tau}(q,T)$ then $S_{L+k}(q_{\tau},T_{\tau})\rightarrow0$,
$T_{\tau}\rightarrow0$ and $q_{\tau}$ converges to a constant loop
as $\tau\uparrow\tau_{+}(q,T)$. In particular this happens if \[
S_{L+k}(q,T)<\inf_{\textrm{\emph{Crit}}(S_{L+k};0)}S_{L+k}.
\]

\item If $\alpha\in[S^{1},M]$ is a non-trivial free homotopy class then
$\tau_{-}(q,T)=-\infty$ for all $(q,T)\in\Lambda_{\alpha}M\times\mathbb{R}^{+}$. 
\item There exists $h_{1}>0$ with the following properties: given $S>0$
define\[
\mathcal{A}(S):=\{S_{L+k}|_{\Lambda_{0}M\times\mathbb{R}^{+}}<S\}\cap\{T<h_{1}S\}
\]
Then $\mathcal{A}(S)\cap\mbox{\emph{Crit}}(S_{L+k})=\emptyset$ for
all $S>0$, and for any $S>0$, if $(q,T)\in\mathcal{A}(S)$ then
$\Psi_{\tau}(q,T)\in\mathcal{A}(S)$ for all $\tau\in(\tau_{-}(q,T),0]$.
Finally if $(q,T)\in\Lambda_{0}M\times\mathbb{R}^{+}$ is such that
$\tau_{-}(q,T)>-\infty$ and $S_{L+k}(q,T)\geq S$ then there exists
$\tau<0$ such that $\Psi_{\tau}(q,T)\in\mathcal{A}(S)$.
\end{enumerate}
\end{thm}
\begin{proof}
The fact that $S_{L+k}$ is bounded below is proved%
\footnote{Strictly speaking, all the proofs in \cite{Merry2010} are given only
in the special case $U\equiv0$, but there are no changes in this
case.%
} in \cite[Lemma 4.2]{Merry2010}. The fact that $S_{L+k}$ is strictly
positive on $\Lambda_{0}M\times\mathbb{R}^{+}$ follows from the fact
that given $(q,T)\in\Lambda_{0}M\times\mathbb{R}^{+}$ we have\[
S_{L+k}(q,T)=\mathbb{S}_{\widetilde{L}+\theta+k}(\widetilde{q},T)=\mathbb{S}_{\widetilde{L}+\theta+c}(\widetilde{q},T)+(k-c)T\geq0+(k-c)T.
\]
If $q$ is a constant loop then $\lim_{T\rightarrow0}S_{L+k}(q,T)=0$,
and hence the infimum of $S_{L+k}$ on $\Lambda_{0}M\times\mathbb{R}^{+}$
is zero. To see that the infimum of $S_{L+k}$ on $\mbox{Crit}(S_{L+k};0)$
is strictly positive, we use Lemma \ref{lem:relating critical points lemma},
to be proved in the next section, which says that $(q,T)\in\mbox{Crit}(S_{L+k})$
if and only if $(x,T)\in\mbox{Crit}(A_{H-k})$, where $x=(q,\dot{q})\in\Lambda T^{*}M$.
Since $\Sigma_{k}:=H^{-1}(k)$ is compact and $k$ is a regular value
of $H$, the period of its Hamiltonian orbits is bounded away from
zero, and thus \[
\inf\left\{ \eta>0\,:\,(x,\eta)\in\mbox{Crit}(A_{H-k})\right\} >0.
\]
Thus the infimum of $S_{L+k}$ on $\mbox{Crit}(S_{L+k};0)$ is strictly
positive. This proves (1).

Statement (2) is proved in \cite[Theorem 3.2, Lemma 4.4]{Merry2010}.
Since $S_{L+k}$ is bounded below, if $(q,T)\in\Lambda M\times\mathbb{R}^{+}$
is such that $\tau_{+}(q,T)<\infty$ then if $(q_{\tau},T_{\tau}):=\Psi_{\tau}(q,T)$,
we must have $\lim_{\tau\uparrow\tau_{+}(q,T)}T_{\tau}=0$ (see for
instance \cite[Proposition 8.4]{MawhimWillem1989}). This can only
happen if $(q,T)\in\Lambda_{0}M\times\mathbb{R}^{+}$, since if $q$
is non-contractible then $T$ is bounded away from zero (\cite[Lemma 4.3]{Merry2010}).
If $(q,T)\in\Lambda_{0}M\times\mathbb{R}^{+}$ then we have \begin{align*}
\frac{\partial T_{\tau}}{\partial\tau} & =\left\langle \left\langle \frac{\partial}{\partial\tau}\Psi_{\tau}(q,T),\left(0,\frac{\partial}{\partial T}\right)\right\rangle \right\rangle _{g}\\
 & =-\frac{\partial S_{L+k}}{\partial T}(q_{\tau},T_{\tau}),
\end{align*}
and thus Lemma \ref{lem:d/deta} tells us that if $S_{L+k}(q_{\tau},T_{\tau})>h_{0}T_{\tau}$
then $\frac{\partial T_{\tau}}{\partial\tau}>0$. Thus the decreasing
function $\tau\mapsto S_{L+k}(q_{\tau},T_{\tau})$ must converge to
zero. Using \eqref{eq:superlinear-1} it is easy to see that the fact
that both $S_{L+k}(q_{\tau},T_{\tau})$ and $T_{\tau}$ tend to zero
implies that $\int_{S^{1}}\left|\dot{q}_{\tau}(t)\right|^{2}dt$ also
tends to zero as $\tau\uparrow\tau_{+}(q,T)$. This proves (3). The
proof of (4) follows in exactly the same way (see \cite[Proposition 11.2]{AbbondandoloSchwarz2009}).
\end{proof}

\subsection{Fixing the period}

$\ $\vspace{6 pt}

It will be useful to consider the \textbf{fixed period action functional}.
Given $T\in\mathbb{R}^{+}$ let us denote by $S_{L+k}^{T}:\Lambda M\rightarrow\mathbb{R}$
the functional defined by \[
S_{L+k}^{T}(q):=S_{L+k}(q,T).
\]
Note that \[
d_{q}S_{L+k}^{T}(\zeta)=d_{(q,T)}S_{L+k}(\zeta,0).
\]
Thus if $(q,T)\in\mbox{Crit}(S_{L+k})$ then $q\in\mbox{Crit}(S_{L+k}^{T})$.

\subsection{The Morse index and the non-degeneracy assumption}

$\ $\vspace{6 pt}

By definition, the \textbf{Morse index}\emph{ }$i(q,T)$ of a critical
point $(q,T)\in\mbox{Crit}(S_{L+k})$ is the maximal dimension of
a subspace $W\subseteq W^{1,2}(S^{1},q^{*}TM)\times\mathbb{R}$ on
which the Hessian $\nabla_{(q,T)}^{2}S_{L+k}$ of of $S_{L+k}$ at
$(q,T)$ is negative definite. It is well known that for the Lagrangians
$L=L_{g}-\pi^{*}U$ that we work with the Morse index $i(q,T)$ is
always finite \cite[Section 1]{Duistermaat1976}. Similarly let $i_{T}(q)$
denote the Morse index\textbf{ }of a critical point $q\in\mbox{Crit}(S_{L+k}^{T})$,
that is, the dimension of a maximal subspace of $W^{1,2}(S^{1},q^{*}TM)$
on which the Hessian $\nabla_{q}^{2}S_{L+k}^{T}$ of the $S_{L+k}$
at $(q,T)$ (this time taken with respect to the $W^{1,2}$ metric
on $\Lambda M$) is negative definite. 
\begin{defn}
Let us say that a critical point $(q,T)\in\mbox{\emph{Crit}}(S_{L+k})$
is \textbf{\emph{non-degenerate }}if the kernel of $\nabla_{q}^{2}S_{L+k}^{T}$
is one-dimensional, spanned by the vector $\dot{q}\in T_{q}\Lambda M$. 
\end{defn}
Suppose $(q,T)$ is a non-degenerate critical point. One consequence
of this assumption (cf. the discussion at the start of Section \ref{sub:Grading-the-Rabinowitz})
is the existence of an \textbf{orbit cylinder }about $(q,T)$. That
is, there exists $\varepsilon>0$ and a unique smooth (in $s$) family
$(q_{s},T_{s})\in\mbox{Crit}(S_{L+k+s})$ for $s\in(-\varepsilon,\varepsilon)$,
where $(q_{0},T_{0})=(q,T)$. Moreover $\frac{\partial T_{s}}{\partial s}(0)\ne0$.
Given such a non-degenerate critical point $(q,T)$, we may therefore
define \begin{equation}
\chi(q,T):=\mbox{sign}\left(-\frac{\partial T_{s}}{\partial s}(0)\right)\in\{-1,1\}.\label{eq:chi q T}
\end{equation}

Recall that a function $S:\mathcal{M}\rightarrow\mathbb{R}$ on on
a Hilbert manifold $\mathcal{M}$ equipped with a Riemannian metric
$G$ is called \textbf{Morse-Bott}\emph{ }if the set $\mbox{Crit}(S)$
of its critical points is a submanifold of $\mathcal{M}$ (possibly
with components of differing dimensions) and such that for each $x\in\mbox{Crit}(S)$,
the Hessian $\nabla_{x}^{2}S$ of $S$ (defined with respect to $G$)
is a Fredholm operator and satisfies \[
\ker\,\nabla_{x}^{2}S=T_{x}\mbox{Crit}(S).
\]

Denote by \[
\mathcal{O}_{\textrm{reg}}\subseteq\mathcal{O}
\]
the set of quadruples $(g,\sigma,U,k)\in\mathcal{O}$ with the property
that if $L:=L_{g}-\pi^{*}U$ then every critical point of $S_{L+k}$
is non-degenerate. In this case $S_{L+k}$ is a Morse-Bott function,
and $\mbox{Crit}(S_{L+k})$ consists of a discrete union of circles.
The following theorem can be proved by adapting the proofs of \cite[Theorem B1]{CieliebakFrauenfelder2009}
(see also the Corrigendum \cite{CieliebakFrauenfelder2010}) together
with a version of the Klingenberg-Takens theorem \cite{KlingenbergTakens1972}
for magnetic flows. Full details can be found in \cite{Merry2011}.
\begin{thm}
\label{thm:generic}Suppose $(g,\sigma,U,k)\in\mathcal{O}$. Then
given any $\varepsilon>0$ there exists $g'\in\mathcal{R}(M)$ with
$\left\Vert g-g'\right\Vert _{\infty}<\varepsilon$ such that $(g',\sigma,U,k)\in\mathcal{O}_{\textrm{\emph{reg}}}$. 
\end{thm}
The following theorem is proved in \cite{MerryPaternain2010}.
\begin{thm}
\label{thm:lag index}Assume $(g,\sigma,U,k)\in\mathcal{O}_{\textrm{\emph{reg}}}$,
and set $L:=L_{g}-\pi^{*}U$. Let $(q,T)\in\mbox{\emph{Crit}}(S_{L+k})$.
Then \[
i(q,T)=i_{T}(q)+\frac{1}{2}-\frac{1}{2}\chi(q,T).
\]
\end{thm}
\begin{rem}
\label{rem:Index jumping on the morse side}In \cite[Section 10]{AbbondandoloSchwarz2009}
Abbondandolo and Schwarz work with a Lagrangian which is the Fenchel
transform of a Hamiltonian which is homogeneous of degree $2$ in
a neighborhood of $\Sigma_{k}$. In this case one can show $\chi(q,T)=+1$
for every critical point $(q,T)$, and hence the Morse index of the
free time action functional always agrees with the corresponding index
of the fixed period action functional. In the more general situation
that we are interested in here however it is possible that there exist
critical points $(q,T)$ with $\chi(q,T)=-1$. In \cite{MerryPaternain2010}
we provide an example of an exact magnetic Lagrangian $L:TS^{2}\rightarrow\mathbb{R}$
for which there exists a non-degenerate critical point $(q,T)$ of
$ $$S_{L+k}$ for $k>c$ such that $\chi(q,T)=-1$.
\end{rem}

\subsection{The Morse (co)chain complex}

$\ $\vspace{6 pt}

In this section we construct the Morse co(chain) complex and state
the Morse homology theorem, which says that the corresponding Morse
(co)homology coincides with the singular (co)homology of the free
loop space $\Lambda M$. Fix $(g,\sigma,U,k)\in\mathcal{O}_{\textrm{reg}}$,
and put $L:=L_{g}-\pi^{*}U$.

It will be convenient to put \[
\overline{\mbox{Crit}}(S_{L+k}):=\mbox{Crit}(S_{L+k})\cup(M\times\{0\}),
\]
where points in $M$ should be thought of as constant loops in $\Lambda M$.
We refer to elements of the set $\overline{\mbox{Crit}}(S_{L+k})\backslash\mbox{Crit}(S_{L+k})$
as \textbf{critical points at infinity}%
\footnote{In a lot of ways this is a poor choice of name, as these critical
points lie at $ $$T=0$, not at $T=\infty$!%
}.\emph{ }

We will need three pieces of auxiliary data to define the Morse (co)complex.
Firstly, let $G$ denote a metric on $\Lambda M\times\mathbb{R}^{+}$
that is a generic perturbation of the metric $\left\langle \left\langle \cdot,\cdot\right\rangle \right\rangle _{g}$
(in particular $G$ should be uniformly equivalent to $\left\langle \left\langle \cdot,\cdot\right\rangle \right\rangle _{g}$).
Write $\Psi_{\tau}$ for the flow of $-\nabla S_{L+k}$, now taken
with respect to the metric $G$. Secondly, let $f:\overline{\mbox{Crit}}(S_{L+k})\rightarrow\mathbb{R}$
denote a Morse function on $\overline{\mbox{Crit}}(S_{L+k})$, and
write $\overline{\mbox{Crit}}(f)\subseteq\overline{\mbox{Crit}}(S_{L+k})$
for the set of critical points of $f$, and $\mbox{Crit}(f):=\mbox{Crit}(S_{L+k})\cap\overline{\mbox{Crit}}(f)$.
Thirdly, let $g_{0}$ denote a Riemannian metric on $\overline{\mbox{Crit}}(S_{L+k})$
such that the flow $\phi_{t}^{-\nabla f}$ of $-\nabla f$ is Morse-Smale.
The Morse-Smale assumption implies that for every pair $w_{-},w_{+}$
of critical points of $f$ the unstable manifold $W^{u}(w_{-};-\nabla f)$
intersects the stable manifold $W^{s}(w_{+};-\nabla f)$ transversely.
Denote by $i_{f}(z):=\dim\, W^{u}(w;-\nabla f)$ the Morse index of
a critical point $z\in\overline{\mbox{Crit}}(f)$.

Finally for $w\in\overline{\mbox{Crit}}(f)$ write \[
\widehat{i}_{f}(w):=i(w)+i_{f}(w),
\]
where by definition we put $i(w)=0$ for $w\in\overline{\mbox{Crit}}(S_{L+k})\backslash\mbox{Crit}(S_{L+k})$.
Let\[
\overline{\mbox{Crit}}_{i}(f):=\left\{ w\in\overline{\mbox{Crit}}(f)\,:\,\widehat{i}_{f}(w)=i\right\} .
\]
Given $w_{-},w_{+}\in\overline{\mbox{Crit}}(f)$, denote by \[
\widetilde{\mathcal{W}}_{0}(w_{-},w_{+}):=W^{u}(w_{-};-\nabla f)\cap W^{s}(w_{+};-\nabla f).
\]
Let \[
\mathcal{W}_{0}(w_{-},w_{+}):=\widetilde{\mathcal{W}}_{0}(w_{-},w_{+})/\mathbb{R}
\]
 denote the quotient of $\widetilde{\mathcal{W}}_{0}(w_{-},w_{+})$
by the obvious free $\mathbb{R}$-action (if $w_{-}=w_{+}$, $\mathcal{W}_{0}(w_{-},w_{+})=\emptyset$).

Suppose now that $w_{-}\in\mbox{Crit}(f)$, that is, $w_{-}$ is \textbf{not}\emph{
}a critical point at infinity. If $m\in\mathbb{N}$ and $w_{+}\in\overline{\mbox{Crit}}(f)$,
let $\widetilde{\mathcal{W}}_{m}(w_{-},w_{+})$ denote the set of
tuples $\boldsymbol{w}=(w_{1},\dots,w_{m})$ where each $w_{i}\in(\Lambda M\times\mathbb{R}^{+})\backslash\mbox{Crit}(S_{L+k})$
is such that \[
\Psi_{-\infty}(w_{1})\in W^{u}(w_{-};-\nabla f),\dots,\Psi_{\infty}(w_{m})\in W^{s}(w_{+};-\nabla f),
\]
and such that \[
\Psi_{-\infty}(w_{i+1})\in\phi_{\mathbb{R}^{+}}^{-\nabla f}(\Psi_{\infty}(w_{i})).
\]
Note that if $m\geq1$ then $\widetilde{\mathcal{W}}_{m}(w^{-},w^{+})$
admits a free action of $\mathbb{R}^{m}$ via\[
(w_{1},\dots,w_{m})\mapsto(\Psi_{s_{1}}(w_{1}),\dots,\Psi_{s_{m}}(w_{m})),\ \ \ (s_{1},\dots,s_{m})\in\mathbb{R}^{m}.
\]
We denote by $\mathcal{W}_{m}(w_{-},w_{+})$ the quotient of $\widetilde{\mathcal{W}}_{m}(w_{-},w_{+})$
by this action. Put \[
\mathcal{W}(w_{-},w_{+}):=\bigcup_{m\in\mathbb{N}\cup\{0\}}\mathcal{W}_{m}(w_{-},w_{+}).
\]

Finally if $w_{-}\in\overline{\mbox{Crit}}(f)\backslash\mbox{Crit}(f)$
is a critical point at infinity, set \[
\widetilde{\mathcal{W}}_{m}(w_{-},w_{+})=\mathcal{W}_{m}(w_{-},w_{+}):=\emptyset
\]
for all $m\in\mathbb{N}$ and $w_{+}\in\overline{\mbox{Crit}}(f)$,
so that $\mathcal{W}(w_{-},w_{+})=\mathcal{W}_{0}(w_{-},w_{+})$.\newline

The next theorem, together with Theorem \ref{thm:morse homology}
below, follows from Theorem \ref{thm:(Properties-of-)-1} exactly
as in \cite[Section 11]{AbbondandoloSchwarz2009}. See also \cite[Appendix A]{Frauenfelder2004}
for more information.
\begin{thm}
\label{thm:generic morse}For a generic choice of $G$ and $g_{0}$
the set $\mathcal{W}(w_{-},w_{+})$ is a finite dimensional smooth
manifold of dimension \[
\dim\,\mathcal{W}(w_{-},w_{+})=\widehat{i}_{f}(w_{-})-\widehat{i}_{f}(w_{+})-1.
\]
Moreover if $\widehat{i}_{f}(w_{-})-\widehat{i}_{f}(w_{+})=1$ then
$\mathcal{W}(w_{-},w_{+})$ is compact, and hence a finite set. 
\end{thm}
If $\widehat{i}_{f}(w_{-})-\widehat{i}_{f}(w_{+})=1$ we may therefore
define \[
n_{\textrm{Morse}}(w_{-},w_{+}):=\#\mathcal{W}(w_{-},w_{+}),\ \ \ \mbox{taken mod }2.
\]
Put \[
CM_{i}(S_{L+k},f):=\bigoplus_{w\in\overline{\textrm{Crit}}_{i}(f)}\mathbb{Z}_{2}w,\ \ \ CM^{i}(S_{L+k},f):=\prod_{w\in\overline{\textrm{Crit}}_{i}(f)}\mathbb{Z}_{2}w.
\]
Define \[
\partial^{\textrm{Morse}}=\partial^{\textrm{Morse}}(G,g_{0}):CM_{i}(S_{L+k},f)\rightarrow CM_{i-1}(S_{L+k},f)
\]
by \[
\partial^{\textrm{Morse}}w=\sum_{w'\in\overline{\textrm{Crit}}_{i-1}(f)}n_{\textrm{Morse}}(w,w')w'.
\]
Define \[
\delta^{\textrm{Morse}}=\delta^{\textrm{Morse}}(G,g_{0}):CM^{i}(S_{L+k},f)\rightarrow CM^{i+1}(S_{L+k},f)
\]
by \[
\delta^{\textrm{Morse}}w:=\sum_{w'\in\overline{\textrm{Crit}}_{i+1}(f)}n_{\textrm{Morse}}(w',w)w'.
\]

The next result is the \textbf{Morse homology theorem}. 
\begin{thm}
\label{thm:morse homology}Let $G$ and $g_{0}$ be as Theorem \ref{thm:generic morse}.
Then it holds that $\partial^{\textrm{\emph{Morse}}}\circ\partial^{\textrm{\emph{Morse}}}=0$
and also that $\delta^{\textrm{\emph{Morse}}}\circ\delta^{\textrm{\emph{Morse}}}=0$.
Thus $\{CM_{*}(S_{L+k},f),\partial^{\textrm{\emph{Morse}}}(G,g_{0})\}$
and $\{CM^{*}(S_{L+k},f),\delta^{\textrm{\emph{Morse}}}(G,g_{0})\}$
form a chain (respectively cochain) complex. The isomorphism class
of these complexes is independent of the choice of $f$, $G$ and
$g_{0}$. The associated (co)homology, known as the \textbf{\emph{Morse
(co)homology}}\emph{ }of $S_{L+k}$ is isomorphic to the singular
(co)homology of $\Lambda M\times\mathbb{R}^{+}$:\[
HM_{*}(S_{L+k})\cong H_{*}(\Lambda M\times\mathbb{R}^{+};\mathbb{Z}_{2}),\ \ \ HM^{*}(S_{L+k})\cong H^{*}(\Lambda M\times\mathbb{R}^{+};\mathbb{Z}_{2}).
\]

\end{thm}
Moreover this isomorphism respects the splitting $\Lambda M=\bigoplus_{\alpha\in[S^{1},M]}\Lambda_{\alpha}M$:
if $CM_{*}(S_{L+k},f;\alpha)$ denotes the subcomplex of $CM_{*}(S_{L+k},f)$
generated by the critical points $w\in\overline{\mbox{Crit}}(f)\cap\overline{\mbox{Crit}}(S_{L+k};\alpha)$
then the homology $HM_{*}(S_{L+k};\alpha)$ of this subcomplex is
isomorphic to $H_{*}(\Lambda_{\alpha}M\times\mathbb{R}^{+};\mathbb{Z}_{2})$
under the isomorphism of the previous theorem. The same statements
holds for cohomology: $HM^{*}(S_{L+k};\alpha)\cong H^{*}(\Lambda_{\alpha}M\times\mathbb{R}^{+})$.

\section{The Rabinowitz action functional}

In this section we finally define the Rabinowitz action functional,
and its associated Rabinowitz Floer homology.

\subsection{Definition of the Rabinowitz action functional}

$\ $\vspace{6 pt}

Fix an autonomous potential $U\in C^{\infty}(M,\mathbb{R})$, and
put $H=H_{g}+\pi^{*}U$. Fix a regular value $k\in\mathbb{R}$ of
$H$, and put $\Sigma_{k}:=H^{-1}(k)$. We define the \textbf{Rabinowitz
action functional}\emph{ }$A_{H-k}:\Lambda T^{*}M\times\mathbb{R}\rightarrow\mathbb{R}$
by \begin{align*}
A_{H-k}(x,\eta): & =\int_{C}\bar{x}^{*}\omega-\eta\int_{S^{1}}(H(x(t))-k)dt,\\
 & =\int_{S^{1}}x^{*}\lambda_{0}+\int_{C}\bar{x}^{*}\pi^{*}\sigma-\eta\int_{S^{1}}(H(x(t))-k)dt
\end{align*}
(see Section \ref{sub:The-crucial-observation} for the definition
of the term $\int_{C}\bar{x}^{*}\omega$; the latter equality follows
from \eqref{eq:liouville 1 form and c}). Denote by $\mbox{Crit}(A_{H-k})$
the set of critical points of $A_{H-k}$, and given $\alpha\in[S^{1},M]$,
let $\mbox{Crit}(A_{H-k};\alpha):=\mbox{Crit}(A_{H-k})\cap(\Lambda_{\alpha}T^{*}M\times\mathbb{R})$.
Given an interval $(a,b)\subseteq\mathbb{R}$, denote by $\mbox{Crit}^{(a,b)}(A_{H-k})$
the set $\mbox{Crit}(A_{H-k})\cap A_{H-k}^{-1}((a,b))$.

The critical points of $A_{H-k}$ are easily seen to satisfy:\[
\dot{x}=\eta X_{H}(x(t))\ \ \ \mbox{for all }t\in S^{1};
\]
\[
\int_{S^{1}}(H(x(t))-k)dt=0.
\]
Since $H$ is invariant under its Hamiltonian flow, the second equation
implies \[
H(x(t))-k=0\ \ \ \mbox{for all }t\in S^{1},
\]
that is, \[
x(S^{1})\subseteq\Sigma_{k}.
\]
Thus we can characterize $\mbox{Crit}(A_{H-k})$ by\begin{eqnarray*}
\mbox{Crit}(A_{H-k}) & = & \left\{ (x,\eta)\in\Lambda T^{*}M\times\mathbb{R}\,:\, x\in C^{\infty}(S^{1},T^{*}M)\right.\\
 &  & \left.\ \ \ \dot{x}(t)=\eta X_{H}(x(t)),\ x(S^{1})\subseteq\Sigma_{k}\right\} .
\end{eqnarray*}
The circle $S^{1}$ acts on $\Lambda T^{*}M$ via rotation:\[
r_{*}(x)(t):=x(r+t),\ \ \ r\in S^{1},\ x\in\Lambda T^{*}M.
\]
This action extends to an action on $\Lambda T^{*}M\times\mathbb{R}$
by ignoring the $\mathbb{R}$-factor. Since $H$ is autonomous, the
Rabinowitz action functional $A_{H-k}$ is invariant under this action.
In particular, its critical set $\mbox{Crit}(A_{H-k})$ is invariant. 

Thus the elements of $\mbox{Crit}(A_{H-k})$ come in two flavours.
Firstly, for each periodic orbit $y:\mathbb{R}/T\mathbb{Z}\rightarrow\Sigma_{k}$
of $X_{H}$ on $\Sigma_{k}$ with minimal period $T>0$, and for each
$m\in\mathbb{Z}\backslash\{0\}$, we have a copy of $S^{1}$: \[
\{(r_{*}(y)(mTt),mT)\,:\, r\in S^{1}\}
\]
contained in $\mbox{Crit}(A_{H-k})$. Secondly, $\mbox{Crit}(A_{H-k})$
contains the set $\{(x,0)\,:\, x\in\Sigma_{k}\}$, where a point in
$\Sigma_{k}$ should be interpreted as a constant loop in $\Lambda T^{*}M$.

Let us fix a $1$-periodic almost complex structure $J\in\mathcal{J}(\omega)$.
We denote by $\nabla A_{H-k}$ the $L^{2}$-gradient of $A_{H-k}$
with respect to the $L^{2}$-metric $\left\langle \left\langle \cdot,\cdot\right\rangle \right\rangle _{J}$:
\[
\nabla A_{H-k}(x,\eta)=\left(\begin{array}{c}
J(t,x)(\dot{x}-\eta X_{H}(x)\\
-\int_{S^{1}}(H(x(t))-k)dt
\end{array}\right).
\]

\subsection{Comparing the functionals $S_{L+k}$ and $A_{H-k}$}

$\ $\vspace{6 pt}

Let $H$ be as above and set $L:=L_{g}-\pi^{*}U$. The following lemma
outlines the relationship between the critical points of $S_{L+k}$
and $A_{H-k}$. The proof is identical to the analogous statements
in \cite[Section 5]{AbbondandoloSchwarz2009}, and will be omitted. 
\begin{lem}
\label{lem:relating critical points lemma}(Properties of $S_{L+k}$
and $A_{H-k}$)
\begin{enumerate}
\item Given $w=(q,T)\in\mbox{\emph{Crit}}(S_{L+k};\alpha)$, define \[
Z^{+}(w):=(x,T)\in\Lambda_{\alpha}T^{*}M\times\mathbb{R},\ \ \ \mbox{where }x(t):=\left(q(t),\dot{q}(t)\right),
\]
and define \[
Z^{-}(w):=(x^{-},-T)\in\Lambda_{-\alpha}T^{*}M\times\mathbb{R},\ \ \ \mbox{where }x^{-}(t):=x(-t).
\]
Then $Z^{+}(w)\in\mbox{\emph{Crit}}(A_{H-k};\alpha)$ and $Z^{-}(w)\in\mbox{\emph{Crit}}(A_{H-k};-\alpha)$,
and moreover the map\[
\mbox{\emph{Crit}}(S_{L+k})\times\{-1,1\}\rightarrow\left\{ (x,\eta)\in\mbox{\emph{Crit}}(A_{H-k})\,:\,\eta\ne0\right\} 
\]
given by \[
(w,\pm1)\mapsto Z^{\pm}(w)
\]
is a bijection, and \[
A_{H-k}(Z^{\pm}(w))=\pm S_{L+k}(w).
\]

\item Given any $(x,\eta)\in\Lambda T^{*}M\times\mathbb{R}$ with $\eta>0$,
if $q:=\pi\circ x$ then\begin{equation}
A_{H-k}(x,\eta)\leq S_{L+k}(q,\eta),\label{eq:bound plus}
\end{equation}
with equality if and only if $x=(q,\dot{q})$. If $x^{-}(t):=x(-t)$
then\begin{equation}
A_{H-k}(x^{-},-\eta)\geq-S_{L+k}(q,\eta)\label{eq:bound minus}
\end{equation}
with equality if and only if $x=(q,\dot{q})$. 
\item Let $w\in\mbox{\emph{Crit}}(S_{L+k})$. Then for all $(\xi,b)\in T_{Z^{+}(w)}(\Lambda T^{*}M\times\mathbb{R})$
it holds that \[
d_{Z^{+}(w)}^{2}A_{H-k}((\xi,b),(\xi,b))\leq d_{w}^{2}S_{L+k}((d\pi(\xi),b),(d\pi(\xi),b)),
\]
and similarly for all $(\xi,b)\in T_{Z^{-}(w)}(\Lambda T^{*}M\times\mathbb{R})$
it holds that \[
d_{Z^{-}(w)}^{2}A_{H-k}((\xi,b),(\xi,b))\geq-d_{w}^{2}S_{L+k}((d\pi(\xi)^{-},-b),(d\pi(\xi)^{-},-b)),
\]
where $d\pi(\xi)^{-}(t):=d\pi(\xi)(-t)$.
\item Given $w\in\mbox{\emph{Crit}}(S_{L+k})$, a pair $(\xi,b)$ lies in
the kernel of the Hessian of $A_{H-k}$ at $Z^{+}(w)$ if and only
if the pair $(d\pi(\xi),b)$ lies in the kernel of the Hessian of
$S_{L+k}$ at $w$, and similarly $(\xi,b)$ lies in the kernel of
the Hessian of $A_{H-k}$ at $Z^{-}(w)$ if and only if the pair $(d\pi(\xi)^{-},-b)$
lies in the kernel of the Hessian of $S_{L+k}$ at $w$.
\end{enumerate}
\end{lem}
As an immediate corollary of the preceding lemma and the definition
of $\mathcal{O}_{\textrm{reg}}$ (cf. Theorem \ref{thm:generic})
we obtain the following statement.
\begin{cor}
\label{pro:generic-1}If $(g,\sigma,U,k)\in\mathcal{O}_{\textrm{\emph{reg}}}$
and $H:=H_{g}+\pi^{*}U$ then every periodic orbit of $H$ lying in
$\Sigma_{k}$ is \textbf{\emph{strongly transversely non-degenerate}}.
In other words, if $y:\mathbb{R}/T\mathbb{Z}\rightarrow\Sigma_{k}$
is a periodic orbit of $X_{H}$ then the \textbf{\emph{nullity}}\emph{
}of $y$, $\nu(y)$ satisfies\[
\nu(y):=\dim\,\ker(d_{y(0)}\phi_{T}^{H}-\mathbb{1})=1.
\]
This implies that the Rabinowitz action functional $A_{H-k}$ is Morse-Bott,
and $\mbox{\emph{Crit}}(A_{H-k})$ consists of a copy of $\Sigma_{k}\times\{0\}$
together with a discrete union of circles.\end{cor}
\begin{proof}
It remains only to check that $A_{H-k}$ is Morse-Bott at the constant
orbits $(x,0)\in\Sigma_{k}\times\{0\}\subseteq\mbox{Crit}(A_{H-k})$.
A short computation tells us that $(\xi,b)$ lies in the kernel of
the Hessian of $A_{H-k}$ at $(x,0)$ if and only if \[
-\nabla_{t}\xi(t)+bX_{H}(x)=0;
\]
\[
\int_{S^{1}}d_{x}H(\xi(t))dt=0.
\]
Integrating the first equation and using the fact that $\xi$ is a
loop and $X_{H}(x)\ne0$ (as $k$ is a regular value of $H$ and $x\in\Sigma_{k}$),
we see that $b=0$. Thus $\xi(t)\equiv\xi(0)$ is constant, and the
second equation then says that $\xi(0)\in\ker\, d_{x}H=T_{x}\Sigma_{k}$.
Thus $A_{H-k}$ is Morse-Bott at the constant orbits.
\end{proof}

\subsection{\label{sub:Grading-the-Rabinowitz}Grading the Rabinowitz Floer complex}

$\ $\vspace{6 pt}

Fix $(g,\sigma,U,k)\in\mathcal{O}_{\textrm{reg}}$. Let $H=H_{g}+\pi^{*}U$.
Suppose $y:\mathbb{R}/T\mathbb{Z}\rightarrow\Sigma_{k}$ is a periodic
orbit of $X_{H}$. Our non-degeneracy assumption on $y$ implies that
there exists $\varepsilon>0$ together with a smooth (in $s$) family
$y_{s}:\mathbb{R}/T_{s}\mathbb{Z}\rightarrow T^{*}M$ for $s\in(-\varepsilon,\varepsilon)$
of $T_{s}$-periodic orbits of $X_{H}$ with $y_{0}=y$ and $H(y_{s})\equiv k+s$.
Such a family $(y_{s})$ is known as an \textbf{orbit cylinder }about
$y$, and the family $(y_{s})$ is unique. Actually the existence
of such an orbit cylinder requires only that $y$ has exactly two
\textbf{Floquet multipliers }equal to one (see for instance \cite[Proposition 4.2]{HoferZehnder1994}).
Our non-degeneracy assumption is strictly stronger than this: it implies
in addition that $\frac{\partial T_{s}}{\partial s}(0)\ne0$. Indeed,
let $N$ denote a hypersurface inside of $\Sigma_{k}$ which is transverse
to $y(\mathbb{R}/T\mathbb{Z})$ at the point $y(0)$, with $T_{y(0)}N$
equal to the symplectic orthogonal to the tangent space of the orbit
cylinder. Let $P_{y}:\mathcal{U}\rightarrow\mathcal{V}$ denote the
associated \textbf{Poincar\'e map}, where $\mathcal{U}$ and $\mathcal{V}$
are neighborhoods of $y(0)$. $P$ is a diffeomorphism that fixes
$y(0)$. Then there exists a unique symplectic splitting of $T_{y(0)}T^{*}M$
such that $d_{y(0)}\phi_{T}^{H}$ is given by\[
d_{y(0)}\phi_{T}^{H}=\left(\begin{array}{ccc}
1 & -\frac{\partial T_{s}}{\partial s}(0) & 0\\
0 & 1 & 0\\
0 & 0 & \begin{array}{ccc}
\\
 & d_{y(0)}P_{y}\\
\\
\end{array}
\end{array}\right).
\]
Here $\mathbb{1}-d_{y(0)}P_{z}$ is invertible. The assumption that
$\nu(y)=1$ therefore implies that $\frac{\partial T_{s}}{\partial s}(0)\ne0$.
Let us define \[
\chi(y):=\mbox{sign}\left(-\frac{\partial T_{s}}{\partial s}(0)\right).
\]
Now suppose $(x,\eta)\in\mbox{Crit}(A_{H-k})$ with $\eta>0$. Let
$y:\mathbb{R}/\eta\mathbb{Z}\rightarrow\Sigma_{k}$ be defined by
$y(t):=x(t/\eta)$. Define \[
\chi(x,\eta):=\chi(y).
\]
If $(x,\eta)\in\mbox{Crit}(A_{H-k})$ with $\eta<0$ define \[
\chi(x,\eta):=-\chi(x^{-},-\eta)
\]
where $x^{-}(t):=x(-t)$ (note that $(x^{-},-\eta)\in\mbox{Crit}(A_{H-k})$,
so this makes sense). 

Thus \[
\chi(q,T)=\chi(Z^{+}(q,T))=-\chi(Z^{-}(q,T)).
\]
for any $(q,T)\in\mbox{Crit}(S_{L+k})$, where $\chi(q,T)$ is defined
as in \eqref{eq:chi q T}.\newline 

We define a grading $\mu:\mbox{Crit}(A_{H-k})\rightarrow\mathbb{Z}$
on $\mbox{Crit}(A_{H-k})$ as follows.
\begin{defn}
Given $(x,\eta)\in\mbox{\emph{Crit}}(A_{H-k})$ with $\eta\ne0$ define
$y:\mathbb{R}/\left|\eta\right|\mathbb{Z}\rightarrow\Sigma_{k}$ by
$y(t):=x(t/\left|\eta\right|)$. Then $y$ is an $\left|\eta\right|$-periodic
orbit of $\mbox{\emph{sign}}(\eta)H$. Let us denote by $\mu_{\textrm{\emph{CZ}}}(y)$
the \textbf{\emph{Conley-Zehnder index}}\emph{ }of $y$. See \cite{RobbinSalamon1993}
for the definition of the Conley-Zehnder index in the degenerate case
that we are using (note however that our sign conventions match those
of \cite{AbbondandoloPortaluriSchwarz2008} not \cite{RobbinSalamon1993}).
Define \[
\mu(x,\eta):=\begin{cases}
\mu_{\textrm{\emph{CZ}}}(y)-\frac{1}{2}\chi(x,\eta) & \eta\ne0\\
-n+1 & \eta=0.
\end{cases}
\]

\end{defn}
We wish to compare $\mu(Z^{\pm}(q,T))$ with $i(q,T)$ for $(q,T)\in\mbox{Crit}(S_{L+k})$.
We will need an extension of the \textbf{Morse index theorem} of Duistermaat
\cite{Duistermaat1976} to the twisted symplectic form $\omega$:
\begin{thm}
\label{thm:morse index-1}Let $(q,T)\in\mbox{\emph{Crit}}(S_{L+k})$.
Let $y:\mathbb{R}/T\mathbb{Z}\rightarrow\Sigma$ be defined by $y(t):=Z^{+}(q,T)(t/T)$.
Then \[
\mu_{\textrm{\emph{CZ}}}(y)-\frac{1}{2}=i_{T}(q).
\]
\end{thm}
\begin{proof}
We deduce this from the equivalent statement for the standard symplectic
form $\omega_{0}$ (specifically, from \cite[Corollary 4.2]{AbbondandoloPortaluriSchwarz2008})
by arguing as follows: take a tubular neighborhood $W$ of $q(S^{1})$
in $M$. Since $H^{2}(W)=0$, $\sigma|_{W}=d\theta$ for some $\theta\in\Omega^{1}(W)$.
The flow $\phi_{t}^{H}|_{W}$ is conjugated to the flow $\psi_{t}^{H_{\theta}}:T^{*}W\rightarrow T^{*}W$,
where $H_{\theta}(q,p)=H(q,p-\theta_{q})$ and $\psi_{t}^{H_{\theta}}$
denotes the flow of the symplectic gradient of $H_{\theta}$ with
respect to the standard symplectic form $\omega_{0}$. Since both
the Maslov index and the Morse index are local invariants, the theorem
now follows directly from \cite[Corollary 4.2]{AbbondandoloPortaluriSchwarz2008}. \end{proof}
\begin{rem}
In \cite{Merry2011} we provide a direct proof of Theorem \ref{thm:morse index-1},
based on Weber's proof \cite[Theorem 1.3]{Weber2002} of the corresponding
statement for the standard symplectic form.
\end{rem}
The next corollary is an immediate consequence of Theorems \ref{thm:lag index}
and \ref{thm:morse index-1}, and the definition of the Conley-Zehnder
index.
\begin{cor}
\label{pro:morse indices and cz indices}Let $(q,T)\in\mbox{\emph{Crit}}(S_{L+k})$.
Then \[
\mu(Z^{\pm}(q,T))=\pm i(q,T).
\]

\end{cor}

\subsection{The moduli spaces of Rabinowitz Floer homology}

$\ $\vspace{6 pt}

Throughout this subsection assume $(g,\sigma,U,k)\in\mathcal{O}_{\textrm{reg}}$
is fixed (recall by assumption this means $k>c(g,\sigma,U)$, cf.
Definition \ref{def:the set Uk}), and put $H=H_{g}+\pi^{*}U$. Fix
$J\in\mathcal{J}(\omega)$. We are interested in maps $u:\mathbb{R}\rightarrow\Lambda T^{*}M\times\mathbb{R}$
that satisfy the \textbf{Rabinowitz Floer equation}:\begin{equation}
u'(s)+\nabla A_{H-k}(u(s))=0\label{eq:rfeq}
\end{equation}
together with the asymptotic conditions\[
\lim_{s\rightarrow\pm\infty}u(s)\in\mbox{Crit}(A_{H-k}).
\]
It is well known that any such map $u$ is smooth, and extends to
a map (also denoted by) $u:\overline{\mathbb{R}}\rightarrow C^{\infty}(S^{1},T^{*}M)\times\mathbb{R}$.
We shall often regard such a map $u$ as an element of $C^{\infty}(\mathbb{R}\times S^{1},T^{*}M)\times C^{\infty}(\mathbb{R},\mathbb{R})$.
If we write $u(s,t)=(x(s,t),\eta(s))$ then \eqref{eq:rfeq} implies
that $x$ and $\eta$ solve the coupled equations\[
x'+J(t,x)(\dot{x}-\eta X_{H}(x))=0;
\]
\[
\eta'-\int_{S^{1}}(H(x(t))-k)dt=0.
\]
Choose a Morse function $h:\mbox{Crit}(A_{H-k})\rightarrow\mathbb{R}$
and a Riemannian metric $g_{1}$ on $\mbox{Crit}(A_{H-k})$ such that
the negative gradient flow $\phi_{t}^{-\nabla h}$ of $-\nabla h$
is Morse-Smale. Denote by $\mbox{Crit}(h)\subseteq\mbox{Crit}(A_{H-k})$
the set of critical points of $h$. The Morse-Smale assumption implies
that for every pair $z_{-},z_{+}$ of critical points of $h$ the
unstable manifold $W^{u}(z_{-};-\nabla h)$ intersects the stable
manifold $W^{s}(z_{+};-\nabla h)$ transversely. Denote by $i_{h}(z):=\dim\, W^{u}(z;-\nabla h)$
the Morse index of a critical point $z\in\mbox{Crit}(h)$. We define
a new grading $\widehat{\mu}_{h}:\mbox{Crit}(h)\rightarrow\mathbb{Z}$
by putting\[
\widehat{\mu}_{h}(z):=\mu(z)+i_{h}(z).
\]
Suppose $z_{\pm}=(x_{\pm},\eta_{\pm})\in\mbox{Crit}(h)$ are critical
points of $h$. Denote by \[
\widetilde{\mathcal{M}}_{0}(z_{-},z_{+}):=W^{u}(z_{-};-\nabla h)\cap W^{s}(z_{+};-\nabla h).
\]
Let \[
\mathcal{M}_{0}(z_{-},z_{+}):=\widetilde{\mathcal{M}}_{0}(z_{-},z_{+})/\mathbb{R}
\]
 denote the quotient of $\widetilde{\mathcal{M}}_{0}(z_{-},z_{+})$
by the obvious free $\mathbb{R}$-action (if $z_{-}=z_{+}$, $\mathcal{M}_{0}(z_{-},z_{+})=\emptyset$).
Given $m\in\mathbb{N}$, let \[
\widetilde{\mathcal{M}}_{m}(z_{-},z_{+})
\]
denote the set of tuples of maps $\boldsymbol{u}=(u_{1},\dots,u_{m})$
such that each $u_{i}:\mathbb{R}\rightarrow C^{\infty}(S^{1},T^{*}M)\times\mathbb{R}$
satisfies the Rabinowitz Floer equation \eqref{eq:rfeq} and is \textbf{non-stationary}
(here a \textbf{stationary}\emph{ }solution is one that does not depend
on $s$) and such that\[
u_{1}(-\infty)\in W^{u}(z_{-};-\nabla h),\dots,u_{m}(\infty)\in W^{s}(z_{+};-\nabla h);
\]
\[
u_{i+1}(-\infty)\in\phi_{\mathbb{R}^{+}}^{-\nabla h}(u_{i}(\infty)).
\]
Note that if $m\geq1$ then $\widetilde{\mathcal{M}}_{m}(z_{-},z_{+})$
admits a free action of $\mathbb{R}^{m}$ via \[
(u_{1}(s),\dots,u_{m}(s))\mapsto(u_{1}(s+s_{1}),\dots,u_{m}(s+s_{m})),\ \ \ (s_{1},\dots,s_{m})\in\mathbb{R}^{m}.
\]
We denote by $\mathcal{M}_{m}(z_{-},z_{+})$ the quotient of $\widetilde{\mathcal{M}}_{m}(z_{-},z_{+})$
by this action. Put \[
\mathcal{M}(z_{-},z_{+}):=\bigcup_{m\in\mathbb{N}\cup\{0\}}\mathcal{M}_{m}(z_{-},z_{+}).
\]

Since $A_{H-k}$ is strictly decreasing on non-stationary solutions
of the Rabinowitz Floer equation, if $z_{-}$ and $z_{+}$ belong
to the same connected component of $\mbox{Crit}(A_{H-k})$ then $\mathcal{M}_{m}(z_{-},z_{+})=\emptyset$
for all $m\geq1$, and if $\mathcal{M}_{m}(z_{-},z_{+})\ne\emptyset$
for some $m\geq1$, then $A_{H-k}(z_{-})>A_{H-k}(z_{+})$ and $\mathcal{M}_{0}(z_{-},z_{+})=\emptyset$.
\newline

The central result we need to construct the Rabinowitz Floer complex
is the following:
\begin{thm}
\label{thm:main theorem-1}There exists $\varepsilon_{1}>0$ such
that if $J\in\mathcal{J}(\omega)\cap B_{\varepsilon_{1}}(J_{g})$
is a generically chosen almost complex structure and $g_{1}$ is a
generically chosen Morse-Smale metric for $h$ then the moduli spaces
$\mathcal{M}(z_{-},z_{+})$ are all finite dimensional smooth manifolds,
and their components of dimension zero are compact. Moreover we have
\begin{equation}
\dim\,\mathcal{M}(z_{-},z_{+})=\widehat{\mu}_{h}(z_{-})-\widehat{\mu}_{h}(z_{+})-1.\label{eq:dimension formula}
\end{equation}

\end{thm}
The proof of the theorem has four ingredients:
\begin{enumerate}
\item Exhibit $\mathcal{M}(z_{-},z_{+})$ as the zero set of a certain section
of a Banach bundle. 
\item Show that the linearization of this operator is Fredholm, and compute
its index.
\item Show that for generic $J,g_{1}$ the linearization is surjective.
\item Exhibit uniform $C_{\textrm{loc}}^{\infty}$ bounds for gradient flow
lines.
\end{enumerate}
We refer to one of the many references (perhaps the two most relevant
are \cite[Appendix A]{Frauenfelder2004} and \cite[Section 3]{AbbondandoloSchwarz2006})
as to why solving these four problems does indeed lead to a proof
of the theorem. Problem (1) was solved in \cite[Appendix A]{Frauenfelder2004}.
Problem (2) was solved for \textbf{defining Hamiltonians }and \textbf{restricted
contact type }hypersurfaces in \cite[Section 4]{CieliebakFrauenfelder2009}.
In our situation there is an additional complication in computing
the indices (stemming from the correction term $-\frac{1}{2}\chi(z)$).
Full details of the computation of the index are contained in \cite{MerryPaternain2010}.
Alternatively one could probably use the methods of \cite[Section 3.2]{BourgeoisOancea2009a}. 

Problem (3) can be solved using the methods in \cite{FloerHoferSalamon1996}
combined with the Morse-Bott formalism of \cite[Theorem A.14]{Frauenfelder2004}.
Problem (4) was solved for Hamiltonians that are constant outside
a compact set in \cite[Section 3]{CieliebakFrauenfelder2009} and
extended to Hamiltonians that are linear at infinity \cite[Section 5]{CieliebakFrauenfelderOancea2010}
and then Hamiltonians which grow quadratically and radially at infinity
\cite[Section 2]{AbbondandoloSchwarz2009}. None of these are applicable
for the Hamiltonians $H_{g}+\pi^{*}U$ that we consider, and hence
we will give a complete proof of this below. Our methods are essentially
those of \cite{AbbondandoloSchwarz2006}. Since $\omega|_{\pi_{2}(M)}=0$
and $c_{1}(T^{*}M,\omega)=0$, in order to get $C_{\textrm{loc}}^{\infty}$
bounds on gradient flow lines of the Rabinowitz Floer equation it
is sufficient to obtain $L^{\infty}$ bounds (in short, this is because
the so-called `bubbling' phenomenon cannot occur). Obtaining these
$L^{\infty}$ estimates is the subject of Subsection \ref{sub:The--estimates}
below.
\begin{rem}
It is perhaps useful to explain exactly where our various hypotheses
are used. The fact that $(g,\sigma,U,k)\in\mathcal{O}$ (i.e. $k>c(g,\sigma,U)$)
is used in order in order to obtain $L^{\infty}$ bounds on the $\eta$-component
of gradient flows lines $u\in\mathcal{M}(z_{-},z_{+})$. The bound
on the $x$-component requires two assumptions: firstly that the $\eta$-component
is uniformly bounded, and secondly that $J\in B_{\varepsilon_{1}}(J_{g})$.
Finally, the assumption $(g,\sigma,U,k)\in\mathcal{O}_{\textrm{\emph{reg}}}$
is used in order to compute the index of the operator defining the
moduli space $\mathcal{M}(z_{-},z_{+})$ - recall that our grading
$\mu$ explicitly used the existence of an orbit cylinder, which need
not exist if only $(g,\sigma,U,k)\in\mathcal{O}$. 
\end{rem}

\begin{rem}
\label{rem:shrinking sigma}The constant $\varepsilon_{1}>0$ appearing
in the statement of Theorem \ref{thm:main theorem-1} is a universal
constant (cf. Theorem \ref{thm:calderon } below). In order for the
statement of Theorem \ref{thm:main theorem-1} not to be completely
vacuous one of course needs to know that such almost complex structures
exist. This can be guaranteed by assuming $\left\Vert \sigma\right\Vert $
is sufficiently small. Indeed, suppose $\sigma$ satisfies\begin{equation}
\left\Vert \sigma\right\Vert _{\infty}\leq\frac{\varepsilon_{1}}{2\varepsilon_{0}}.\label{eq:epsilon 2}
\end{equation}
Then by \eqref{eq:ball not empty} we have \[
B_{\varepsilon_{1}/2}(J_{\sigma})\subseteq B_{\varepsilon_{1}}(J_{g}).
\]

\end{rem}

\subsection{Constructing the chain complex}

$\ $\vspace{6 pt}

Deferring the proof of Problem (4), we first explain the construction
of Rabinowitz Floer chain complex. Assume that the hypotheses of Theorem
\ref{thm:main theorem-1} are satisfied. Denote by $RF(A_{H-k},h)$
the $\mathbb{Z}_{2}$-vector space generated by all formal sums \[
\sum_{z\in V}z,
\]
where $V\subseteq\mbox{Crit}(h)$ is a (possibly infinite) subset
of $\mbox{Crit}(h)$ satisfying the \textbf{Novikov finiteness condition}\emph{
}that for all $a\in\mathbb{R}$ one has \[
\#\left\{ z\in V\,:\, A_{H-k}(z)<a\right\} <\infty.
\]
Let us write $\mbox{Crit}_{i}(h)\subseteq\mbox{Crit}(h)$ for the
set of critical points $z$ of $h$ with $\widehat{\mu}_{h}(z)=i$.
The vector space $RF(A_{H-k},h)$ is given a $\mathbb{Z}$-grading
by the index $\widehat{\mu}_{h}$: an element $\sum_{z\in V}z\in RF(A_{H-k},h)$
belongs to $RF_{i}(A_{H-k},h)$ if $V\subseteq\mbox{Crit}_{i}(h)$.
Similarly, given an interval $(a,b)\subseteq\overline{\mathbb{R}}$,
denote by $RF^{(a,b)}(A_{H-k},h)$ the $\mathbb{Z}_{2}$-vector space
of all formal sums \[
\sum_{z\in V}z,
\]
where $V\subseteq\mbox{Crit}^{(a,b)}(h)$ is a (possibly infinite)
subset of $\mbox{Crit}^{(a,b)}(h)$ satisfying the finiteness condition
above (note that if $a$ and $b$ are finite then such a set $V$
is necessarily finite and the Novikov finiteness condition is automatic).

If $z_{\pm}\in\mbox{Crit}(h)$ satisfy $\widehat{\mu}_{h}(z_{-})-\widehat{\mu}_{h}(z_{+})=1$
then Theorem \ref{thm:main theorem-1} tells us that $\mathcal{M}(z_{-},z_{+})$
is a finite set. We can therefore define $n_{\textrm{Rab}}(z_{-},z_{+})$
by \[
n_{\textrm{Rab}}(z_{-},z_{+}):=\#\mathcal{M}(z_{-},z_{+}),\ \ \ \mbox{taken mod }2.
\]
Then we define\[
\partial^{\textrm{Rab}}=\partial^{\textrm{Rab}}(J,g_{1}):RF_{i}(A_{H-k},h)\rightarrow RF_{i-1}(A_{H-k},h)
\]
 by \[
\partial^{\textrm{Rab}}z:=\sum_{z'\in\textrm{Crit}_{i-1}(h)}n_{\textrm{Rab}}(z,z')z',
\]
and extending by linearity. A standard gluing argument tells us that
$\partial^{\textrm{Rab}}\circ\partial^{\textrm{Rab}}=0$, and therefore
we conclude that $\{RF_{*}(A_{H-k},h),\partial^{\textrm{Rab}}(J,g_{1})\}$
is a chain complex of Abelian groups. The boundary map $\partial^{\textrm{Rab}}$
respects the $\mathbb{R}$-filtration determined by $A_{H-k}$: if
$(a,b)\subseteq\overline{\mathbb{R}}$ then \[
\partial^{\textrm{Rab}}\left(RF_{i}^{(a,b)}(A_{H-k},h)\right)\subseteq RF_{i-1}^{(a,b)}(A_{H-k},h),
\]
and so $\{RF_{*}^{(a,b)}(A_{H-k},h),\partial^{\textrm{Rab}}(J,g_{1})\}$
is a subcomplex. Finally it is clear that $\partial^{\textrm{Rab}}$
also respects the splitting $\Lambda T^{*}M\oplus\mathbb{R}=\bigoplus_{\alpha\in[S^{1},M]}\Lambda_{\alpha}T^{*}M\times\mathbb{R}$:
if $RF_{*}(A_{H-k},h;\alpha)$ denotes the subspace of $RF_{*}(A_{H-k},h)$
generated by the elements of $\mbox{Crit}(h)\cap\mbox{Crit}(A_{H-k};\alpha)$
then $RF_{*}(A_{H-k},h;\alpha)$ is a subcomplex.

We write $RFH_{*}(A_{H-k})$ for the homology of $\{RF_{*}(A_{H-k},a),\partial^{\textrm{Rab}}(J,g_{1})\}$
and call it the \textbf{Rabinowitz Floer homology}\emph{ }of $A_{H-k}$.
Similarly we write $RFH_{*}(A_{H-k};\alpha)$ (resp. $RFH_{*}^{(a,b)}(A_{H-k})$)
for the homology of the subcomplex $RF_{*}(A_{H-k},h;\alpha)$ (resp.
$RF_{*}^{(a,b)}(A_{H-k},h)$). Standard arguments show that $RFH_{*}(A_{H-k})$
is independent of the data $(h,J,g_{1})$. 
\begin{rem}
\label{rem:inv under g}In fact, if $(g_{s},\sigma_{s},U_{s},k_{s})_{s\in[0,1]}\subseteq\mathcal{O}$
is a smooth family that satisfies $(g_{s},\sigma_{s},U_{s},k_{s})\in\mathcal{O}_{\textrm{\emph{reg}}}$
for generic $s\in[0,1]$ and in particular for $s=0,1$ then if $H_{s}(q,p):=\frac{1}{2}\left|p\right|_{g_{s}}^{2}+U_{s}(q)$
and $\omega_{s}:=\omega_{0}+\pi^{*}\sigma_{s}$ then $RFH_{*}(A_{H_{0}-k_{0}};\omega_{0})\cong RFH_{*}(A_{H_{1}-k_{1}};\omega_{1})$.
One can prove this directly using the methods of \cite[Section 1.8]{AbbondandoloSchwarz2006}
and \cite{BaeFrauenfelder2010}. However we can deduce this indirectly
via Theorem 1.1.(b) in \cite{CieliebakFrauenfelderPaternain2010}
and Theorem 1.4 in \cite{BaeFrauenfelder2010}, by making use of Proposition
\ref{quadr is const} below, which states that the Rabinowitz Floer
homology $RFH_{*}(A_{H-k})$ is the same as the Rabinowitz Floer homology
$RFH_{*}(\Sigma_{k},T^{*}M)$ from \cite{CieliebakFrauenfelderPaternain2010}. 

As a consequence we are free to define the Rabinowitz Floer homology
$RFH_{*}(A_{H-k})$ for the Hamiltonian $H=H_{g}+\pi^{*}U$ if only
$(g,\sigma,U,k)\in\mathcal{O}$ (rather than $(g,\sigma,U,k)\in\mathcal{O}_{\textrm{\emph{reg}}}$).
Indeed, by Theorem \ref{thm:generic} we can find a metric $g'$ lying
arbitrarily close to $g$ such that $(g,\sigma,U,k)\in\mathcal{O}_{\textrm{\emph{reg}}}$.
Set $H':=H_{g'}+\pi^{*}U$ and \textbf{\emph{define }}$RFH_{*}(A_{H-k}):=RFH_{*}(A_{H'-k}).$
This is well defined, as if $g''$ is another such metric and $H'':=H_{g''}+\pi^{*}U$
then the previous paragraph implies $RFH_{*}(A_{H'-k})\cong RFH_{*}(A_{H''-k})$.
\end{rem}

\subsection{\label{sub:The--estimates}The $L^{\infty}$ estimates}

$\ $\vspace{6 pt}

In this subsection we prove the two theorems on $L^{\infty}$ estimates
for solutions of the Rabinowitz Floer equation alluded to above, as
well as a third $L^{\infty}$ estimate for gradient flow lines defined
on half-cylinders that will be needed in the next section. The first
result we state is an extension of part of \cite[Theorem 3.1]{CieliebakFrauenfelder2009},
which obtains uniform $L^{\infty}$ bounds for the $\eta$-component
of flow lines $u=(x,\eta)\in C^{\infty}(\mathbb{R}\times S^{1},T^{*}M)\times C^{\infty}(\mathbb{R},\mathbb{R})$
satisfying the Rabinowitz Floer equation and having bounded $A_{H-k}$-action.
This result (for contractible loops only) was stated without proof
in \cite[Section 7]{CieliebakFrauenfelderPaternain2010}. 
\begin{thm}
\label{thm:bounding the lagrange multiplier}Let $(g,\sigma,U,k)\in\mathcal{O}$
and put $H=H_{g}+\pi^{*}U$. Pick $J\in\mathcal{J}(\omega)$ and $\alpha\in[S^{1},M]$,
and fix $-\infty<a<b<\infty$. There exists a constant $C_{0}>0$
such that if $u=(x,\eta)\in C^{\infty}(\mathbb{R}\times S^{1},T^{*}M)\times C^{\infty}(\mathbb{R},\mathbb{R})$
is any map that satisfies the Rabinowitz Floer equation \eqref{eq:rfeq}
and has action bounds \[
A_{H-k}(u(\mathbb{R}))\subseteq[a,b]
\]
and satisfies\[
x(\mathbb{R},\cdot)\in\Lambda_{\alpha}T^{*}M
\]
then \[
\left\Vert \eta\right\Vert _{L^{\infty}(\mathbb{R})}\leq C_{0}.
\]
\end{thm}
\begin{rem}
We emphasize that the following proof uses only that $\Sigma_{k}:=H^{-1}(k)$
is of \textbf{\emph{virtual restricted contact type}}\textbf{ }(see
\cite[p1767]{CieliebakFrauenfelderPaternain2010} for the definition)
for $k>c(g,\sigma,U)$; it makes \textbf{\emph{no}} assumptions on
the behaviour of the Hamiltonian $H$ at infinity. In other words,
the proof would go through if instead of $H$ we used any other Hamiltonian
$K\in C^{\infty}(T^{*}M,\mathbb{R})$ with the property that $X_{K}|_{\Sigma_{k}}=fX_{H}|_{\Sigma_{k}}$
for some smooth function $f\in C^{\infty}(\Sigma_{k},\mathbb{R}^{+})$.\end{rem}
\begin{proof}
\emph{(of Theorem \ref{thm:bounding the lagrange multiplier})}

The proof is a slight modification of the arguments of \cite[Section 3]{CieliebakFrauenfelder2009}.
Let $\widetilde{H}:T^{*}\widetilde{M}\rightarrow\mathbb{R}$ denote
the lift of $H$ to $\widetilde{\pi}:T^{*}\widetilde{M}\rightarrow\widetilde{M}$.
Let $\widetilde{\omega}:=\widetilde{\omega}_{0}+\widetilde{\pi}^{*}\widetilde{\sigma}$,
where $\widetilde{\omega}_{0}=d\widetilde{\lambda}_{0}$ is the canonical
symplectic form on $T^{*}\widetilde{M}$. Since $k>c(g,\sigma,U)$,
by \cite[Lemma 5.1]{CieliebakFrauenfelderPaternain2010} there exists
a primitive $\theta$ of $\widetilde{\sigma}$ and $\delta>0$ such
that \begin{equation}
\widetilde{\lambda}(X_{\widetilde{H}}(x))\geq2\delta\ \ \ \mbox{for all }x\in\widetilde{H}^{-1}([k-\delta,k+\delta]).\label{eq:condition star}
\end{equation}
Here \[
\widetilde{\lambda}:=\widetilde{\lambda}_{0}+\widetilde{\pi}^{*}\theta,
\]
and $X_{\widetilde{H}}$ is the symplectic gradient of the lifted
function $\widetilde{H}$ with respect to the symplectic form $\widetilde{\omega}=d\widetilde{\lambda}$.
Observe that it follows from \eqref{eq:relating S and F} that for
any $x\in\Lambda_{\alpha}T^{*}M$ and any lift $\widetilde{x}:[0,1]\rightarrow T^{*}\widetilde{M}$
we have \begin{equation}
\int_{C}\bar{x}^{*}\omega=\int_{0}^{1}\widetilde{x}^{*}\widetilde{\lambda}+I(\alpha,\theta).\label{eq:lambda theta}
\end{equation}

The first part of the proof is the following statement: there exists
a constant $\rho_{0}=\rho(\delta)>0$ such that: \begin{equation}
\left\Vert \nabla A_{H-k}(x,\eta)\right\Vert _{J}\leq\rho_{0}\ \ \ \Rightarrow\ \ \ x(S^{1})\subseteq H^{-1}([k-\delta,k+\delta])\label{eq:bdelta-1}
\end{equation}
(where $\delta>0$ is the constant from \eqref{eq:condition star}).
This part of the proof is identical to \cite[Proposition 3.2, Step 2]{CieliebakFrauenfelder2009},
and hence is omitted. 

Next we show that there exists a constant $D<\infty$ such that if
$(x,\eta)\in\Lambda_{\alpha}T^{*}M\times\mathbb{R}$ is any loop that
satisfies \[
x(S^{1})\subseteq H^{-1}([k-\delta,k+\delta]),
\]
(where $\delta>0$ is the constant from \eqref{eq:condition star})
then \begin{equation}
\left|\eta\right|<\frac{1}{\delta}\left|A_{H-k}(x,\eta)\right|+\frac{D}{\delta}\left\Vert \nabla A_{H-k}(x,\eta)\right\Vert _{J}+\frac{1}{\delta}\left|I(\alpha,\theta)\right|.\label{eq:first equation to prove}
\end{equation}
Indeed, set \[
D:=\left\Vert \widetilde{\lambda}|_{\widetilde{H}^{-1}([k-\delta,k+\delta])}\right\Vert _{\infty},
\]
and compute using using \eqref{eq:lambda theta}: \begin{align*}
\left|A_{H-k}(x,\eta)\right| & =\left|\int_{C}\bar{x}^{*}\omega-\eta\int_{S^{1}}(H(x(t))-k)dt\right|\\
 & \geq\left|\int_{0}^{1}\widetilde{\lambda}(\dot{\widetilde{x}})dt\right|-\left|I(\alpha,\theta)\right|-\left|\eta\right|\left|\int_{\mathbb{T}}(H(x(t))-k)dt\right|\\
 & \geq\left|\int_{0}^{1}\widetilde{\lambda}(\eta X_{\widetilde{H}}(\widetilde{x}))dt\right|-\left|\int_{0}^{1}\widetilde{\lambda}(\dot{\widetilde{x}}-\eta X_{\widetilde{H}}(\widetilde{x}))dt\right|-\left|I(\alpha,\theta)\right|-\left|\eta\right|\delta.\\
 & \geq\left|\eta\right|(2\delta-\delta)-D\int_{S^{1}}\left|\dot{x}-\eta X_{H}(x)\right|dt-\left|I(\alpha,\theta)\right|\\
 & \geq\left|\eta\right|\delta-D\left\Vert \nabla A_{H-k}(x,\eta)\right\Vert _{J}-\left|I(\alpha,\theta)\right|.
\end{align*}
This proves \eqref{eq:first equation to prove}. Combining \eqref{eq:bdelta-1}
and \eqref{eq:first equation to prove} we see that if \[
\rho_{1}:=\frac{1}{\delta}\max\{1,D\rho_{0}+\left|I(\alpha,\theta)\right|\}
\]
then the following implication holds: for any $(x,\eta)\in\Lambda_{\alpha}T^{*}M\times\mathbb{R}$,
\begin{equation}
\left\Vert \nabla A_{H-k}(x,\eta)\right\Vert _{J}\leq\rho_{0}\ \ \ \Rightarrow\ \ \ \left|\eta\right|\leq\rho_{1}\left(A_{H-k}(x,\eta)+1\right).\label{eq:cf prop 32}
\end{equation}
We can now prove the theorem. Let $u=(x,\eta)\in C^{\infty}(\mathbb{R}\times S^{1},T^{*}M)\times C^{\infty}(\mathbb{R},\mathbb{R})$
satisfy the hypotheses of the theorem. Given $s\in\mathbb{R}$ let
\begin{equation}
\tau(s):=\inf\left\{ r\geq0\,:\,\left\Vert \nabla A_{H-k}(u(s+r,\cdot))\right\Vert _{J}\leq\rho_{0}\right\} .\label{eq:def of tau of s}
\end{equation}
Then for any $s\in\mathbb{R}$ we have: \begin{align*}
b-a & \geq\int_{-\infty}^{\infty}\left\Vert \nabla A_{H-k}(u(r,\cdot))\right\Vert _{J}^{2}dr\\
 & \geq\int_{s}^{s+\tau(s)}\left\Vert \nabla A_{H-k}(u(r,\cdot))\right\Vert _{J}^{2}dr\\
 & \geq\tau(s)\rho_{0}^{2},
\end{align*}
and hence \[
\tau(s)\leq\frac{b-a}{\rho_{0}^{2}}.
\]
Thus given any $s\in\mathbb{R}$ we have\begin{align*}
\left|\eta(s)\right| & =\left|\eta(s+\tau(s))-\int_{s}^{s+\tau(s)}\eta'(r)dt\right|\\
 & \leq\rho_{1}\left(\left|A_{H-k}(u(s+\tau(s),\cdot))\right|+1\right)+\int_{s}^{s+\tau(s)}\left|\eta'(r)\right|dr\\
 & \leq\rho_{1}(\max\{|a|,|b|\}+1)+\left(\tau(s)\int_{s}^{s+\tau(s)}\left|\eta'(r)\right|^{2}ds\right)^{1/2}\\
 & \leq\rho_{1}(\max\{|a|,|b|\}+1)+\left(\frac{b-a}{\rho_{0}^{2}}\int_{s}^{s+\tau(s)}\left\Vert u'(r,\cdot)\right\Vert _{J}^{2}ds\right)^{1/2}\\
 & \leq\rho_{1}(\max\{|a|,|b|\}+1)+\frac{b-a}{\rho_{0}}.
\end{align*}
Thus the theorem follows with \[
C_{0}:=\rho_{1}(\max\{|a|,|b|\}+1)+\frac{b-a}{\rho_{0}}.
\]

\end{proof}
In the next result we are interested in obtaining bounds on the loop
component $x$ of a flow line $u$. The proof uses the same idea as
\cite[Theorem 1.14, Theorem 1.22]{AbbondandoloSchwarz2006}, and is
based upon isometrically embedding $(M,g)$ into Euclidean space,
and combining Calderon-Zygmund estimates for the Cauchy-Riemann operator
with certain interpolation inequalities. In the course of the proof
we will need the following statement, which is a consequence of the
Calderon-Zygmund inequalities. Let \[
W_{V}^{1,r}(\mathbb{R}\times S^{1},\mathbb{R}^{2d})=W_{0}^{1,r}(\mathbb{R}\times S^{1},\mathbb{R}^{d})\times W^{1,3}(\mathbb{R}\times S^{1},\mathbb{R}^{d})
\]
denote the Sobolev space of $\mathbb{R}^{2d}$-valued maps taking
values in the vertical Lagrangian subspace $V:=(0)\times\mathbb{R}^{d}\subseteq\mathbb{R}^{2d}$
on the boundary.
\begin{thm}
\label{thm:calderon }Let $J_{0}$ denote the standard complex structure
on $\mathbb{R}^{2d}$ given by \[
J_{0}=\left(\begin{array}{cc}
0 & -\mathbb{1}\\
\mathbb{1} & 0
\end{array}\right).
\]
Consider the \textbf{\emph{Cauchy-Riemann operator}} \[
\partial_{s}+J_{0}\partial_{t}:W_{V}^{1,3}(\mathbb{R}\times S^{1},\mathbb{R}^{2d})\rightarrow L^{3}(\mathbb{R}\times S^{1},\mathbb{R}^{2d}).
\]
Then there exists a constant $\varepsilon_{1}>0$ such that for any
$v\in W_{V}^{1,3}(\mathbb{R}\times S^{1},\mathbb{R}^{2d})$ it holds
that \[
\left\Vert \nabla v\right\Vert _{L^{3}(\mathbb{R}\times S^{1})}\leq\frac{1}{2\varepsilon_{1}}\left\Vert (\partial_{s}+J_{0}\partial_{t})v\right\Vert _{L^{3}(\mathbb{R}\times S^{1})}.
\]

\end{thm}
We now prove:
\begin{thm}
\label{thm:l infinity}Fix $(g,\sigma,U,k)\in\mathcal{O}$. Suppose
$J\in\mathcal{J}(\omega)\cap B_{\varepsilon_{1}}(J_{g})$ (where $\varepsilon_{1}>0$
is as in Theorem \ref{thm:calderon }), $\alpha\in[S^{1},M]$ and
$-\infty<a<b<\infty$. Put $H=H_{g}+\pi^{*}U$. Assume there exists
a constant $C_{0}>0$ such that if $u=(x,\eta)\in C^{\infty}(\mathbb{R}\times S^{1},T^{*}M)\times C^{\infty}(\mathbb{R},\mathbb{R})$
is any map that satisfies the Rabinowitz Floer equation \eqref{eq:rfeq}
and has action bounds \[
A_{H-k}(u(\mathbb{R}))\subseteq[a,b]
\]
and satisfies\[
x(\mathbb{R},\cdot)\in\Lambda_{\alpha}T^{*}M
\]
then \[
\left\Vert \eta\right\Vert _{L^{\infty}(\mathbb{R})}\leq C_{0}.
\]
Then there exists another constant $C_{1}>0$ such that for any such
map $u=(x,\eta)$ it also holds that \[
\left\Vert x\right\Vert _{L^{\infty}(\mathbb{R})}<C_{1}.
\]

\end{thm}
In the proof below we will repeatedly use the fact there exists a
constant $b_{0}>0$ such that\begin{equation}
\left|X_{H}(q,p)\right|\leq b_{0}\left(1+\left|p\right|^{2}\right)\ \ \ \mbox{for all }(q,p)\in T^{*}M.\label{eq:b0}
\end{equation}

\begin{proof}
\emph{(of Theorem \ref{thm:l infinity})}

We begin by choosing an isometric embedding of $i:(M,g)\rightarrow(\mathbb{R}^{d},g_{0})$,
where $g_{0}$ is the Euclidean inner product. Such an embedding exists
by Nash's theorem. It induces an embedding (also denoted by) $i:T^{*}M\rightarrow\mathbb{R}^{2d}$
which is actually a \textbf{unitary embedding} (with respect to the
\textbf{standard }symplectic form), that is,\[
i^{*}\omega_{0}=\omega_{0},\ \ \ i_{*}J_{0}=J_{g},
\]
where \[
J_{0}=\left(\begin{array}{cc}
0 & -\mathbb{1}\\
\mathbb{1} & 0
\end{array}\right)
\]
is the standard almost complex structure on $\mathbb{R}^{2d}$. Thus
under this embedding the metric almost complex structure $J_{g}$
\eqref{eq:metric acs} is simply the restriction of the canonical
almost complex structure $J_{0}$ to $T^{*}M$, and hence the assumption
that $J\in\mathcal{J}(\omega)\cap B_{\varepsilon_{1}}(J_{g})$ corresponds
to $J\in\mathcal{J}(\omega)\cap B_{\varepsilon_{1}}(J_{0})$. We will
use this embedding to define the various $L^{r}$ and $W^{1,r}$ spaces
that come up in the proof below. 

The proof of the theorem is in two steps.\newline

\textbf{Step 1.}\newline

We show that there exists a constant $K>0$ such that for any map
$u=(x,\eta)$ satisfying the hypotheses of the theorem, and any finite
interval $I\subseteq\mathbb{R}$, writing $x=(q,p)$ it holds that\begin{equation}
\left\Vert p\right\Vert _{L^{2}(I\times S^{1})}\leq K\left|I\right|^{1/2},\ \ \ \left\Vert \nabla p\right\Vert _{L^{2}(I\times S^{1})}\leq K\left(1+\left|I\right|^{1/2}\right).\label{eq:step 1 estiamte}
\end{equation}
This part of the proof closely follows \cite[Lemma 1.12]{AbbondandoloSchwarz2006},
and heavily uses the fact that our Hamiltonian $H$ is \textbf{quadratic}.
This step does not use the fact that $J\in B_{\varepsilon_{1}}(J_{0})$. 

We first note that there exists a constant $b_{1}>0$ such that for
any map $u=(x,\eta)$ satisfying the hypotheses of the theorem, \[
\left\Vert x'\right\Vert _{L^{2}(\mathbb{R}\times S^{1})}\leq b_{1},\ \ \ \left\Vert \eta'\right\Vert _{L^{2}(\mathbb{R})}\leq b_{1}.
\]
Indeed, if $s_{0}<s_{1}$ then \begin{align*}
\left\Vert x'\right\Vert _{L^{2}((s_{0},s_{1})\times S^{1})}^{2} & =\int_{s_{0}}^{s_{1}}\int_{S^{1}}\left|x'\right|^{2}dtds\\
 & \leq\left\Vert J^{-1}\right\Vert _{\infty}^{2}\int_{s_{0}}^{s_{1}}\int_{S^{1}}\left\Vert u'\right\Vert _{J}^{2}dtds\\
 & \leq\left\Vert J\right\Vert _{\infty}^{2}(b-a).
\end{align*}
Exactly the same computation holds for $\left\Vert \eta'\right\Vert _{L^{2}(\mathbb{R})}$,
and hence we may take \begin{equation}
b_{1}:=\left\Vert J\right\Vert _{\infty}\sqrt{b-a}.\label{eq:b1}
\end{equation}
We next claim that there exists a constant $b_{2}>0$ such that for
any finite interval $I\subseteq\mathbb{R}$ and for any map $u=(x,\eta)$
satisfying the hypotheses of the theorem, if we write $x=(q,p)$,
then\begin{equation}
\left\Vert p\right\Vert _{L^{2}(I\times S^{1})}^{2}\leq b_{2}\max\left\{ \left|I\right|,|I|^{1/2}\right\} .\label{eq:b2}
\end{equation}
Indeed, \begin{align}
\eta'(s) & =\int_{S^{1}}(H(x(s,t))-k)dt\nonumber \\
 & \geq\int_{S^{1}}\frac{1}{2}\left|p(s,t)\right|^{2}dt-(\left\Vert U\right\Vert _{\infty}+k).\label{eq:eta bounding p}
\end{align}
Hence \begin{align*}
\frac{1}{2}\left\Vert p\right\Vert _{L^{2}((s_{0},s_{1})\times S^{1})}^{2} & \leq\left\Vert \eta'\right\Vert _{L^{1}((s_{0},s_{1}))}+(\left\Vert U\right\Vert _{\infty}+k)(s_{1}-s_{0})\\
 & \leq\sqrt{s_{1}-s_{0}}\left\Vert \eta'\right\Vert _{L^{2}((s_{0},s_{1}))}+(\left\Vert U\right\Vert _{\infty}+k)(s_{1}-s_{0})\\
 & \leq\sqrt{s_{1}-s_{0}}b_{1}+(\left\Vert U\right\Vert _{\infty}+k)(s_{1}-s_{0}).
\end{align*}
Then \eqref{eq:b2} follows with \begin{equation}
b_{2}=2b_{1}+2(\left\Vert U\right\Vert _{\infty}+k).\label{eq:b2-1}
\end{equation}
Next we prove that for any map $u=(x,\eta)$ satisfying the hypotheses
of the theorem, and every $0<\varepsilon\leq1$, the closed subsets
\begin{equation}
S_{\varepsilon}(u):=\left\{ s\in\mathbb{R}\,:\,\left\Vert p(s,\cdot)\right\Vert _{L^{2}(S^{1})}^{2}\leq\frac{b_{2}}{\sqrt{\varepsilon}}\right\} ;\label{eq:sepsilon set}
\end{equation}
\begin{equation}
S_{\varepsilon}'(u):=\left\{ s\in\mathbb{R}\,:\,\left\Vert x'(s,\cdot)\right\Vert _{L^{2}(S^{1})}^{2}\leq\frac{b_{1}}{\sqrt{\varepsilon}}\right\} \label{eq:s dash}
\end{equation}
are \textbf{$\varepsilon$-dense}, that is, they\textbf{ }have non-empty
intersection with any interval of length $\geq\varepsilon$. Indeed,
for every $s_{0}\in\mathbb{R}$ we have that if $0<\varepsilon\leq1$
then\begin{align*}
\min_{s\in[s_{0},s_{0}+\varepsilon]}\left\Vert p(s,\cdot.)\right\Vert _{L^{2}(S^{1})}^{2} & \leq\frac{1}{\varepsilon}\int_{s_{0}}^{s_{0}+\varepsilon}\left\Vert p(s,\cdot)\right\Vert _{L^{2}(S^{1})}^{2}ds.\\
 & =\frac{1}{\varepsilon}\left\Vert p\right\Vert _{L^{2}((s_{0},s_{0}+\varepsilon)\times S^{1})}^{2}\\
 & \leq\frac{b_{2}}{\sqrt{\varepsilon}},
\end{align*}
and hence \[
S_{\varepsilon}(u)\cap[s_{0},s_{0}+\varepsilon]\ne\emptyset.
\]
This proves \eqref{eq:sepsilon set}. Exactly the same computation
with $\left\Vert x'(s,\cdot)\right\Vert _{L^{2}(S^{1})}^{2}$ instead
of $\left\Vert p(s,\cdot)\right\Vert _{L^{2}(S^{1})}^{2}$ proves
\eqref{eq:s dash}.

We can now improve \eqref{eq:b2} by finding a constant $b_{3}>0$
such that for all $s\in\mathbb{R}$ it holds that \begin{equation}
\left\Vert p(s,\cdot)\right\Vert _{L^{2}(S^{1})}\leq b_{3}.\label{eq:b3}
\end{equation}
Indeed, given $s\in\mathbb{R}$, choose $s_{0}\in S_{1}(u)$ such
that $\left|s-s_{0}\right|\leq1$ (i.e. take $\varepsilon=1$). Without
loss of generality assume $s\geq s_{0}$. Then we have \begin{align*}
\left\Vert p(s,\cdot)\right\Vert _{L^{2}(S^{1})}^{2} & =\left\Vert p(s_{0},\cdot)\right\Vert _{L^{2}(S^{1})}^{2}+\int_{s_{0}}^{s}\frac{d}{dr}\left\Vert p(r,\cdot)\right\Vert _{L^{2}(S^{1})}^{2}dr\\
 & =\left\Vert p(s_{0},\cdot)\right\Vert _{L^{2}(S^{1})}^{2}+2\int_{s_{0}}^{s}\int_{S^{1}}\left\langle p(r,t),p'(r,t)\right\rangle dtdr\\
 & \leq b_{2}+2\left|\int_{s_{0}}^{s}\left\Vert p(r,\cdot)\right\Vert _{L^{2}(S^{1})}^{2}dr\right|^{1/2}\left\Vert p'\right\Vert _{L^{2}((s_{0},s)\times S^{1})}\\
 & \leq b_{2}+2\sqrt{b_{2}}\left\Vert x'\right\Vert _{L^{2}((s_{0},s)\times S^{1})}\\
 & \leq b_{2}+2\sqrt{b_{2}}b_{1}.
\end{align*}
Thus \eqref{eq:b3-1} follows with \begin{equation}
b_{3}:=\sqrt{b_{2}+2\sqrt{b_{2}}b_{1}}.\label{eq:b3-1}
\end{equation}
 Next, we show how to improve \eqref{eq:sepsilon set} to obtain a
similar result with the $L^{2}(S^{1})$ norm replaced by the $L^{\infty}(S^{1})$
norm. Observe that\begin{align*}
\left\Vert \dot{p}(s,\cdot)\right\Vert _{L^{1}(S^{1})} & \leq\left\Vert \dot{x}(s,\cdot)\right\Vert _{L^{1}(S^{1})}\\
 & \leq\left\Vert J(\cdot,x)x'(s,\cdot)\right\Vert _{L^{1}(S^{1})}+\left|\eta(s)\right|\left\Vert X_{H}(x(s,\cdot))\right\Vert _{L^{1}(S^{1})}\\
 & \leq\left\Vert J\right\Vert _{\infty}\left\Vert x'(s,\cdot)\right\Vert _{L^{2}(S^{1})}+C_{0}b_{0}\left(1+\left\Vert p(s,\cdot)\right\Vert _{L^{2}(S^{1})}^{2}\right).\\
 & \leq\left\Vert J\right\Vert _{\infty}\left\Vert x'(s,\cdot)\right\Vert _{L^{2}(S^{1})}+C_{0}b_{0}\left(1+b_{3}^{2}\right),
\end{align*}
and hence \begin{align*}
\left\Vert p(s,\cdot)\right\Vert _{W^{1,1}(S^{1})} & \leq\left\Vert p(s,\cdot)\right\Vert _{L^{2}(S^{1})}+\left\Vert \dot{p}(s,\cdot)\right\Vert _{L^{1}(S^{1})}\\
 & \leq b_{3}+\left\Vert J\right\Vert _{\infty}\left\Vert x'(s,\cdot)\right\Vert _{L^{2}(S^{1})}+C_{0}b_{0}\left(1+b_{3}^{2}\right).
\end{align*}
Thus if $N>0$ is the uniform constant such that for any map $f\in W^{1,1}(S^{1},\mathbb{R})$
it holds that \begin{equation}
\left\Vert f\right\Vert _{L^{\infty}(S^{1})}\leq N\left\Vert f\right\Vert _{W^{1,1}(S^{1})},\label{eq:N}
\end{equation}
then \begin{equation}
\left\Vert p(s,\cdot)\right\Vert _{L^{\infty}(S^{1})}\leq Nb_{3}+N\left\Vert J\right\Vert _{\infty}\left\Vert x'(s,\cdot)\right\Vert _{L^{2}(S^{1})}+NC_{0}b_{0}\left(1+b_{3}^{2}\right).\label{eq:linfty p}
\end{equation}
Set \begin{equation}
b_{4}:=Nb_{3}+NC_{0}b_{0}(1+b_{3}^{2}),\ \ \ b_{5}:=N\left\Vert J\right\Vert _{\infty}b_{1}.\label{eq:b4 and b5}
\end{equation}
It now follows from \eqref{eq:s dash} and \eqref{eq:linfty p} that
for any $0<\varepsilon\leq1$ the subset \[
S_{\varepsilon}''(u):=\left\{ s\in\mathbb{R}\,:\,\left\Vert p(s,\cdot)\right\Vert _{L^{\infty}(S^{1})}\leq b_{4}+\frac{b_{5}}{\sqrt{\varepsilon}}\right\} 
\]
has non-empty intersection with any interval of length $\geq\varepsilon$. 

Next, we observe that for any $(s,t)\in\mathbb{R}\times S^{1}$, we
have \begin{align*}
\left|\nabla p(s,t)\right|^{2} & \leq\left|\nabla x(s,t)\right|^{2}\\
 & =\left|x'(s,t)\right|^{2}+\left|\dot{x}(s,t)\right|^{2}\\
 & =\left|x'(s,t)\right|^{2}+\left|J(t,x)x'(s,t)-\eta(s)X_{H}(x(s,t))\right|^{2}\\
 & \overset{(*)}{\leq}\left(1+2\left\Vert J\right\Vert _{\infty}^{2}\right)\left|x'(s,t)\right|^{2}+2b_{0}^{2}C_{0}^{2}\left(1+\left|p(s,t)\right|^{2}\right)^{2}\\
 & \leq b_{6}\left(1+\left|x'(s,t)\right|^{2}+\left|p(s,t)\right|^{4}\right)
\end{align*}
for some constant $b_{6}>0$, where $(*)$ used $\left|a-b\right|^{2}\leq2a^{2}+2b^{2}$.
Thus for all $s_{0}<s_{1}$ we have \begin{equation}
\left\Vert \nabla p\right\Vert _{L^{2}((s_{0},s_{1})\times S^{1})}^{2}\leq b_{6}\left(\left|s_{1}-s_{0}\right|+b_{1}^{2}\right)+b_{6}\left\Vert p\right\Vert _{L^{4}((s_{0},s_{1})\times S^{1})}^{4}.\label{eq:l4 estiamte}
\end{equation}
The final step of this part of the proof is to show that there exists
$b_{7}>0$ such that for any map $u=(x,\eta)$ satisfying the hypotheses
of the theorem, and any finite interval $I\subseteq\mathbb{R}$ we
have, writing $x=(q,p)$ that \begin{equation}
\left\Vert \nabla p\right\Vert _{L^{2}(I\times S^{1})}\leq b_{7}(1+\left|I\right|^{1/2}).\label{eq:b7 estimate}
\end{equation}
The proof of \eqref{eq:b7 estimate} from \eqref{eq:l4 estiamte}
is based on an interpolation inequality between the $L^{4}$ norm
and the $L^{2}$ and $W^{1,2}$ norms, which is due to Abbondandolo
and Schwarz. There is no difference between the proof in \cite[p278-279]{AbbondandoloSchwarz2006}
and the one in our situation, so we will omit this. It will be important
however in the final section of this paper (see the proof of Proposition
\ref{quadr is const}) to state it precisely. The following lemma
is not explicitly stated in \cite{AbbondandoloSchwarz2006}, but follows
immediately from a careful inspection of \cite[p278-279]{AbbondandoloSchwarz2006}.
\begin{lem}
\label{lem:the gamma lemma}Suppose $x=(q,p):\mathbb{R}\rightarrow\Lambda T^{*}M$
is a smooth map such that there exist constants $\gamma_{1},\gamma_{2},\gamma_{3}>0$
with the following properties:
\begin{enumerate}
\item $\left\Vert x'\right\Vert _{L^{2}(\mathbb{R}\times S^{1})}\leq\gamma_{1}$;
\item $\left\Vert p(s,\cdot)\right\Vert _{L^{2}(S^{1})}\leq\gamma_{2}$
for all $s\in\mathbb{R}$;
\item \textup{$\left\Vert \nabla p\right\Vert _{L^{2}((s_{0},s_{1})\times S^{1})}^{2}\leq\gamma_{3}\left(\left|s_{1}-s_{0}\right|+\gamma_{1}^{2}\right)+\gamma_{3}\left\Vert p\right\Vert _{L^{4}((s_{0},s_{1})\times S^{1})}^{4}$
for all $s_{0},s_{1}\in\mathbb{R}$ with $s_{0}<s_{1}$.}
\end{enumerate}
Then there exists a constant%
\footnote{The constant $\varepsilon_{*}$ corresponds to the constant $\delta=1/(32b_{1}Cc_{4}^{2})$
in \cite[p279]{AbbondandoloSchwarz2006}.%
} $0<\varepsilon_{*}\leq1$ depending only on $\gamma_{2}$ and $\gamma_{3}$
such that if \textbf{\emph{in addition}} there exists a constant $\gamma_{*}>0$
such that the set\[
\left\{ s\in\mathbb{R}\,:\,\left\Vert p(s,\cdot)\right\Vert _{L^{\infty}(S^{1})}\leq\gamma_{*}\right\} 
\]
 is $\varepsilon_{*}$-dense then there exists a constant $\Gamma=\Gamma(\gamma_{1},\gamma_{2},\gamma_{3},\gamma_{*})>0$
such that\[
\left\Vert \nabla p\right\Vert _{L^{2}(I\times S^{1})}\leq\Gamma(1+\left|I\right|^{1/2})
\]
 for any finite interval $I\subseteq\mathbb{R}$.\textup{ }
\end{lem}
The important point (as far as Proposition \ref{quadr is const} is
concerned) is that the constant $\Gamma$ depends only on $\gamma_{1},\gamma_{2},\gamma_{3}$
and $\gamma_{*}$. Anyway, applying the lemma, \eqref{eq:b7 estimate}
follows. The proof of Step 1 now follows with $K:=\max\{b_{3},b_{7}\}$.\newline

\textbf{Step 2.}\newline

The next part of the proof shows how the $L^{2}$ estimates \eqref{eq:step 1 estiamte}
on $p$ and $\nabla p$ on intervals leads to uniform $L^{\infty}$
bounds. This part of the proof closely follows \cite[Theorem 1.14.(i)]{AbbondandoloSchwarz2006},
and uses the fact that $J\in B_{\varepsilon_{1}}(J_{0})$, and the
conclusion of Step 1.

Let $\rho:\mathbb{R}\rightarrow[0,1]$ denote a smooth function such
that $\mbox{supp}(\rho)\subseteq(-1,2)$, $\rho|_{[0,1]}=\mathbb{1}$
and $\left|\rho'\right|\leq2$. Given a map $u=(x,\eta)$ satisfying
the hypotheses of the theorem and $i\in\mathbb{Z}$, define $x_{i}:\mathbb{R}\rightarrow\Lambda T^{*}M$
by \[
x_{i}(s,t):=\rho(s-i)x(s,t).
\]
Note that \[
(\partial_{s}+J_{0}\partial_{t})x_{i}(s,t)=\rho'(s-i)x(s,t)+\rho(s-i)\eta(s)J(t,x)X_{H}(x(s,t))+\rho(s-i)(J_{0}(x)-J(t,x))\dot{x}(s,t).
\]
Since $x_{i}$ is compactly supported, Theorem \ref{thm:calderon }
applies, and we conclude:\begin{align*}
\left\Vert \nabla x_{i}\right\Vert _{L^{3}(\mathbb{R}\times S^{1})} & \leq\frac{1}{2\varepsilon_{1}}\left\Vert (\partial_{s}+J_{0}\partial_{t})x_{i}\right\Vert _{L^{3}(\mathbb{R}\times S^{1})}\\
 & \leq\frac{1}{\varepsilon_{1}}\left(R+\left\Vert p\right\Vert _{L^{3}((i-1,i+2)\times S^{1})}\right)+\frac{C_{0}\left\Vert J\right\Vert _{\infty}}{2\varepsilon_{1}}\left\Vert X_{H}(x)\right\Vert _{L^{3}((i-1,i+2)\times S^{1})}\\
 & +\frac{\left\Vert J_{0}-J\right\Vert _{\infty}}{2\varepsilon_{1}}\left\Vert \dot{x}\right\Vert _{L^{3}(\mathbb{R}\times S^{1})},
\end{align*}
where $R>0$ is a constant depending only on the diameter of the closed
manifold $M$. 

Given $r>2$, let $P_{r}>0$ denote the constant such that for any
$f\in W^{1,r}((0,3)\times S^{1},\mathbb{R})$ it holds that \[
\left\Vert f\right\Vert _{L^{r}((0,3)\times S^{1})}\leq P_{r}\left\Vert f\right\Vert _{W^{1,2}((0,3)\times S^{1})}.
\]
Using \eqref{eq:b0} and Step 1 we see that \begin{align*}
\left\Vert X_{H}(x)\right\Vert _{L^{3}((i-1,i+2)\times S^{1})} & \leq b_{0}\left(3^{1/3}+\left\Vert p\right\Vert _{L^{6}((i-1,i+2)\times S^{1})}^{2}\right)\\
 & \leq b_{0}\left(3^{1/3}+P_{6}^{2}\left\Vert p\right\Vert _{W^{1,2}((i-1,i+2)\times S^{1})}^{2}\right)\\
 & \leq b_{0}\left(3^{1/3}+P_{6}^{2}(K+2K\sqrt{3})^{2}\right).
\end{align*}
Similarly\[
\left\Vert p\right\Vert _{L^{3}((i-1,i+2)\times S^{1})}\leq P_{3}\left\Vert p\right\Vert _{W^{1,2}((i-1,i+2)\times S^{1})}\leq P_{3}(K+2K\sqrt{3}).
\]
Putting this altogether, and using the fact that $J\in B_{\varepsilon_{1}}(J)$
we have therefore proved that there exists a constant $C>0$ that
is independent of $u$ and $i$ such that \[
\left\Vert \nabla x_{i}\right\Vert _{L^{3}(\mathbb{R}\times S^{1})}\leq C+\frac{1}{2}\left\Vert \dot{x}\right\Vert _{L^{3}(\mathbb{R}\times S^{1})}.
\]
Thus \[
\left\Vert \nabla x\right\Vert _{L^{3}((i,i+1)\times S^{1})}\leq2C.
\]
This gives a uniform bound for $x_{i}$ in $W^{1,3}((i,i+1)\times S^{1})$,
and hence also in $L^{\infty}((i,i+1)\times S^{1})$. Since this bound
does not depend on $i$, we have proved the existence of a uniform
bound for $x$ in $L^{\infty}(\mathbb{R}\times S^{1})$. The theorem
follows.
\end{proof}
We now turn to the final $L^{\infty}$ estimate we will need. It is
based on \cite[Theorem 1.14.(iii)]{AbbondandoloSchwarz2006}. It will
be needed to construct the short exact sequence between the Rabinowitz
Floer complex and the Morse (co)complex in the next section. In the
statement of the theorem one should substitute either `$+$' or `$-$'
for `$\pm$' throughout.
\begin{thm}
\label{thm:half spaces linfinity}There exist constants $\varepsilon_{2}^{\pm}>0$
with the following property: Suppose $J\in\mathcal{J}(\omega)\cap B_{\varepsilon_{2}^{\pm}}(J_{g})$.
Fix $(g,\sigma,U,k)\in\mathcal{O}$, $R>0$, $\alpha\in[S^{1},M]$
and $-\infty<a<b<\infty$. Let $H:=H_{g}+\pi^{*}U$. Then there exist
constants $C_{4}^{\pm},C_{5}^{\pm}>0$ such that for any map \[
u=(x,\eta):\mathbb{R}^{\pm}\times S^{1}\rightarrow T^{*}M\times\mathbb{R}
\]
with\[
x\in C^{\infty}(\mathbb{R}^{\pm}\times S^{1},T^{*}M)\cap W^{1,3}((0,\pm1)\times S^{1},T^{*}M);
\]
\[
\eta\in C^{\infty}(\mathbb{R}^{\pm},\mathbb{R})\cap W^{1,3}((0,\pm1),\mathbb{R}),
\]
that satisfies the Rabinowitz Floer equation on $\mathbb{R}^{\pm}\times S^{1}$,
has action bounds $A_{H-k}(u(\mathbb{R}^{\pm}))\subseteq[a,b]$ together
with the extra assumptions:\[
\pm\eta(0)\geq0;
\]
\[
\left\Vert q(0,\pm\cdot)\right\Vert _{W^{2/3,3}(S^{1},\mathbb{R}^{d})}\leq R
\]
(where here $d$ is such that $(M,g)$ embeds isometrically into $(\mathbb{R}^{d},g_{0})$,
and we have written $x=(q,p)$), it holds that: \[
\left\Vert \eta\right\Vert _{L^{\infty}(\mathbb{R}^{\pm})}\leq C_{4}^{\pm},\ \ \ \left\Vert x\right\Vert _{L^{\infty}(\mathbb{R}^{\pm}\times S^{1})}\leq C_{5}^{\pm}.
\]
\end{thm}
\begin{proof}
Firstly, the proof of Theorem \ref{thm:bounding the lagrange multiplier}
will still go through for flow lines defined on $\mathbb{R}^{+}$
instead of $\overline{\mathbb{R}}$, provided we have an a priori
lower bound on $\eta(0)$. If $u$ is defined on $\mathbb{R}^{-}$
then the proof will go through provided we (a) have an a priori upper
bound on $\eta(0)$, and (b), we we redefine the function $\tau(s)$
from \eqref{eq:def of tau of s} to be \[
\tau(s):=\inf\left\{ r\geq0\,:\,\left\Vert \nabla A_{H-k}(u(s-r,\cdot))\right\Vert _{J}\leq\rho_{0}\right\} .
\]
Therefore we have proved the existence of constants $C_{4}^{\pm}>0$
that uniformly bound the $\eta$-component of any map $u$ satisfying
the hypotheses of the theorem. Now Step 1 from the proof of Theorem
\ref{thm:l infinity} goes through without any essential changes (save
of course from the fact that now $u$ is defined on $\mathbb{R}^{\pm}$).
The proof of Step 2 also proceeds similarly, aside from the fact that
instead of following \cite[Theorem 1.14.(ii)]{AbbondandoloSchwarz2006}
we must instead follow \cite[Theorem 1.14.(iii)]{AbbondandoloSchwarz2006}.
In particular, the constants $\varepsilon_{2}^{\pm}>0$ in the statement
of the theorem come from a version of Theorem \ref{thm:calderon }
for maps defined on $\mathbb{R}^{\pm}\times S^{1}$ instead of $\mathbb{R}\times S^{1}$.
\end{proof}

\section{The Abbondandolo-Schwarz short exact sequence}

In this section we state and prove the main result of the paper, which
is the extension of \cite[Theorem 2]{AbbondandoloSchwarz2009} to
the weakly exact case. In the statement of the theorem below it is
implicitly assumed that $\left\Vert \sigma\right\Vert _{\infty}$
is sufficiently small; this ensures that almost complex structures
that fit the hypotheses of theorem exist, cf. Remark \ref{rem:shrinking sigma}.
\begin{thm}
\label{thm:theorem A precise}Fix $(g,\sigma,U,k)\in\mathcal{O}_{\textrm{\emph{reg}}}$
and $\alpha\in[S^{1},M]$. Put $H=H_{g}+\pi^{*}U$ and $L=L_{g}-\pi^{*}U$.
Let $f$ and $h$ be Morse functions on $\overline{\mbox{\emph{Crit}}}(S_{L+k})$
and $\mbox{\emph{Crit}}(A_{H-k})$ satisfying certain compatibility
requirements (stated precisely in Subsection \ref{thm:theorem A precise}
below). Let $J\in\mathcal{J}(\omega)$ denote a generically chosen
almost complex structure lying sufficiently close to the metric almost
complex structure $J_{g}$. Let $G$ denote a generically chosen metric
on $\Lambda M\times\mathbb{R}^{+}$ that is uniformly equivalent to
$\left\langle \left\langle \cdot,\cdot\right\rangle \right\rangle _{g}$,
and let $g_{0}$ and $g_{1}$ denote generically chosen Riemannian
metrics on $\overline{\mbox{\emph{Crit}}}(S_{L+k})$ and $\mbox{\emph{Crit}}(A_{H-k})$
respectively, such that the negative gradient flows of $f$ and $h$
with respect to these metrics are Morse-Smale. Then there exists:
\begin{enumerate}
\item An injective chain map $\Phi_{\textrm{\emph{SA}}}:CM_{*}(S_{L+k},f;\alpha)\rightarrow RF_{*}(A_{H-k},h;\alpha)$
which admits a left inverse $\widehat{\Phi}_{\textrm{\emph{SA}}}:RF_{*}(A_{H-k},h;\alpha)\rightarrow CM_{*}(S_{L+k},f;\alpha)$.
\item A surjective chain map $\Phi_{\textrm{\emph{AS}}}:RF_{*}(A_{H-k},h;\alpha)\rightarrow CM^{1-*}(S_{L+k},-f;-\alpha)$
which admits a right inverse $\widehat{\Phi}_{\textrm{\emph{AS}}}:CM^{1-*}(S_{L+k},-f;-\alpha)\rightarrow RF_{*}(A_{H-k},h;\alpha)$.
\end{enumerate}
Moreover the composition $\Phi_{\textrm{\emph{AS}}}\circ\Phi_{\textrm{\emph{SA}}}:CM_{*}(S_{L+k},f;\alpha)\rightarrow CM^{1-*}(S_{L+k},-f;-\alpha)$
is chain homotopic to zero, that is, there exists a homomorphism $P:CM_{*}(S_{L+k},f;\alpha)\rightarrow CM^{-*}(S_{L+k},-f;-\alpha)$
such that \[
\Phi_{\textrm{\emph{AS}}}\circ\Phi_{\textrm{\emph{SA}}}=P\partial^{\textrm{\emph{Morse}}}+\delta^{\textrm{\emph{Morse}}}P.
\]
Setting\[
\Theta:=\Phi_{\textrm{\emph{SA}}}-\widehat{\Phi}_{\textrm{\emph{AS}}}P\partial^{\textrm{\emph{Morse}}}-\partial^{\textrm{\emph{Morse}}}\widehat{\Phi}_{\textrm{\emph{AS}}}P,
\]
the chain map $\Theta$ is chain homotopic to $\Phi_{\textrm{\emph{SA}}}$,
and satisfies $\Phi_{\textrm{\emph{AS}}}\circ\Theta=0$, and thus
we obtain a short exact sequence of chain complexes\[
0\rightarrow CM_{*}(S_{L+k},f;\alpha)\overset{\Theta}{\rightarrow}RF_{*}(A_{H-k},h;\alpha)\overset{\Phi_{\textrm{\emph{AS}}}}{\rightarrow C}M^{1-*}(S_{L+k},-f;-\alpha)\rightarrow0.
\]
Identifying $HM_{*}(S_{L+k},f;\alpha)\cong H_{*}(\Lambda_{\alpha}M;\mathbb{Z}_{2})$
and $HM^{*}(S_{L+k},-f;-\alpha)\cong H_{\textrm{}}^{*}(\Lambda_{-\alpha}M;\mathbb{Z}_{2})$,
and passing to the associated long exact sequence\[
\xymatrix{\dots\ar[r] & H_{i}(\Lambda_{\alpha}M;\mathbb{Z}_{2})\ar[r]^{\Theta_{*}} & RFH_{i}(A_{H-k};\alpha)\ar[r]^{(\Psi_{\textrm{\emph{AS}}})_{*}} & H_{\textrm{}}^{1-i}(\Lambda_{-\alpha}M;\mathbb{Z}_{2})\ar[r]^{\Delta} & H_{i-1}(\Lambda_{\alpha}M;\mathbb{Z}_{2})\ar[r] & \dots}
\]
the connecting homomorphism $\Delta$ is identically zero unless $\alpha=0$
and $i=0$, in which case it is multiplication by the Euler class
$e(T^{*}M)$. This therefore allows one to obtain a complete description
of the Rabinowitz Floer homology of $A_{H-k}$. 
\end{thm}
As mentioned in the introduction, the proof of this theorem is now
essentially identical to the corresponding proof in \cite{AbbondandoloSchwarz2009}.
We therefore omit almost all of the technical details, referring the
reader to the beautiful and lucid exposition in \cite{AbbondandoloSchwarz2009},
and instead just give an outline of Abbondandolo and Schwarz' constructions.

\subsection{\label{sub:Choosing-the-Morse}Choosing the Morse functions $f$
and $h$}

$\ $\vspace{6 pt}

In order to construct the chain homotopy $P$ in the theorem above
it is essential that the Morse functions $f:\overline{\mbox{Crit}}(S_{L+k})\rightarrow\mathbb{R}$
and $h:\mbox{Crit}(A_{H-k})\rightarrow\mathbb{R}$ are chosen in such
a way that certain compatibility requirements are satisfied. More
precisely, we require that the following four conditions are satisfied:
\begin{enumerate}
\item For all $w\in\mbox{Crit}(f)$, it holds that$f(w)=h(Z^{\pm}(w))$,
and $i_{f}(w)=i_{h}(Z^{\pm}(w))$.
\item The function $f|_{M\times\{0\}}$ has a unique minimum and a unique
maximum and is \textbf{self-indexing}, that is, $f(q,0)=i_{f}((q,0))$
for all $(q,0)\in\overline{\mbox{Crit}}(f)\backslash\mbox{Crit}(f)$.
\item For all $x\in\Sigma_{k}$, we have $f(\pi(x),0)\leq h(x,0)\leq f(\pi(x),0)+1/2$.
\item Every critical point of $h|_{\Sigma_{k}\times\{0\}}$ lies above a
critical point of $f|_{M\times\{0\}}$, and moreover for each critical
point $(q,0)$ of $f|_{M\times\{0\}}$ there are exactly two critical
points of $h|_{\Sigma_{k}\times\{0\}}$ in the fibre $(\Sigma_{k}\cap T_{q}^{*}M)\times\{0\}$.
Denoting these two critical points by $(x_{q}^{\pm},0)$, it holds
that $f(q,0)=h(x_{q}^{-},0)=h(x_{q}^{+},0)-1/2$, and that $i_{f}(q,0)=i_{h}(x_{q}^{-},0)=i_{h}(x_{q}^{+},0)-n+1$.
\end{enumerate}
Such functions exist because $k>e_{0}(g,\sigma,U)$. This is explained
in detail in \cite[Appendix B]{AbbondandoloSchwarz2009}. An immediate
consequence of these requirements and Proposition \ref{pro:morse indices and cz indices}
is the following result.
\begin{lem}
Assume that the Morse functions $f:\overline{\mbox{\emph{Crit}}}(S_{L+k})\rightarrow\mathbb{R}$
and $h:\mbox{\emph{Crit}}(A_{H-k})\rightarrow\mathbb{R}$ satisfy
the requirements above. 

Then \[
\widehat{i}_{f}(w)=\widehat{\mu}_{h}(Z^{+}(w));
\]
\[
\widehat{i}_{-f}(w)=1-\widehat{\mu}_{h}(Z^{-}(w))
\]
for $w\in\mbox{\emph{Crit}}(f)$ and \[
\widehat{i}_{f}((q,0))=n-\widehat{\mu}_{h}(x_{q}^{+},0);
\]
\[
\widehat{i}_{-f}((q,0))=1-\widehat{\mu}_{h}(x_{q}^{-},0)
\]
for $(q,0)\in\overline{\mbox{\emph{Crit}}}(f)\backslash\mbox{\emph{Crit}}(f)$.
\end{lem}

\subsection{The chain map $\Phi_{\textrm{SA}}$}

$\ $\vspace{6 pt}

In order to define the chain map $\Phi_{\textrm{SA}}$, one first
needs to construct a suitable moduli space. Here are the details.
Recall that $G$ denotes a metric on $\Lambda M\times\mathbb{R}^{+}$
that is uniformly equivalent to $\left\langle \left\langle \cdot,\cdot\right\rangle \right\rangle _{g}$
and $g_{0}$ is a Riemannian metric on $\overline{\mbox{Crit}}(S_{L+k})$
such that the negative gradient flow $\phi_{t}^{-\nabla f}$ of $-\nabla f$
is Morse-Smale, and $g_{1}$ is a Riemannian metric on $\overline{\mbox{Crit}}(A_{H-k})$
such that the negative gradient flow $\phi_{t}^{-\nabla h}$ is Morse-Smale.
Fix a generic almost complex structure $J\in\mathcal{J}(\omega)\cap B_{\varepsilon_{2}^{+}}(J_{g})$
(where $\varepsilon_{2}^{+}>0$ is the constant from Theorem \ref{thm:half spaces linfinity}).\newline

Fix $w\in\overline{\mbox{Crit}}(f)$. If $m\in\mathbb{N}$, let $\widetilde{\mathcal{W}}_{m}^{-}(w)$
denote the set of tuples $\boldsymbol{w}=(w_{1},\dots,w_{m})$ such
that $w_{i}\in\Lambda M\times\mathbb{R}^{+}$ for $i=1,\dots,m-1$
and $w_{m}\in\Lambda M\times\mathbb{R}_{0}^{+}$, and such that \[
w_{1}\in W^{u}(W^{u}(w;-\nabla f);-\nabla S_{L+k});
\]
\[
\Psi_{-\infty}(w_{i+1})\in\phi_{\mathbb{R}^{+}}^{-\nabla f}(\Psi_{\infty}(w_{i})).
\]
 Let $\mathcal{W}_{m}^{-}(w)$ denote the quotient of $\widetilde{\mathcal{W}}_{m}^{-}(w)$
under the free $\mathbb{R}^{m-1}$ action given by\[
(w_{1},\dots,w_{m-1})\mapsto(\Psi_{s_{1}}(w_{1}),\dots,\Psi_{s_{m-1}}(w_{m-1})),\ \ \ (s_{1},\dots,s_{m-1})\in\mathbb{R}^{m-1}.
\]
Then put \[
\mathcal{W}^{-}(w):=\bigcup_{m\in\mathbb{N}\cup\{0\}}\mathcal{W}_{m}^{-}(w).
\]
For generically chosen $G$ and $g_{0}$, $\mathcal{W}^{-}(w)$ has
the structure of a smooth finite dimensional manifold of dimension
$\widehat{i}_{f}(w)$.\newline 

Fix $z\in\mbox{Crit}(h)$. Let $\widetilde{\mathcal{M}}_{m}^{+}(z)$
denote the denote the set of tuples of maps $\boldsymbol{u}=(u_{1},\dots,u_{m})$
such that \[
u_{1}:\mathbb{R}_{0}^{+}\rightarrow C^{\infty}(S^{1},T^{*}M)\times\mathbb{R};
\]
\[
u_{2},\dots,u_{m}:\mathbb{R}\rightarrow C^{\infty}(S^{1},T^{*}M)\times\mathbb{R},
\]
 all satisfy the Rabinowitz Floer equation \eqref{eq:rfeq} (which
are possibly stationary solutions) and such that\[
u_{m}(\infty)\in W^{s}(z;-\nabla h);
\]
\[
u_{i+1}(-\infty)\in\phi_{\mathbb{R}^{+}}^{-\nabla h}(u_{i}(\infty))\ \ \ \mbox{for }i=1,\dots,m-1.
\]
Let $\mathcal{M}_{m}^{+}(z)$ denote the quotient of $\widetilde{\mathcal{M}}_{m}^{+}(z)$
under the free $\mathbb{R}^{m-1}$ action given by translation along
the flow lines $u_{2},\dots,u_{m}$. 

Then put \[
\mathcal{M}^{+}(z):=\bigcup_{m\in\mathbb{N}\cup\{0\}}\mathcal{M}_{m}^{+}(z).
\]
The space $\mathcal{M}^{+}(z)$ is not finite dimensional. However,
by restricting where the tuple $\boldsymbol{u}$ can {}``begin'',
we can cut it down to something finite dimensional. This is precisely
what the moduli space $\mathcal{M}_{\textrm{SA}}(w,z)$ does. Namely,
the moduli space $\mathcal{M}_{\textrm{SA}}(w,z)$ is defined to be
the following subset of $\mathcal{W}^{-}(w)\times\mathcal{M}^{+}(z)$.
A pair $([\boldsymbol{w}],[\boldsymbol{u}])$ (where the square brackets
denote the equivalence class after dividing through by the translation
actions) belongs to $\mathcal{M}_{\textrm{SA}}(w,z)$ if and only
if we have, writing \[
\boldsymbol{w}=(w_{1},\dots,w_{m});
\]
\[
\boldsymbol{u}=(u_{1},\dots,u_{j})\mbox{\ \ \ with }u_{j}=(x_{j},\eta_{j}),
\]
that \[
w_{m}=(\pi\circ x_{1}(0),\eta_{1}(0)).
\]
In other words, the tuple $\boldsymbol{w}$ must {}``end'' where
the tuple $\boldsymbol{u}$ {}``begins''. 

For a fixed element $w^{*}\in\Lambda M\times\mathbb{R}_{0}^{+}$,
requiring tuples $\boldsymbol{u}$ to {}``begin'' at $w^{*}$ in
the sense that $(\pi\circ x_{1}(0),\eta_{1}(0))=w^{*}$ defines a
Lagrangian boundary condition. This implies that we have a Fredholm
problem, and since generically $\mathcal{W}^{-}(w)$ is a finite dimensional
manifold, it follows that $\mathcal{M}_{\textrm{SA}}(w,z)$ can be
seen as the zero set of a Fredholm operator, whose index can be computed
to be $\widehat{i}_{f}(w)-\widehat{\mu}_{h}(z)$. In fact, more is
true. Namely, $\mathcal{M}_{\textrm{SA}}(w,z)$ (for generic $G,g_{0},J$
and $g_{1}$) is a precompact finite dimensional manifold of dimension
$\widehat{i}_{f}(w)-\widehat{\mu}_{h}(z)$. 

This requires us to check two more things. Firstly, one needs to have
$C_{\textrm{loc}}^{\infty}$-bounds for the curves $\boldsymbol{u}=(u_{1},\dots,u_{j})$.
Here the following key inequality comes into play. Given $([\boldsymbol{w}],[\boldsymbol{u}])\in\mathcal{M}_{\textrm{SA}}(w,z)$,
equation \eqref{eq:bound plus} from Lemma \ref{lem:relating critical points lemma}
tells us that for all $s\in\mathbb{R}^{+}$: \begin{multline*}
S_{L+k}(w)\geq S_{L+k}(w_{i})\geq S_{L+k}(w_{m})=S_{L+k}(\pi\circ x_{1}(0,\cdot),\eta_{1}(0))\\
\geq A_{H-k}(u_{1}(0,\cdot))\geq A_{H-k}(u_{i}(s,\cdot))\geq A_{H-k}(z).
\end{multline*}
Then uniform $L^{\infty}$ estimates for the solutions $u_{2},\dots,u_{j}$
come from Theorem \ref{thm:l infinity}, and the uniform $L^{\infty}$
estimate for $u_{1}$ comes from Theorem \ref{thm:half spaces linfinity}.
As before, these $L^{\infty}$ bounds give us $C_{\textrm{loc}}^{\infty}$
bounds (since $\omega|_{\pi_{2}(M)}=0$ and $c_{1}(T^{*}M,\omega)=0$).
This shows that the moduli spaces $\mathcal{M}_{\textrm{SA}}(w,z)$
are compact up to breaking. 

The only complication with obtaining transversality is the presence
of stationary solutions, which can appear if $z=Z^{+}(w)$ or $w=(q,0)\in\overline{\mbox{Crit}}(f)\backslash\mbox{Crit}(f)$
is a critical point at infinity and $z=(x_{q}^{\pm},0)$ is one of
the corresponding two critical points of $h$. In the former case
the first inequality of the third statement of Lemma \ref{lem:relating critical points lemma}
forces the linearized operator defining the moduli space $\mathcal{M}_{\textrm{SA}}(w,Z^{+}(w))$
to be an isomorphism (see \cite[Lemma 6.2]{AbbondandoloSchwarz2009}
or \cite[Proposition 3.7]{AbbondandoloSchwarz2006}), and in the second
two cases the four assumptions made earlier on the Morse functions
$f$ and $a$ guarantee that the linearized operator defining the
moduli spaces $\mathcal{M}_{\textrm{SA}}((q,0),(x_{q}^{\pm},0))$
is surjective (see \cite[Lemma 6.3]{AbbondandoloSchwarz2009}).\newline

Putting this together, we deduce that when $\widehat{i}_{f}(w)=\widehat{\mu}_{h}(z)$,
the moduli space $\mathcal{M}_{\textrm{SA}}(w,z)$ is a finite set,
and hence we can define \[
n_{\textrm{SA}}(w,z):=\#\mathcal{M}_{\textrm{SA}}(w,z)\ \ \ \mbox{taken modulo }2.
\]
Then one defines $\Phi_{\textrm{SA}}:CM_{*}(S_{L+k},f)\rightarrow RF_{*}(A_{H-k},h)$
by \[
\Phi_{\textrm{SA}}w=\sum_{z\in\textrm{Crit}_{i}(h)}n_{\textrm{SA}}(w,z)z,\ \ \ w\in\overline{\mbox{Crit}}_{i}(f).
\]
A standard gluing argument shows that $\Phi_{\textrm{SA}}$ is a chain
map. It is clear that $\Phi_{\textrm{SA}}$ restricts to define a
chain map $CM(S_{L+k},f;\alpha)\rightarrow RF_{*}(A_{H-k},h;\alpha)$
for each $\alpha\in[S^{1},M]$.

\subsection{The chain map $\Phi_{\textrm{AS}}$}

$\ $\vspace{6 pt}

The chain map $\Phi_{\textrm{AS}}$ is defined in much the same way.
One begins by defining spaces $\mathcal{M}^{-}(z)$ for $z\in\mbox{Crit}(h)$.
Let $\widetilde{\mathcal{M}}_{m}^{-}(z)$ denote the denote the set
of tuples of maps $\boldsymbol{u}=(u_{1},\dots,u_{m})$ such that
\[
u_{1},\dots,u_{m-1}:\mathbb{R}\rightarrow C^{\infty}(S^{1},T^{*}M)\times\mathbb{R};
\]
\[
u_{m}:\mathbb{R}_{0}^{-}\rightarrow C^{\infty}(S^{1},T^{*}M)\times\mathbb{R},
\]
all satisfy the Rabinowitz Floer equation \eqref{eq:rfeq} (which
are possibly stationary solutions) and such that\[
u_{-}(\infty)\in W^{u}(z;-\nabla h),
\]
 and such that \[
u_{i+1}(-\infty)\in\phi_{\mathbb{R}^{+}}^{-\nabla h}(u_{i}(\infty))\ \ \ \mbox{for }i=1,\dots,m-1.
\]
 Let $\mathcal{M}_{m}^{-}(z)$ denote the quotient of $\widetilde{\mathcal{M}}_{m}^{-}(z)$
under the free $\mathbb{R}^{m-1}$ action and put \[
\mathcal{M}^{-}(z):=\bigcup_{m\in\mathbb{N}}\mathcal{M}_{m}^{-}(z).
\]
Given $z\in\mbox{Crit}(h)$ and $w\in\overline{\mbox{Crit}}(-f)$,
the moduli space $\mathcal{M}_{\textrm{AS}}(z,w)$ consists of the
subset $\mathcal{M}^{-}(z)\times\mathcal{W}^{-}(w)$ of elements $([\boldsymbol{u}],[\boldsymbol{w}])$
such that, writing\[
\boldsymbol{u}=(u_{1},\dots,u_{j})\mbox{\ \ \ with }u_{i}=(x_{i},\eta_{i});
\]
\[
\boldsymbol{w}=(w_{1},\dots,w_{m})\ \ \ \mbox{with }w_{i}=(q_{i},T_{i}),
\]
we have \[
(q_{m}(t),T_{m})=(\pi\circ x_{j}(0,-t),-\eta_{j}(0)).
\]
This time the moduli space $\mathcal{M}_{\textrm{AS}}(z,w)$ admits
the structure of a precompact smooth manifold of finite dimension
$\widehat{\mu}_{h}(z)+\widehat{i}_{-f}(w)-1$. Here one uses equation
\eqref{eq:bound minus} from Lemma \ref{lem:relating critical points lemma}
to deduce the inequality \begin{multline*}
A_{H-k}(z)\geq A_{H-k}(u_{i}(s,\cdot))\geq A_{H-k}(u_{j}(0,\cdot))\\
\geq-S_{L+k}(\pi\circ x_{j}(0,-\cdot),-\eta_{j}(0))\geq-S_{L+k}(w_{m})\geq-S_{L+k}(w_{i})\geq-S_{L+k}(w),
\end{multline*}
which gives the required $L^{\infty}$ estimates on the $u_{i}$,
and the second inequality in the third statement of Lemma \ref{lem:relating critical points lemma}
to obtain the automatic transversality in the case $z=Z^{-}(w)$.
Thus if $z\in\mbox{Crit}(h)$ and $w\in\overline{\mbox{Crit}}(-f)$
satisfy $\widehat{\mu}_{h}(z)+\widehat{i}_{-f}(w)=1$, $\mathcal{M}_{\textrm{AS}}(z,w)$
is a finite set, and hence we may define $n_{\textrm{AS}}(z,w)$ to
be its parity. This defines the chain map $\Phi_{\textrm{AS}}$. As
before $\Phi_{\textrm{AS}}$ restricts to define a chain map $RF_{*}(A_{H-k},h;\alpha)\rightarrow CM^{1-*}(S_{L+k},f;-\alpha)$
for each $\alpha\in[S^{1},M]$.

\subsection{The chain homotopy $P$}

$\ $\vspace{6 pt}

The final ingredient is the chain homotopy $P:CM_{*}(S_{L+k},f)\rightarrow CM^{-*}(S_{L+k},-f)$.
This involves counting a slightly different sort of object. Let $\mathcal{F}_{0}$
denote the set of pairs $(u,T)$ where $T\in\mathbb{R}^{+}$ and $u:[-T,T]\rightarrow T^{*}M\times\mathbb{R}$
satisfies the Rabinowitz Floer equation \eqref{eq:rfeq}. Given $m\geq1$,
let $\widetilde{\mathcal{F}}_{m}$ denote the set of tuples $\boldsymbol{u}=(u_{0},\dots,u_{m})$
such that \[
u_{0}:\mathbb{R}^{+}\rightarrow C^{\infty}(S^{1},T^{*}M)\times\mathbb{R};
\]
\[
u_{2},\dots,u_{m-1}:\mathbb{R}\rightarrow C^{\infty}(S^{1},T^{*}M)\times\mathbb{R};
\]
\[
u_{m}:\mathbb{R}^{-}\rightarrow C^{\infty}(S^{1},T^{*}M)\times\mathbb{R},
\]
all satisfy the Rabinowitz Floer equation \eqref{eq:rfeq}, and such
that \[
u_{i}(-\infty)\in\phi_{\mathbb{R}^{+}}^{-\nabla h}(u_{i-1}(-\infty))\ \ \ \mbox{for }i=1,\dots,m.
\]
Let $\mathcal{F}_{m}$ denote the quotient of $\widetilde{\mathcal{F}}_{m}$
by dividing through by the $\mathbb{R}^{m-1}$ action on the middle
curves $u_{1},\dots,u_{m-1}$. Put\[
\mathcal{F}=\bigcup_{m\in\mathbb{N}\cup\{0\}}\mathcal{F}_{m}.
\]
Given $w_{-},w_{+}\in\overline{\mbox{Crit}}(f)$, denote by $\mathcal{M}_{P}(w_{-},w_{+})$
the subset of $\mathcal{W}^{-}(w_{-})\times\mathcal{F}\times\mathcal{W}^{-}w_{+})$
of elements that {}``begin'' at $w^{-}$ and {}``pass through''
an element of $\mathcal{F}$ and then {}``end'' at $w^{+}$ (we
refer to \cite[p46-47]{AbbondandoloSchwarz2009} for the precise definition).
Then $\mathcal{M}_{P}(w_{-},w_{+})$ turns out to be a finite dimensional
smooth manifold of dimension $\widehat{i}_{f}(w_{-})+\widehat{i}_{f}(w_{+})$.
Here the key issue in the analysis is to check that if $(u,T)\in\mathcal{F}_{0}$
then $T$ is strictly bounded away from zero (\cite[Lemma 8.2]{AbbondandoloSchwarz2009}). 

Now we move onto the key proposition behind the proof of Theorem \ref{thm:theorem A precise}.
The first statement belows shows that if $w_{\pm}\in\overline{\mbox{Crit}}(\mp f)$
satisfy $\widehat{i}_{f}(w_{-})+\widehat{i}_{-f}(w_{+})=1$, we can
define $n_{P}(w_{-},w_{+})$ as the parity of the finite set $\mathcal{M}_{P}(w_{-},w_{+})$.
This defines the map $P$. The fact that $P$ is a chain homotopy
between $\Phi_{\textrm{SA}}$ and $\Phi_{\textrm{AS}}$ involves studying
the compactification of $\mathcal{M}_{P}(w_{-},w_{+})$ by adding
in the broken trajectories, and is the content of the second and third
statements of the proposition below.
\begin{prop}
\label{pro:MP precompactness}(\cite[Proposition 8.1]{AbbondandoloSchwarz2009})

Let $\alpha\in[S^{1},M]$ and choose \[
w_{0}\in\overline{\mbox{\emph{Crit}}}_{0}(f;\alpha),\ \ \ w_{1}\in\overline{\mbox{\emph{Crit}}}_{1}(f;\alpha);
\]
\[
w^{0}\in\overline{\textrm{\emph{Crit}}}_{0}(-f;-\alpha),\ \ \ w^{1}\in\overline{\mbox{\emph{Crit}}}_{1}(-f;-\alpha).
\]
Then:
\begin{enumerate}
\item The moduli space $\mathcal{M}_{P}(w_{0},w^{0})$ is compact.
\item The moduli space $\mathcal{M}_{P}(w_{0},w^{1})$ is precompact, and
we can identify the boundary $\partial\widehat{\mathcal{M}}_{P}(w_{0},w^{1})$
of compactification $\widehat{\mathcal{M}}_{P}(w_{0},w^{1})$ as follows:\begin{eqnarray*}
\partial\widehat{\mathcal{M}}_{P}(w_{0},w^{1}) & = & \left\{ \bigcup_{z\in\textrm{\emph{Crit}}_{0}(h)\cap\textrm{\emph{Crit}}(A_{H-k};\alpha)}\mathcal{M}_{\textrm{\emph{SA}}}(w_{0},z)\times\mathcal{M}_{\textrm{\emph{AS}}}(z,w^{1})\right\} \\
 &  & \bigcup\left\{ \bigcup_{w\in\overline{\textrm{\emph{Crit}}}_{0}(-f)\cap\textrm{\emph{Crit}}(S_{L+k};-\alpha)}\mathcal{M}_{P}(w_{0},w)\times\mathcal{W}(w,w^{1})\right\} .
\end{eqnarray*}

\item The moduli space $\mathcal{M}_{P}(w_{1},w^{0})$ is precompact, and
we can identify the boundary $\partial\widehat{\mathcal{M}}_{P}(w_{1},w^{0})$
of compactification $\widehat{\mathcal{M}}_{P}(w_{1},w^{0})$ as follows:\begin{eqnarray*}
\partial\widehat{\mathcal{M}}_{P}(w_{1},w^{0}) & = & \left\{ \bigcup_{z\in\textrm{\emph{Crit}}_{1}(h)\cap\textrm{\emph{Crit}}(A_{H-k};\alpha)}\mathcal{M}_{\textrm{\emph{SA}}}(w_{1},z)\times\mathcal{M}_{\textrm{\emph{AS}}}(z,w^{0})\right\} \\
 &  & \bigcup\left\{ \bigcup_{w\in\overline{\textrm{\emph{Crit}}}_{1}(-f)\cap\textrm{\emph{Crit}}(S_{L+k};-\alpha)}\mathcal{W}(w_{1},w)\times\mathcal{M}_{P}(w,w^{0})\right\} .
\end{eqnarray*}

\end{enumerate}
\end{prop}
Theorem \ref{thm:theorem A precise} essentially follows from this
proposition; see \cite[Section 9]{AbbondandoloSchwarz2009} for the
details.

\section{\label{sec:Non-displaceability-above-the}Non-displaceability and
leaf-wise intersections above the critical value}

\subsection{Relating $RFH_{*}(A_{H-k})$ with $RFH_{*}(\Sigma_{k},T^{*}M)$}

$\ $\vspace{6 pt}

Rabinowitz Floer homology was defined originally in \cite{CieliebakFrauenfelder2009}
for restricted contact type hypersurfaces and Hamiltonians which are
constant at infinity. This was extended in \cite{CieliebakFrauenfelderPaternain2010}
to cover (amongst other things) the hypersurfaces $\Sigma_{k}$ that
we study here. A natural question therefore becomes whether the Rabinowitz
Floer homology we work with in this paper is isomorphic to that of
\cite{CieliebakFrauenfelderPaternain2010}. The aim of this section
is to prove this in the affirmative. 

Let $(g,\sigma,U,k)\in\mathcal{O}_{\textrm{reg}}$ and put $H=H_{g}+\pi^{*}U$
and $\Sigma_{k}:=H^{-1}(k)$. Let $\rho:\mathbb{R}\rightarrow(-\infty,1]$
denote a smooth function \begin{equation}
\rho(t):=\begin{cases}
t & t\in(-\infty,1-\delta]\\
1 & t\in[1+\delta,\infty)
\end{cases}\ \ \ 0\leq\rho'\leq1,\label{eq:def of rho}
\end{equation}
where $0<\delta<1/3$. Given $R>1$, let $\rho_{R}(t):=R\rho\left(\frac{t}{R}\right)$.
Let $H_{R}:T^{*}M\rightarrow\mathbb{R}$ be defined by \[
H_{R}:=\rho_{R}\circ H.
\]
Assuming $R\gg k$, the Hamiltonian $H_{R}$ satisfies $\Sigma_{k}=H_{R}^{-1}(k)$
and $X_{H_{R}}|_{\Sigma_{k}}=X_{H}|_{\Sigma_{k}}$. However the Hamiltonian
$H_{R}$ is \textbf{constant} at infinity. This makes no difference
to the proof of Theorem \ref{thm:bounding the lagrange multiplier},
or to that of Step 2 in the proof of Theorem \ref{thm:l infinity},
but the proof of Step 1 of Theorem \ref{thm:l infinity} explicitly
required the Hamiltonian to be quadratic. In the course of the proof
below we will show that the proof of Theorem \ref{thm:l infinity}
will still go though for the Hamiltonian $H_{R}$, provided $R\gg0$
is sufficiently large. However exactly what constitutes {}``sufficiently
large'' depends on the action interval $(a,b)\subseteq\mathbb{R}$.
Thus this method is not good enough to define the full Rabinowitz
Floer homology with $H_{R}$. 
\begin{rem}
\label{rem:geometrically bounded}In \cite{CieliebakFrauenfelderPaternain2010}
this is overcome by using an entirely different method to obtain $L^{\infty}$
bounds on the $x$-component of gradient flow lines. Namely, they
work with a compatible almost complex structure $J$ that is \textbf{\emph{geometrically
bounded}} outside of a compact set. We refer to \cite[Chapter V]{AudinLafontaine1994}
for the precise definition, and also for an explanation as to why
working with Hamiltonians which are constant outside of a compact
set and almost complex structures that are geometrically bounded outside
of a compact set leads to such $L^{\infty}$ bounds. Proofs that twisted
cotangent bundles are geometrically bounded can be found in \cite[Proposition 2.2]{CieliebakGinzburgKerman2004}
or \cite[Proposition 4.1]{Lu1996}. The latter proof also shows that
it is possible (if $\left\Vert \sigma\right\Vert _{\infty}$ is small
enough) to find geometrically bounded almost complex structures $J\in\mathcal{J}(\omega)\cap B_{\varepsilon_{1}}(J_{g})$.
\end{rem}
We now prove the following result.
\begin{prop}
\label{quadr is const}Given a fixed finite interval $(a,b)\subseteq\mathbb{R}$,
there exists a constant $R(a,b)>0$ such that for all $R>R(a,b)$
there is a chain complex isomorphism \[
RFH_{*}^{(a,b)}(A_{H_{R}-k})\cong RFH_{*}^{(a,b)}(A_{H-k}).
\]
\end{prop}
\begin{proof}
Assuming $R$ is sufficiently large compared to $k$, since all the
critical points of $A_{H-k}$ are either points on $\Sigma_{k}$ or
parametrizations of periodic orbits of $X_{H}$ lying on $\Sigma_{k}$,
we conclude that all the critical points of $A_{H_{R}-k}$ are non-degenerate,
and that $\mbox{Crit}(A_{H_{R}-k})=\mbox{Crit}(A_{H-k})$. This shows
that the two chain complexes coincide (as \textbf{groups}): \[
RF_{*}(A_{H_{R}-k})\cong RF_{*}(A_{H-k}).
\]
Fix an almost complex structure $J\in\mathcal{J}(\omega)\cap B_{\varepsilon_{1}}(J_{g})$
and fix a finite interval $(a,b)\subseteq\mathbb{R}$. Let us denote
by $\mathcal{M}_{R}$ the set of all maps $u=(x,\eta)\in C^{\infty}(\mathbb{R}\times S^{1},T^{*}M)\times C^{\infty}(\mathbb{R},\mathbb{R})$
that satisfy the Rabinowitz Floer equation $u'(s)+\nabla A_{H_{R}-k}(u(s,\cdot))=0$
and have action bounds \[
A_{H_{R}-k}(u(\mathbb{R}))\subseteq[a,b]
\]
and satisfy\[
x(\mathbb{R},\cdot)\in\Lambda_{\alpha}T^{*}M.
\]
We will show how for $R\gg0$ large enough, one can follow through
the proof of Theorem \ref{thm:l infinity} and obtain a constant $C_{1}'>0$
that serves as a uniform $L^{\infty}$ bound for the $x$-component
of elements of $\mathcal{M}_{R}$. Here the key point is that the
constant $C_{1}'$ is \textbf{independent }of $R$. It will however
depend on the interval $(a,b)$. Anyway, this will imply the proposition,
as then it is immediate that if we choose $R$ large enough such that
\[
\left\{ (q,p)\in T^{*}M\,:\,\left|p\right|\leq C_{1}'\right\} \subseteq\left\{ (q,p)\in T^{*}M\,:\, H(q,p)\leq R\right\} 
\]
then the boundary homomorphisms of the two truncated Rabinowitz Floer
complexes necessarily coincide, and hence the two truncated Rabinowitz
Floer homologies coincide.

Firstly though let us discuss the $\eta$-component of elements of
$\mathcal{M}_{R}$. Nothing in the proof of the bound on the Lagrange
multiplier (Theorem \ref{thm:bounding the lagrange multiplier}) used
anything about the behavior of $H$ at infinity, and thus there exists
$R_{0}>0$ such that if $R>R_{0}$, the same constant $C_{0}>0$ obtained
in Theorem \ref{thm:bounding the lagrange multiplier} serves as uniform
$L^{\infty}$ bound on the $\eta$-component of any element $u=(x,\eta)\in\mathcal{M}_{R}$. 

Parts of the argument from the proof of Step 1 of Theorem \ref{thm:l infinity}
are unchanged for the new Hamiltonian $H_{R}$. Indeed, for any $u=(x,\eta)\in\mathcal{M}_{R}$
if $b_{1}:=\left\Vert J\right\Vert _{\infty}\sqrt{b-a}$ then as before\begin{equation}
\left\Vert x'\right\Vert _{L^{2}(\mathbb{R}\times S^{1})}\leq b_{1};\label{eq:b1 second time}
\end{equation}
\[
\left\Vert \eta'\right\Vert _{L^{2}(\mathbb{R})}\leq b_{1}.
\]
Fix $u=(x,\eta)\in\mathcal{M}_{R}$ and write $x=(q,p)$. Let us introduce
the auxiliary smooth function\[
P_{R}:\mathbb{R}\times S^{1}\rightarrow\mathbb{R};
\]
\[
P_{R}(s,t):=\rho_{2R}\left(\left|p(s,t)\right|^{2}\right)=2\rho_{R}\left(\frac{1}{2}\left|p(s,t)\right|^{2}\right).
\]
Note that \[
X_{H_{R}}(q,p)=\rho_{R}'(H(q,p))\, X_{H}(q,p).
\]
Since $\rho_{R}'\leq1$ we see that \eqref{eq:b0} still holds. Let
us note that for any $(q,p)\in T^{*}M$ the following two implications
hold: \[
\left|p\right|^{2}\geq2R+2R\delta\ \ \ \Rightarrow\ \ \ \rho_{R}'\left(\frac{1}{2}\left|p\right|^{2}\right)\left|p\right|^{2}=0;
\]
\[
\left|p\right|^{2}\leq4R-4R\delta\ \ \ \Rightarrow\ \ \ \rho_{4R}\left(\left|p\right|^{2}\right)=\left|p\right|^{2}
\]
(where $\delta>0$ is the constant from the definition \eqref{eq:def of rho}
of $\rho$). Thus since $\delta<1/3$ we always have\begin{equation}
\rho_{R}'\left(\frac{1}{2}\left|p\right|^{2}\right)\left|p\right|^{2}\leq\rho_{4R}\left(\frac{1}{2}\left|p\right|^{2}\right).\label{eq:first delta ineq}
\end{equation}
In fact, we can improve on this by choosing $R$ sufficiently large.
Indeed, suppose \[
R>R_{1}:=\frac{\left\Vert U\right\Vert _{\infty}}{1-3\delta}.
\]
Then for any $(q,p)\in T^{*}M$ one has \[
H(q,p)=R+R\delta\ \ \ \Rightarrow\ \ \ \left|p\right|^{2}\leq4R-4R\delta,
\]
and hence for $R>R_{1}$ we have\[
\rho_{R}'(H(q,p))\left|p\right|^{2}\leq\rho_{4R}\left(\left|p\right|^{2}\right)\ \ \ \mbox{for every }(q,p)\in T^{*}M.
\]
Thus for $R>R_{1}$, \begin{equation}
\left|X_{H_{R}}(x(s,t))\right|\leq b_{0}(1+P_{2R}(s,t)))\ \ \ \mbox{for all }(s,t)\in\mathbb{R}\times S^{1},\label{eq:xh bound-1}
\end{equation}
where $b_{0}>0$ is defined as before%
\footnote{In order to aid the reader, throughout this proof the constants $b_{i}$
that appear are the \textbf{same }as the constants $b_{i}$ from the
proof of Theorem \ref{thm:l infinity}. In some cases it is not possible
to use exactly the same constant; in this case we denote it by $b_{i}'$. %
}. 

In the truncated case, $\eta'(s)$ no longer bounds the $L^{2}$ norm
of $p(s,\cdot)$, as in \eqref{eq:eta bounding p}, but instead it
bounds the $L^{1}$ norm of $P_{R}(s,\cdot)$. Indeed, since\[
H_{R}(q,p)\geq\rho_{R}\left(\frac{1}{2}\left|p\right|^{2}\right)-\left\Vert U\right\Vert _{\infty}
\]
for every $(q,p)\in T^{*}M$, we have

\[
\eta'(s)\geq\int_{S^{1}}\frac{1}{2}P_{R}(s,t)dt-(\left\Vert U\right\Vert _{\infty}+k).
\]
The same arguments as before successively prove:
\begin{itemize}
\item $\left\Vert P_{R}\right\Vert _{L^{1}(I\times S^{1})}\leq b_{2}\max\left\{ \left|I\right|,\left|I\right|^{1/2}\right\} $
for any finite interval $I\subseteq\mathbb{R}$.
\item For any $0<\varepsilon\leq1$ the sets \[
\left\{ s\in\mathbb{R}\,:\,\left\Vert P_{R}(s,\cdot)\right\Vert _{L^{1}(S^{1})}\leq\frac{b_{2}}{\sqrt{\varepsilon}}\right\} ;
\]
\[
\left\{ s\in\mathbb{R}\,:\,\left\Vert x'(s,\cdot)\right\Vert _{L^{2}(S^{1})}^{2}\leq\frac{b_{1}}{\sqrt{\varepsilon}}\right\} 
\]
are $\varepsilon$-dense in $\mathbb{R}$.
\item For any $s\in\mathbb{R}$, $\left\Vert P_{R}(s,\cdot)\right\Vert _{L^{1}(S^{1})}\leq b_{3}$.
We will go through this one in detail: given $s\in\mathbb{R}$, choose
$s_{0}\in\mathbb{R}$ such that $\left|s-s_{0}\right|\leq1$ and $\left\Vert P_{R}(s_{0},\cdot)\right\Vert _{L^{1}(S^{1})}\leq b_{2}$.
Without loss of generality assume $s\geq s_{0}$. Then we have\begin{align*}
\left\Vert P_{R}(s,\cdot)\right\Vert _{L^{1}(S^{1})} & =\left\Vert P_{R}(s_{0},\cdot)\right\Vert _{L^{1}(S^{1})}+\int_{s_{0}}^{s}\frac{d}{dr}\left\Vert P_{R}(r,\cdot)\right\Vert _{L^{1}(S^{1})}dr\\
 & =b_{2}+2\int_{s_{0}}^{s}\int_{S^{1}}\rho_{R}'\left(\frac{1}{2}\left|p\right|^{2}\right)\left\langle p(r,t),p'(r,t)\right\rangle dtdr\\
 & \leq b_{2}+2\left|\int_{s_{0}}^{s}\int_{S^{1}}\rho_{R}'\left(\frac{1}{2}\left|p(r,t)\right|^{2}\right)\left|p(r,t)\right|^{2}dtdr\right|^{1/2}\left|\int_{s_{0}}^{s}\int_{S^{1}}\left|p'(r,t)\right|^{2}dtdr\right|^{1/2}\\
 & \overset{(*)}{\leq}b_{2}+2\left|\int_{s_{0}}^{s}\left\Vert P_{2R}(r,\cdot)\right\Vert _{L^{1}(S^{1})}dr\right|^{1/2}\left\Vert p'\right\Vert _{L^{2}((s_{0},s)\times S^{1})}\\
 & \leq b_{2}+2\sqrt{b_{2}}\left\Vert x'\right\Vert _{L^{2}((s_{0},s)\times S^{1})}\\
 & \leq b_{2}+2\sqrt{b_{2}}b_{1}.
\end{align*}
Here $(*)$ used \eqref{eq:first delta ineq} (note that the assertion
from the first bullet point also holds for $P_{2R}$!).
\end{itemize}
Note however that this last assertion does \textbf{not }imply that
$\left\Vert p(s,\cdot)\right\Vert _{L^{2}(S^{1})}\leq b_{3}$ for
all $s\in\mathbb{R}$, even if $R\gg b_{3}$. In order to prove this
we we argue as follows. Let $\vartheta_{R}:\mathbb{R}\rightarrow[0,1]$
denote a smooth function such that agrees with $\rho_{R}$ for $t\geq\delta$,
and is equal to $\delta/2$ for $t\leq0$, again with $0\leq\vartheta_{R}'\leq1$.
We now introduce another auxiliary smooth function\[
f_{R}:\mathbb{R}\times S^{1}\rightarrow\mathbb{R};
\]
\[
f_{R}(s,t):=\vartheta_{R}(\left|p(s,t)\right|).
\]
Observe that for any $(s,t)\in\mathbb{R}\times S^{1}$, if $R>R_{2}:=\max\{R_{0},R_{1}\}$
then \begin{align*}
\dot{f}_{R}(s,t) & \leq\left|\dot{p}(s,t)\right|\\
 & \leq\left|\dot{x}(s,t)\right|\\
 & \leq\left|J(t,x(s,t))\cdot x'(s,t)\right|+\left|\eta(s)\right|\left|X_{H_{R}}(x(s,t))\right|\\
 & \leq\left|J(t,x(s,t))\cdot x'(s,t)\right|+C_{0}b_{0}(1+P_{2R}(s,t)).
\end{align*}
Thus for $R>R_{2}$, \[
\left\Vert \dot{f}_{R}(s,\cdot)\right\Vert _{L^{1}(S^{1})}\leq\left\Vert J\right\Vert _{\infty}\left\Vert x'(s,\cdot)\right\Vert _{L^{2}(S^{1})}+C_{0}b_{0}\left(1+\left\Vert P_{2R}(s,\cdot)\right\Vert _{L^{1}(S^{1})}\right).
\]
Now observe that \[
\left\Vert f_{R}(s,\cdot)\right\Vert _{L^{1}(S^{1})}\leq1+\left\Vert P_{R}(s,\cdot)\right\Vert _{L^{1}(S^{1})}
\]
(since at any point $t\in S^{1}$, either $f_{R}(s,t)\leq1$ or $f_{R}(s,t)\leq P_{R}(s,t)$).
It now follows from \eqref{eq:s dash} and \eqref{eq:linfty p} that
if\[
b_{4}':=N(1+b_{3})+NC_{0}b_{0}(1+b_{3}^{2}),
\]
where $N>0$ is defined as in \eqref{eq:N} then for any $0<\varepsilon\leq1$
the subset \[
\left\{ s\in\mathbb{R}\,:\,\left\Vert f_{R}(s,\cdot)\right\Vert _{L^{\infty}(S^{1})}\leq b'_{4}+\frac{b_{5}}{\sqrt{\varepsilon}}\right\} 
\]
is $\varepsilon$-dense. Let \[
R_{3}:=\max\left\{ R_{2},\frac{b'_{4}+b_{5}}{1-\delta}\right\} .
\]
Then for $R>R_{3}$, we know that the set $\{s\in\mathbb{R}\,:\,\left\Vert p(s,\cdot)\right\Vert _{L^{\infty}(S^{1})}\leq b_{4}'+b_{5}\}$
is $1$-dense in $\mathbb{R}$, and then arguing as before, we discover
that there exists a constant $b_{3}'>0$ such that \begin{equation}
\left\Vert p(s,\cdot)\right\Vert _{L^{2}(S^{1})}\leq b_{3}'\ \ \ \mbox{for all }s\in\mathbb{R}.\label{eq:b3 dash}
\end{equation}
Next, since $\rho_{R}'\leq1$, the same argument as before shows that
for all $s_{0}<s_{1}$ we have \begin{equation}
\left\Vert \nabla p\right\Vert _{L^{2}((s_{0},s_{1})\times S^{1})}^{2}\leq b_{6}\left(\left|s_{1}-s_{0}\right|+b_{1}^{2}\right)+b_{6}\left\Vert p\right\Vert _{L^{4}((s_{0},s_{1})\times S^{1})}^{4}.\label{eq:b6 second time}
\end{equation}
Now let us choose $\varepsilon=\varepsilon_{*}$ where $\varepsilon_{*}=\varepsilon_{*}(b_{3}',b_{6})$
is the constant from Lemma \ref{lem:the gamma lemma}, and choose
\[
R>R(a,b):=\max\left\{ R_{3},b'_{4}+\frac{b_{5}}{\sqrt{\varepsilon_{*}}}\right\} .
\]
Then for $R>R(a,b)$ the set \begin{equation}
\left\{ s\in\mathbb{R}\,:\,\left\Vert p(s,\cdot)\right\Vert _{L^{\infty}(S^{1})}\leq b'_{4}+\frac{b_{5}}{\sqrt{\varepsilon_{*}}}\right\} \label{eq:final eq for gamma lemma}
\end{equation}
is $\varepsilon_{*}$-dense in $\mathbb{R}$. Thus by \eqref{eq:b3 dash},
\eqref{eq:b6 second time} and \eqref{eq:final eq for gamma lemma},
Lemma \ref{lem:the gamma lemma} implies that there exists a constant
$b_{7}'=\Gamma(b_{1},b_{3}',b_{6},b_{4}'+b_{5}/\sqrt{\varepsilon_{*}})$
such that for $R>R(a,b)$ the following holds: for any finite interval
$I\subseteq\mathbb{R}$ we have \[
\left\Vert \nabla p\right\Vert _{L^{2}(I\times S^{1})}\leq b_{7}'(1+\left|I\right|^{1/2}).
\]

In other words, for $R>R(a,b)$, Step 1 of Theorem \ref{thm:l infinity}
goes through, and the constant $K'>0$ that we obtain is independent
of $R$. Moving onto Step 2, we note that the proof of Step 2 used
nothing about the Hamiltonian other than the fact that Step 1 holds,
and that \eqref{eq:b0} holds. Thus the proof goes through immediately
for the Hamiltonian $H_{R}$ with $R>R(a,b)$. Moreover the constant
$C'_{1}>0$ that Step 2 produced depended only on the constants $K'$
and $b_{0}$. Thus we have proved that there exists a constant $C_{1}'>0$
such that if $R>R(a,b)$ and $u=(x,\eta)\in\mathcal{M}_{R}$ then
\[
\left\Vert x\right\Vert _{L^{\infty}(\mathbb{R}\times S^{1})}\leq C_{1}'.
\]
By the remarks at the beginning of the proof this implies the result.
\end{proof}
Let us denote by $RFH_{*}(\Sigma_{k},T^{*}M)$ the Rabinowitz Floer
homology of the hypersurface $\Sigma_{k}$ as defined%
\footnote{Technically the Rabinowitz Floer homology $RFH_{*}(\Sigma_{k},T^{*}M)$
as defined in \cite{CieliebakFrauenfelderPaternain2010} is only defined
for contractible loops. If however one uses the observation that $\omega$
is symplectically atoroidal then the construction in \cite{CieliebakFrauenfelderPaternain2010}
allows one to define Rabinowitz Floer homology $RFH_{*}(\Sigma_{k},T^{*}M)$
for any free homotopy class; see Remark \ref{rem:remark on cfp rfh}.%
} in \cite{CieliebakFrauenfelderPaternain2010}. We now prove that
$RFH_{*}(\Sigma_{k},T^{*}M)\cong RFH_{*}(A_{H-k})$. It is sufficient
to prove this when $(g,\sigma,U,k)\in\mathcal{O}_{\textrm{reg}}$
(cf. Remark \ref{rem:inv under g}). Since the Hamiltonian $H_{R}$
is constant outside of a compact set, using the invariance result
\cite[Theorem 1.1]{CieliebakFrauenfelderPaternain2010} we conclude
that we can compute $RFH_{*}(\Sigma_{k},T^{*}M)$ using%
\footnote{Here we are implicitly using the last sentence of Remark \ref{rem:geometrically bounded}.%
} $H_{R}$, and thus for $R>R(a,b)$: \[
RFH_{*}^{(a,b)}(A_{H_{R}-k})\cong RFH_{*}^{(a,b)}(\Sigma_{k},T^{*}M).
\]
Then using \cite[Theorem A]{CieliebakFrauenfelder2009a}, which tells%
\footnote{This is the only time in the entire paper where it is absolutely \textbf{essential}
that we used field coefficients for the Rabinowitz Floer homology
rather than, say, $\mathbb{Z}$-coefficients.%
} us that we can determine both the Rabinowitz Floer homologies $RFH_{*}(A_{H-k})$
and $RFH_{*}(\Sigma_{k},T^{*}M)$ from the truncated Rabinowitz Floer
homologies via:\[
RFH_{*}(A_{H-k})\cong\underset{a\downarrow-\infty}{\underrightarrow{\lim}}\underset{b\uparrow\infty}{\underleftarrow{\lim}}RFH_{*}^{(a,b)}(A_{H-k});
\]
\[
RFH_{*}(\Sigma_{k},T^{*}M)\cong\underset{a\downarrow-\infty}{\underrightarrow{\lim}}\underset{b\uparrow\infty}{\underleftarrow{\lim}}RFH_{*}^{(a,b)}(\Sigma_{k},T^{*}M).
\]
We conclude that \begin{equation}
RFH_{*}(A_{H-k})\cong RFH_{*}(\Sigma_{k},T^{*}M).\label{eq:main conc}
\end{equation}

We can now prove the main result of this paper. In the proof below
for clarity we will continue to write $RFH_{*}(A_{H-k})$ for the
Rabinowitz Floer homology as defined in this paper, and $RFH_{*}(\Sigma_{k},T^{*}M)$
for the Rabinowitz Floer homology defined in \cite{CieliebakFrauenfelderPaternain2010},
despite the fact that we have just proved the two are isomorphic.
\begin{proof}
\emph{(of Theorem }\ref{thm:my main theorem}\emph{)}

We are given a closed weakly exact 2-form $\sigma\in\Omega_{\textrm{we}}^{2}(M)$
and a potential $U\in C^{\infty}(M,\mathbb{R})$, together with a
value $k\in\mathbb{R}$ such that $k>c(g,\sigma,U)$. Put $H:=H_{g}+\pi^{*}U$
and $\Sigma_{k}:=H^{-1}(k)$. We will compute $RFH_{*}(\Sigma_{k},T^{*}M)$.
By Remark \ref{rem:inv under g} we may assume that $(g,\sigma,U,k)\in\mathcal{O}_{\textrm{reg}}$.
We begin by choosing $r>0$ such that $\left\Vert r\sigma\right\Vert _{\infty}$
is sufficiently small such that the conclusion of Theorem \ref{thm:theorem A precise}
holds. Let us temporarily write $RFH_{*}(\Sigma_{k},T^{*}M;\omega)$
to indicate which symplectic form we are working with. Then \[
RFH_{*}(\Sigma_{k},T^{*}M;\omega)\cong RFH_{*}(\Sigma_{k},T^{*}M;r\omega).
\]
To see we argue as follows. If $F\in C_{c}^{\infty}(T^{*}M,\mathbb{R})$
is a \textbf{defining Hamiltonian }for $(H,\Sigma_{k},\omega)$ in
the sense of \cite{CieliebakFrauenfelderPaternain2010}, that is,
$F$ is a compactly supported Hamiltonian such that $\Sigma_{k}=F^{-1}(0)$,
and $X_{F}|_{\Sigma_{k}}=X_{H}|_{\Sigma_{k}}$, then since$X_{F}^{\omega}=X_{rF}^{r\omega}$
(here $X_{F}^{\omega}$ denotes the symplectic gradient of $F$ with
respect to $\omega$, etc.), the Hamiltonian $rF$ is a defining Hamiltonian
for $(H,\Sigma_{k},r\omega)$. Next, there is a natural identification
between flow lines of the two Rabinowitz action functionals $A_{F}$
and $A_{rF}$: if $u(s,t)=(x(s,t),\eta(s))$ satisfies $u'(s)+\nabla A_{F}(u(s))=0$
then $u_{r}(s,t):=(x(s,t),\eta(rs))$ satisfies $u_{r}'(s)+\nabla A_{rF}(u_{r}(s))=0$,
and vice versa. This identification defines a chain isomorphism between
the two chain complexes.

Set $\omega_{r}:=\omega_{0}+r\pi^{*}\sigma$ so that $\omega=\omega_{1}$.
Next we claim \[
RFH_{*}(\Sigma_{k},T^{*}M;r\omega)\cong RFH_{*}(H_{r}^{-1}(r^{2}k),T^{*}M;\omega_{r}),
\]
where $H_{r}(q,p):=H_{g}+r^{2}\pi^{*}U$ (note that the latter is
well defined, as by Lemma \ref{lem:scaling c} we have $k>c(g,\sigma,U)$
if and only if $r^{2}k>c(g,r\sigma,r^{2}U)$). Indeed, the exact symplectomorphism
$\varphi_{r}:T^{*}M\rightarrow T^{*}M$ defined by \[
\varphi_{r}(q,p):=(q,rp)
\]
satisfies \[
\varphi_{r}^{*}\omega_{r}=r\omega;
\]
\[
\varphi_{r}^{*}H_{r}=r^{2}H,
\]
and hence $\varphi_{r}(\Sigma_{k})=H_{r}^{-1}(r^{2}k)$. The Rabinowitz
Floer homology of \cite{CieliebakFrauenfelderPaternain2010} is invariant
under such symplectomorphisms, and hence the claim follows. Next,
by \eqref{eq:main conc} we have

\[
RFH_{*}(H_{r}^{-1}(r^{2}k),T^{*}M;\omega_{r})\cong RFH_{*}(A_{H_{r}-r^{2}k};\omega_{r}),
\]
and finally by our choice of $r$ we can compute $RFH_{*}(A_{H_{r}-r^{2}k};\omega_{r})$
via Theorem \ref{thm:theorem A precise}.
\end{proof}

\subsection{Leaf-wise intersections}

$\ $\vspace{6 pt}

We conclude this paper by showing how the fact that $RFH_{*}(\Sigma_{k},T^{*}M)$
is non-zero for $k>c(g,\sigma,U)$ implies the existence of \textbf{leaf-wise
intersections}, following \cite{AlbersFrauenfelder2010c,AlbersFrauenfelder2008}.
Throughout this section assume that $(g,\sigma,U,k)\in\mathcal{O}$
(in general we do \textbf{not}\emph{ }need to assume that $(g,\sigma,U,k)\in\mathcal{O}_{\textrm{reg}}$,
although this will be needed to get infinitely many leaf-wise intersections),
and put $H:=H_{g}+\pi^{*}U$ and $\Sigma_{k}:=H^{-1}(k)$.

The hypersurface $\Sigma_{k}$ is foliated by the leaves $\{\mathcal{L}_{x}\,:\, x\in\Sigma_{k}\}$,
where \[
\mathcal{L}_{x}:=\{\phi_{t}^{H}(x)\,:\, t\in\mathbb{R}\}.
\]
Let $\mbox{Ham}_{c}(T^{*}M,\omega)$ denote the set of compactly supported
1-periodic Hamiltonian diffeomorphisms of the symplectic manifold
$(T^{*}M,\omega)$, that is \[
\mbox{Ham}_{c}(T^{*}M,\omega):=\left\{ \phi_{1}^{F}\,:\, F\in C_{c}^{\infty}(S^{1}\times T^{*}M,\mathbb{R})\right\} ,
\]
where $\phi_{t}^{F}$ is the flow of $X_{F}$; the latter being the
time-dependent symplectic gradient of $F$ with respect to $\omega$.
Given $\psi\in\mbox{Ham}_{c}(T^{*}M,\omega)$, a point $x\in\Sigma_{k}$
is called a \textbf{leaf-wise intersection point for $\psi$}\emph{
}if $\psi(x)\in\mathcal{L}_{x}$.\newline

In order to explain the beautiful idea of Albers and Frauenfelder
that links Rabinowitz Floer homology to leaf-wise intersections, we
will need some preliminary definitions. First let us define \[
\mathcal{X}:=\left\{ \chi\in C^{\infty}(S^{1},\mathbb{R})\,:\,\int_{S^{1}}\chi(t)dt=1,\ \mbox{supp}(\chi)\subseteq(0,1/2)\right\} .
\]

We will say that a time-dependent Hamiltonian $G:S^{1}\times T^{*}M\rightarrow\mathbb{R}$
is \textbf{$H$-admissible}\emph{ }if:
\begin{enumerate}
\item $G(t,x)=\chi(t)G_{0}(x)$ for some $\chi\in\mathcal{X}$ and some
\textbf{compactly supported} $G_{0}\in C_{c}^{\infty}(T^{*}M,\mathbb{R})$.
\item $G_{0}^{-1}(0)=\Sigma_{k}$.
\item It holds that $X_{G_{0}}|_{\Sigma_{k}}=X_{H}|_{\Sigma_{k}}$. 
\end{enumerate}
Let us write $\mathcal{H}(H)$ for the set of $H$-admissible Hamiltonians.
Finally set \[
\mathcal{F}:=\left\{ F\in C_{c}^{\infty}(S^{1}\times T^{*}M,\mathbb{R})\,:\, F(t,\cdot)\equiv0\ \mbox{for }t\in[1/2,1]\right\} .
\]
It is easy to see that $\mathcal{F}$ generates $\mbox{Ham}_{c}(T^{*}M,\omega)$
in the sense that given any $\psi\in\mbox{Ham}_{c}(T^{*}M,\omega)$,
there exists $F\in\mathcal{F}$ such that $\psi=\phi_{1}^{F}$ (see
for example \cite[Lemma 2.3]{AlbersFrauenfelder2010c}). 

Let us call a pair $(G,F)\in\mathcal{H}(H)\times\mathcal{F}$ a \textbf{Moser
pair} for $\Sigma_{k}$. Given a Moser pair $(G,F)$ for $\Sigma_{k}$,
define the \textbf{perturbed Rabinowitz action functional}\emph{ }$A_{G-k}^{F}:\Lambda T^{*}M\times\mathbb{R}\rightarrow\mathbb{R}$
by\[
A_{G-k}^{F}(x,\eta):=\int_{C}\bar{x}^{*}\omega-\eta\int_{S^{1}}G(t,x)dt-\int_{S^{1}}F(t,x)dt
\]
(where $\bar{x}$ and $C$ are defined as before). A short calculation
shows that \[
\mbox{Crit}(A_{G-k}^{F})=\left\{ (x,\eta)\in C^{\infty}(S^{1},T^{*}M)\times\mathbb{R}\,:\,\dot{x}=\eta\chi(t)X_{G_{0}}(x)+X_{F}(t,x),\ \int_{S^{1}}\chi(t)G_{0}(x)dt=0\right\} .
\]

The key observation of Albers and Frauenfelder that makes the whole
approach work is the following lemma \cite[Proposition 2.4]{AlbersFrauenfelder2010c}.
\begin{lem}
\label{lem:albersfrau}Suppose $(x,\eta)\in\mbox{\emph{Crit}}(A_{G-k}^{F})$.
Then if $\psi=\phi_{1}^{F}$ and $y:=x(1/2)\in\Sigma$ then $\psi(y)=\mathcal{L}_{y}$,
that is, $y$ is a leaf-wise intersection point for $\psi$ in $\Sigma_{k}$.\end{lem}
\begin{proof}
For $t\in[0,1/2]$ we have $G_{0}(x(t))$ constant, since $X_{F}(t,\cdot)=0$,
and hence $x(t)\in\Sigma_{k}$ for $t\in[0,1/2]$. For $t\in[1/2,1]$,
$x(t)$ satisfies $\dot{x}(t)=X_{F}(t,x(t))$ and hence $x(1)=\psi(x(1/2))$.
Thus if $y:=x(1/2)$ then $y$ and $\psi(y)$ both lie in $\Sigma_{k}$.
Moreover since on $[0,1/2]$ we have $\dot{x}(t)=\eta\chi(t)G_{0}(x(t))$
we have $\psi(y)=x(0)\in\mathcal{L}_{y}$. The proof is complete. 
\end{proof}
Let us say that a leaf-wise intersection point $y\in\Sigma_{k}$ for
$\psi\in\mbox{Ham}_{c}(T^{*}M,\omega)$ is a \textbf{periodic leaf-wise
intersection point for $\psi$}\emph{ }if the leaf $\mathcal{L}_{x}$
is a closed orbit of $\phi_{t}^{H}$. It is clear from the proof above
that the map $\mbox{Crit}(A_{G-k}^{F})\rightarrow\{\mbox{leaf-wise intersection points for }\phi_{1}^{F}\}$
is injective if there do not exist any periodic leaf-wise intersection
points for $\psi$.\newline

We will now state the two analytic results about the perturbed twisted
Rabinowitz action functional $A_{G-k}^{F}$ that allow one to do Rabinowitz
Floer homology with it. The proof of the first theorem is essentially
identical to \cite[Theorem 2.14]{AlbersFrauenfelder2010c} and \cite[Theorem 3.3]{AlbersFrauenfelder2008}.
\begin{thm}
Fix $G\in\mathcal{H}(H)$. Let $\mathcal{F}_{\textrm{\emph{reg}}}(G)\subseteq\mathcal{F}$
denote the set of functions $F$ such that $A_{G-k}^{F}$ is a Morse
function. Then $\mathcal{F}_{\textrm{\emph{reg}}}(G)$ is residual
in $\mathcal{F}$. Moreover if $(g,\sigma,U,k)\in\mathcal{O}_{\textrm{\emph{reg}}}$
then the set $\widetilde{\mathcal{F}}_{\textrm{\emph{reg}}}(G)\subseteq\mathcal{F}_{\textrm{\emph{reg}}}(G)$
consisting of those functions $F\in\mathcal{F}_{\textrm{\emph{reg}}}(G)$
for which there do not exist any periodic leaf-wise intersection points
for $\phi_{1}^{F}$ in $\Sigma_{k}$, is also residual in $\mathcal{F}$.
\end{thm}
The following result is proved exactly as in \cite[Theorem 2.9]{AlbersFrauenfelder2010c},
aside from the fact that one needs to use the modifications already
present in the proof of Theorem \ref{thm:bounding the lagrange multiplier}
above to deal with the fact that $\Sigma_{k}$ is only of virtual
restricted contact type.
\begin{thm}
Let $-\infty<a<b<\infty$ and $\alpha\in[S^{1},M]$, and let $\mathcal{M}$
denote the set of gradient flow lines $u\in C^{\infty}(\mathbb{R}\times S^{1},T^{*}M)\times C^{\infty}(\mathbb{R},\mathbb{R})$
of $A_{G-k}^{F}$ (with respect to a suitable compatible almost complex
structure) such that $A_{G-k}^{F}(u(\mathbb{R}))\subseteq[a,b]$ and
$x(\mathbb{R},\cdot)\subseteq\Lambda_{\alpha}T^{*}M$. Then $\mathcal{M}$
is precompact in $C^{\infty}(\mathbb{R}\times S^{1},T^{*}M)\times C^{\infty}(\mathbb{R}\times S^{1},\mathbb{R})$,
where this space is given the $C_{\textrm{\emph{loc}}}^{\infty}$
topology.
\end{thm}
Using the previous two theorems (see \cite[Section 2]{AlbersFrauenfelder2010c}
for the full details), if $F\in\mathcal{F}_{\textrm{reg}}(G)$ one
can define the Rabinowitz Floer homology $RFH_{*}(A_{G-k}^{F})$ of
the perturbed Rabinowitz action functional $A_{G-k}^{F}$, and show
moreover that \[
RFH_{*}(A_{G-k}^{F})\cong RFH_{*}(A_{G-k}^{F=0})\overset{\textrm{def}}{=}RFH_{*}(\Sigma_{k},T^{*}M).
\]

In particular, given $F\in\mathcal{F}_{\textrm{reg}}(G)$ we have
the following corollary of Theorem \ref{thm:theorem A precise}.
\begin{cor}
For degrees $*\ne0,1$, \[
RFH_{*}(A_{G-k}^{F})\cong\begin{cases}
H_{*}(\Lambda M;\mathbb{Z}_{2})\\
H^{1-*}(\Lambda M;\mathbb{Z}_{2}).
\end{cases}
\]

\end{cor}
Using the corollary it is easy to complete the proof of Theorem \ref{thm:Leafwise}
from the introduction.
\begin{proof}
\emph{(of Theorem \ref{thm:Leafwise})}

First we show that any $\psi\in\mbox{Ham}_{c}(T^{*}M,\omega)$ has
a leaf-wise intersection point. Indeed, if not then we can find $F\in\mathcal{F}$
and $G\in\mathcal{H}(H)$ such that $\psi=\phi_{1}^{F}$ and $\mbox{Crit}(A_{G-k}^{F})=\emptyset$
(see for instance \cite[p279-280]{CieliebakFrauenfelder2009}). In
this case $A_{G-k}^{F}$ is trivially Morse, and hence $F\in\mathcal{F}_{\textrm{reg}}(G)$.
But if $\mbox{Crit}(A_{G-k}^{F})=\emptyset$ then $RFH_{*}(A_{G-k}^{F})=0$,
a contradiction. 

Suppose now that $\dim\, H_{*}(\Lambda M;\mathbb{Z}_{2})=\infty$
and $(g,\sigma,U,k)\in\mathcal{O}_{\textrm{reg}}$. Then for a generic
$\psi\in\mbox{Ham}_{c}(T^{*}M,\omega)$, we can write $\psi=\phi_{1}^{F}$
for some $F\in\widetilde{\mathcal{F}}_{\textrm{reg}}(G)$. In this
case the previous corollary combined with Lemma \ref{lem:albersfrau}
implies the existence of infinitely many leaf-wise intersection points
for $\psi$ in $\Sigma_{k}$.
\end{proof}
\bibliographystyle{amsplain}
\bibliography{willbibtex}

\end{document}